\documentclass[11 pt]{amsart}
\usepackage{ graphicx, mathrsfs}
\usepackage{amssymb,amsmath,amsthm}
\usepackage{mathrsfs,dsfont,latexsym}

\theoremstyle{plain}
\newtheorem{theorem}{Theorem}[section]
\newtheorem*{nth}{Theorem}
\newtheorem{proposition}{Proposition}[section]
\newtheorem{claim}[proposition]{Claim}
\newtheorem{lemma}[proposition]{Lemma}
\newtheorem{defi}[proposition]{Definition}
\newtheorem{corollary}[proposition]{Corollary}

\newtheorem*{remark}{Remark}
\numberwithin{equation}{section}
\newcommand{\RR}{\mathbb{R}}

\newcommand{\p}{\partial}
\newcommand{\De}{\Delta}

\newcommand{\CD}{\mathcal{D}}
\newcommand{\CH}{\mathcal{H}}

\newcommand{\D}{{\boldsymbol  D}}

\newcommand{\CN}{\mathcal{N}}

\newcommand{\intl}{\int\limits}

\newcommand{\lf}{\left}
\newcommand{\rt}{\right}

\DeclareMathSymbol{\Gamma}{\mathalpha}{letters}{"00}
\DeclareMathSymbol{\Lambda}{\mathalpha}{letters}{"03}

\begin{document}

\author{P. Germain, N. Masmoudi and J. Shatah}

\title{Global solutions for the gravity water waves equation in dimension 3}

\begin{abstract}
We show existence of global solutions for the gravity water waves equation in dimension 3, in the case of small data. The proof combines energy estimates, which yield control of $L^2$ related norms, with dispersive estimates, which give decay in $L^\infty$. To obtain these dispersive estimates, we use an analysis in Fourier space; the study of space and time resonances is then the crucial point.
\end{abstract}
\maketitle

\section{Introduction}

We consider the three-dimensional irrotational water wave problem in the presence of the gravity~$g$. The velocity of the fluid is denoted by $v$. We assume that the domain of the fluid is given by  $\Omega = \{ (x,z) = (x_1,x_2,z) \in \mathbb{R}^3 ,  z\leq h(x,t) \}$, where the graph of the  function $h$, i.e.
$S= \{ (x,h(x,t)),  x \in \mathbb{R}^2\}$  represents the free boundary of the fluid which moves by the normal velocity of the  fluid.  In this setting the Euler equation and the boundary conditions are  given by

\begin{align*}
&\begin{cases}
\tag{E}
\D_tv\stackrel{def}{=}\p_t v + \nabla_v v = -\nabla p -  g e_3 \quad  (x, z) \in \Omega,\\
\nabla \cdot v =0 \quad (x, z) \in \Omega,\\
\end{cases}\\
&
\begin{cases}\tag{BC}
\p_t h+ v \cdot\nabla_{(x,z)} (h-z)=0 \quad x\in \RR^2 ,\\
p|_{S} =  0.
\end{cases}
\end{align*}

Since the flow is assumed to be irrotational, the Euler equation can be reduced to an equation  on the boundary and thus the system of equations (E--BC) reduces to a system defined on $S$. This is achieved by introducing the potential ${\psi_\CH}$ where $v= \nabla \psi_\CH$. Denoting the trace of the potential on the free boundary by $\psi(x,t) = {\psi_\CH} (x,h(x,t),t)$,  the system of equations   for $\psi$ and $h$ are \cite{sulemsulem}
\[
\tag{WW}
\left\{
\begin{array}{l} \partial_t h = G(h) \psi \\
\partial_t \psi = -g h -\frac{1}{2} |\nabla \psi|^2 + \frac{1}{2(1+|\nabla h|^2)} \left(G(h) \psi + \nabla h \cdot \nabla \psi \right)^2 \\
(h,\psi)(t=2) = (h_0,\psi_0).  \end{array}
\right.
\]
where $G(h) = \sqrt{1 + |\nabla h|^2} \CN$, $\CN$ being the Dirichlet-Neumann operator associated with $\Omega$.

Before stating our result we introduce the Calderon operator $\Lambda \overset{def}{=} |D|$,  the complex function $u \overset{def}{=}  h + i \Lambda^{1/2} \psi$,  its initial data $u_0 \overset{def}{=} h_0 + i \Lambda^{1/2} \psi_0$, and its profile $f \overset{def}{=}  e^{it\Lambda^{1/2}}u$. Then 
\begin{nth} Let $\delta$ denote a small constant and $N$ a big enough integer. Define the norm
$$
||u||_X\stackrel{def}{=} \sup_{t\geq 2} \quad t \left\| u  \right\|_{W^{4,\infty}} +  t^{-\delta}\left\|  u \right\|_{H^{N}} +  t^{-\delta} \left\| x f \right\|_{L^2} +\| u \|_{L^2} \,\,.
$$
There exists an $\varepsilon>0$ such that if the data satisfies $\displaystyle \| e^{it\Lambda^{1/2}} u_0 \|_X < \varepsilon$, then there exists a unique global solution $u$ of $(WW)$ such that $\displaystyle \|u\|_X \lesssim \varepsilon$.
\end{nth}
  
The estimates leading to the above theorem give easily that the solutions scatter.

\begin{corollary}\label{scat}
There exists a constant $C_0$ such that: under the conditions of the previous theorem, there exists $f_\infty \in H^{N-N C_0 \delta} \cap L^2(x^{2-C_0 \delta}dx)$ such that $\displaystyle f(t) \rightarrow f_\infty$ in $H^{N- N C_0 \delta} \cap L^2(x^{2-C_0 \delta}dx)$ as~$t \rightarrow \infty$.
\end{corollary}

Some comments on the stated results.  Our approach in writing this article has not been to strive for a minimal $N$; rather, we chose to give a proof as concise as possible, by allowing at some places losses of derivatives.  The fact that the data for (WW) are given at $t=2$ does not have a deep significance, it is simply a matter of  convenience to avoid writing $1/(2 +t)$  when performing estimates, since the $L^\infty$ decay of $1/{t}$ given by the linear part of the equation is not integrable at $0$.  The precise conditions on the constants $\delta,N,\varepsilon$ under which the theorem holds are, as an inspection of the proof reveals: $\delta$ less than a universal constant $\delta_0$, $\frac{1}{N} \lesssim \delta$ and $\varepsilon \lesssim \delta$. Finally, we notice that the condition on $u_0$ in the theorem is satisfied if, for instance, $u_0 \in W^{6,1} \cap H^N \cap L^2(x^2 dx)$.

There is an extensive body of literature  on local well posedness and energy estimates for (WW) in Sobolev spaces starting with the work of  Nalimov \cite{na74}  for small data in 2~dimensions (see also H.~Yoshihara \cite{yo82}).   The first breakthrough in solving the well posedness for general data came in
the work of S.~J.~Wu \cite{wu97,wu99} who solved the problem in $2$ and $3$
dimensions.  There are many other works on local well posedness: we mention the work of  W.~Craig \cite{ca85},  T. J.~Beal, T.~Hou, and J.~Lowengrub
\cite{bhl93},  M.~Ogawa and A.~Tani \cite{ot02}, G.~Schneider and E. C.~Wayne \cite{sw02},  D.~Lannes \cite{la05},  D.~Ambrose  and N.~Mamoudi \cite{am06} \cite{am07},  and  P.~Zhang and Z.~Zhang \cite{zz06}.   We also mention that for the full system (E) there are well posedness results given by    D.~Christodoulou and H.~Lindblad \cite{cl00}, D. ~Coutand   and S.~Shkoller   \cite{cs05},   J.~Shatah and C.~Zeng \cite{sz06}. 
Recently S. J. Wu proved almost global existence for small data for  (WW)  in two space dimensions \cite{wu09}. 

Our  proof of global existence is based on  the method of space-time resonances introduced in \cite{GMS08}.   It is the  key notion to understand the long time behavior of solutions to the water waves problem.  The space time resonance method brings together the  normal forms  method \cite{sh85} and  the vector fields method \cite{kl85}, and is illustrated below.

\subsection*{Space-time resonances}

If one considers a nonlinear dispersive equation on $\RR^n$
\[
\p_t u = i L \Big(\frac 1 i \nabla \Big) u + au^2 + \dots ,
\]
solutions can be represented in Fourier space in terms of Duhamel's formula 
\begin{equation}
\label{duhamelversion}
\hat{f}(t,\xi) = \hat{u}_0(\xi) + \int _0^t \!\!\int e^{is\varphi(\xi,\eta)}a \hat{f}(s,\eta)
\hat{f}(s, \xi - \eta)  + \, \dots  \, d\eta  ds, 
\end{equation}
where $\varphi = -L(\xi)  + L(\eta)  +L(\xi - \eta)$,   $\hat{f}(\xi)=e^{- i L (\xi) t} \hat u$, and where $\hat u$ denotes the Fourier transform of $u$.  The reason for writing the equation in terms of $\hat f$ rather then $\hat u$ is that $\hat f$ is expected to be non oscillatory and thus for the quadratic terms the oscillatory behavior is restricted to $e^{is \varphi(\xi,\eta)}$.

\noindent $\bullet$ \textit{Time resonances} correspond to the classical notion of resonances, which is well known from ODE theory.
Thus quadratic time resonant frequencies are defined by
\[
\mathscr{T} = \{ (\xi, \eta ) ; \varphi(\xi,\eta) =0\} =
\{ (\xi, \eta ) ;
L(\eta) + L(\xi -\eta)
= L (\xi) \}.
\]
In other words, time resonances correspond to stationary phase in $s$ in~(\ref{duhamelversion}).
In the absence of time resonant frequencies, one can integrate by parts in time to eliminate the quadratic nonlinearity  in favor of a cubic nonlinearity.   The absence of time resonances for sufficiently high degree  usually indicates that the asymptotic behavior of solutions to the nonlinear equation is similar to  the asymptotic behavior of linear solutions.

However, for dispersive equations time resonance can only tell part of the story.  The reason is that time resonances are based upon plane wave solutions to the linear equation 
\(
e^{i (L (\xi) t + \xi \cdot x )},
\)
while we are usually interested in  spatially localized solutions.   For this reason we introduced the notion of space resonances. 
\medskip

\noindent $\bullet$ \textit{Space resonances} only occur in a dispersive PDE setting. The physical phenomenon underlying this notion is the following: wave packets corresponding to different frequencies may or not have the same group velocity; if they do, these wave packets are called space resonant and they might interact since they are localized in the same region of space. If they do not, they are not space resonant and their interaction will be very limited since their space-time supports are (asymptotically) disjoint. 
If one considers two linear solutions
$u_1$ and $u_2$ whose data are wave packets localized in
space around the origin
and in frequency around $\eta$ and $\xi - \eta$, respectively,
then the solutions $u_1$ and $u_2$ at large time~$t$
will be spatially localized around
\(
(-\nabla L (\eta) t)
\)
and $(- \nabla L (\xi-\eta) t) $.
Thus quadratic spatially  resonant frequencies  are  defined as 
\[
\mathscr{S} =
\{ (\xi, \eta) ; \nabla L(\eta) = \nabla L (\xi - \eta) \} = \{ (\xi, \eta) ; \nabla_\eta\varphi(\xi,\eta)=0\};
\]
in other words, space resonances correspond to stationary phase in $\eta$ in~(\ref{duhamelversion}). 
Spatial localization can be measured by computing weighted
norms, namely $\left\| x^k u(0,\cdot) \right\|_2 = \left\| \partial_\xi^k \widehat{u}(0,\cdot) \right\|_2$ for the data, or $\left\| x^k f(t,\cdot) \right\|_2 = \left\| \partial_\xi^k \widehat{f}(t,\cdot) \right\|_2$ for the solution.

This leads to the observation that if there are no space resonant  frequencies then one can capture the spatial localization of the solution by integrating by parts in $\eta$ in  equation \eqref{duhamelversion}.

\noindent $\bullet$ \textit{Space-time resonances} correspond to space and time resonances occurring at the same point in Fourier space. The space-time resonant set is given by
\[
\mathscr{R} = \mathscr{T} \cap \mathscr{S}.
\]

\medskip

It is clear from the above description that the  asymptotic behavior of nonlinear dispersive equations, of the type considered here,  is governed to a large part by the space-time resonant set $\mathscr{R}$.  Thus if  the set  $\mathscr{R}\ne \emptyset$ and the dispersive time   decay  is weak, then one faces serious  difficulties in proving that  small solutions exist globally and behave  asymptotically like  linear solutions.  However if the 
quadratic  nonlinearity  vanishes on $\mathscr{R}$ then  we say that the quadratic term is null and try to eliminate it by integration by parts in time or frequency. After dealing with the quadratic nonlinearity we  move on to study the cubic interactions where we have similar definitions of $\mathscr{T}$, $\mathscr{S}$, and $\mathscr{R}$.

\medskip

When applying this method to  the problem on hand, (WW),  one can easily compute that the quadratic terms are null; as for the cubic terms, their resonance structure is more subtle, but they can be estimated by the same method. However the integration by parts procedure, in addition to introducing multipliers with non smooth symbols,  introduces new singularities due to the non smoothness of the dispersion relation which is given by $|\xi|^{\frac12}$.  After dealing with these difficulties,  the problem is reduced to 
studying the asymptotic behavior of  quartic and higher order terms.  But for quartic terms the dispersion is strong enough to overcome any effect of resonance, thus we are able to prove global existence.

\subsection*{Plan of the article}
In order to prove the existence result, we will prove the a priori estimate $\|u\|_X < \infty$.
The organization of the article is as follows:
\smallskip

\emph{Section~\ref{hn} } contains  energy estimates (control of the $H^N$ part of the $X$ norm), using the formalism developed in~\cite{sz06}.
Most of the rest of the paper is dedicated to controlling the $L^2(x^2 dx)$ and $L^\infty$ parts of the $X$ norm. 
\emph{Section~\ref{decay}}  is the first step in this direction: the space-time resonant structure of the quadratic and cubic terms in the nonlinearity of $(WW)$ is studied.  In \emph{Section~\ref{normalform}}, a normal form transform is performed. This leads to distinguishing four kinds of terms: quadratic, cubic weakly resonant, cubic strongly resonant, and quartic and higher-order, or remainder terms.   Quadratic terms are treated in \emph{Section~\ref{quadratic}},   weakly resonant cubic terms are treated in \emph{Section~\ref{sectionweak}}, strongly resonant cubic terms are treated in \emph{Section~\ref{sectionstrong}}, and 
quartic and higher-order terms are treated in \emph{Section~\ref{sectionremainder}}. Finally, scattering (Corollary~\ref{scat}) is proved in \emph{Section~\ref{sectionscat}}.

The appendices contain more technical developments needed in the main body.
Some classical harmonic analysis results are recalled in \emph{appendix~\ref{appendixlin}} and results on standard pseudo-product operators are given in  \emph{appendix~\ref{appendixpp}}. Particular classes of bilinear and  trilinear  pseudo-product operators are defined and studied in \emph{appendices~\ref{appendquad}}  and  \emph{\ref{appendcub}}   respectively.   Estimates on the Dirichlet-Neumann operator are proved in \emph{appendix~\ref{ap:e}}, and 
finally  \emph{appendix~\ref{ap:f}} gives bounds on the quartic and higher order terms in the nonlinearity of (WW).

\subsection*{Notation} Throughout this manuscript we will use the following notations.
\medskip
\\ $A \lesssim B$ or $B \gtrsim A$ means that that $A \leq C_0 B$ for a constant $C_0$.
\\ $A \sim B$ means that $A \lesssim B$ and $B \lesssim A$.
\\ $A << B$ or $B>>A$ means that $A \leq \varepsilon_0 B$ for a small enough constant $\varepsilon_0$.
\medskip

\noindent We denote by $c$  (respectively $C$) small enough (respectively big enough) constants whose values may change from one line to the other.
\medskip
We  use   standard function spaces; their definitions are (see appendix~\ref{appendixlin} for the definition of the Littlewood-Paley operators $P_j$)
\medskip
\\Homogeneous Besov spaces $\dot{B}^s_{p,q}$: $\displaystyle \|f\|_{\dot{B}^s_{p,q}} \overset{def}{=} \left[ \sum_j \left( 2^{js} \|P_j f\|_p \right)^q \right]^{1/q}$.
\\ Homogeneous H\"older spaces $\dot{C}^\alpha$, with $0<\alpha<1$: equivalent to $\dot{B}^\alpha_{\infty,\infty}$.
\\Inhomogeneous Besov spaces $B^s_{p,q}$: $\displaystyle \|f\|_{B^s_{p,q}} \overset{def}{=} \left\| P_{<0} f \right\|_p + \left[ \sum_{j\geq 0} \left( 2^{js} \|P_j f\|_p \right)^q \right]^{1/q}$.
\\Inhomogeneous Sobolev spaces $W^{k,p}$: $\displaystyle \|f\|_{W^{k,p}} \overset{def}{=} \|f\|_{L^p} + \|\nabla^k f\|_{L^p}$.\\
\\Inhomogeneous Sobolev spaces $H^N$: equivalent to $W^{N,2}$ or $B^N_{2,2}$.\\
\medskip

\noindent We also make use of the standard embedding relation between these spaces, and first of all of the Sobolev embedding theorem: if $1 \leq p < q \leq \infty$ and $k > \frac{2}{q} - \frac{2}{p}$, then 
$
\|f\|_p \lesssim \|f\|_{W^{k,q}} .
$
\medskip

\noindent  The energy estimate will be reproduced along the lines of  \cite{sz06},  thus in deriving  the energy estimates we will adhere to the same notation as in \cite{sz06}.  
\medskip
 \\  $A^*$: the adjoint operator of an operator;
 $A_1 \cdot A_2 =\text{trace} (A_1 (A_2)^*)$, for two operators.
 \\  $D$ and $\p$: differentiation with respect to spatial variables.
 \\  $\nabla_X$:  the directional directive in the direction $X$. \\  $\D_t = \p_t + v^i\p_{x^i}$ the material derivative  along the particle path.
 \\$S = \{(x,h(x))\in\RR^3; \quad  h:\RR^2\to \RR\}$ a surface in $\RR^3$ given as the graph of $h$.
 \\$\Omega = \{(x,z)\in\RR^3; \quad z\le h(x)\}$ the region in $\RR^3$ bounded by $S$.
 \\  $N$: the outward unit normal vector to $S =\p \Omega$.
 \\  $\perp$ and $\top$: the normal and the tangential components to $S$ of the relevant quantities.
 \\$\CD_w = \nabla_w^\top$,  for any $w$ tangent $S$;  $\CD$ the covariant differentiation on $S \subset \RR^3$.
 \\  $\Pi$: the second fundamental form of $S$, $\Pi(w) = \nabla_w N$; $\Pi(X, Y)=\Pi(X) \cdot Y$.
\\  $\kappa$: the mean curvature of $S$, i.e. $\kappa = \text{trace}\,\Pi$.
\\$f_\CH$: the harmonic extension of $f$ defined on $S$ into   $\Omega$.
\\$\CN (f) = \nabla_N f_\CH : S \to \RR$:  the Dirichlet-Neumann operator.
 \\  $\Delta_S= \text{trace} \CD^2$: the Beltrami-Lapalace operator on $S$.
 \\ $(- \Delta)^{-1}$  is the inverse Laplacian on $\Omega$ with zero Dirichlet data.
\medskip

\noindent The decay and weighted estimates will be carried out using harmonic analysis
 techniques,  where we employ standard notations and denote by:
 \medskip
 \\$\widehat{f}$ or $\mathcal{F} f$: the Fourier transform of a function $f$.
\\ $m(D) f = \mathcal{F}^{-1} m(\xi) (\mathcal{F} f)(\xi)$: the Fourier multiplier  with  symbol   $m$.  In particular  $\Lambda = |D|$.
 \\ $\mathcal{A}(x_1,\dots ,x_n)$: a smooth scalar or vector-valued function in all its arguments $(x_1,\dots, x_n)$.
\\ $\mathcal{A}(x_1,\dots ,x_n)[y]$: a function which is linear in $y$ and smooth in $(x_1,\dots,x_n)$
\\$\mathcal{A}(x_1,\dots \,x_n)[y,z]$: a function which is bilinear in $(y,z)$ and smooth in $(x_1,\dots,x_n)$.

\section{Local existence and Energy Estimates}\label{hn}
There are a variety of methods that have been developed to prove local  well posedness and to obtain energy estimates  for (E-BC) in Sobolev spaces.   For example in \cite{sz09} local well posedness was established in $H^m$ for $m> 5/2$.  The success of all of these methods hinges upon finding the appropriate
cancellations needed to isolate the highest order derivative term with the appropriate
Rayleigh-Taylor sign condition. Here we follow the method developed in \cite{sz06} to obtain higher energy estimates.  The idea is the following: since the  system  is reduced to  the boundary   we apply the surface Laplacian $\Delta_S$ to (BC) instead of partial derivatives.  This leads to an evolution equation for the mean curvature $\kappa$.

\subsection*{Conserved energy}
We start  by deriving the conserved energy of the (E--BC) system.  Observe that the Euler system in the presence of gravity can be written as 

\[
\p_t v + v \cdot \nabla v
=- \nabla p - ge_3 \overset{\text{\tiny def}}{=} - \nabla q
\]
where $q = p_{vv} + gh_ \mathcal{H}$, 
\begin{alignat*}{2}
p_{v v}\in \dot{H} ^1 (\Omega): - \De p_{vv} &= \operatorname{trace} [(Dv)^2] & \qquad & p_{v v} \big| _S = 0\\
h_ \mathcal{H} \in \dot{H} ^1 (\Omega): - \De h_\CH & = 0     &&h_ \mathcal{H}\big|_S = h
\end{alignat*}
This splitting of $q$ is natural since $p_{v v} $
is the Lagrange multiplier of the volume preserving  condition
$(\operatorname{div} v =0)$,
and $g\nabla  h_ \mathcal{H}$ is the variational derivative of the potential energy
$\mathscr{V} =\!\!\! \intl_{\Omega\cap \{z>0\}} \!\!-\!\! \intl_{\Omega^c \cap \{z <0\}}\!\!\!\!\!gz\, dxdz $.

\begin{proposition}\label{p1}
The conservation of energy for {\rm (E--BC)}  system  is
\[
E  = \int_\Omega \frac 12 |v|^2dx dz  + \mathscr{V}
= \int_S \frac 12 \psi \CN \psi
+ \frac 12 g h^2(1+|\nabla h|^2|)^{-\frac 12} \; dS
= \text{constant}.
\]
\end{proposition}

\proof
The fact that $E$ is conserved is a consequence
of the Lagrangian formulation of the Euler equations.
Here we will present a direct proof.  Differentiating with respect to $t$
\[ \begin{split}
\frac{dE}{dt} & = \int_\Omega
 v \cdot \frac{\p v}{\p t} \, dx dz 
+ \int_S \frac 1 2
|v|^2 v \cdot N \,dS
+ g \int_S h v \cdot N \, dS\\
& = - \int_\Omega v \cdot \nabla \frac{|v|^2}2
+ v \cdot \nabla q\,dxdz 
+ \int_S \frac 12 |v|^2 v \cdot N \, dS
 + g \int_S h v \cdot N \, dS\\
& =  - \int_S q v \cdot N \, dS
+  g \int_S h v \cdot N \, dS = 0 .
\end{split}\]
\endproof

Note that the conservation of energy implies bounds on $\CN^\frac12\psi$ and on $h$ in $L^2(S)$. Thus to derive higher energy estimates on $\psi$ and $h$ we need to  relate the Dirichlet to Neumann operator $\CN$  to differentiation on $S$. This is done by standard   estimate on harmonic functions on $\Omega$ which imply that as far as estimates are concerned,  $\CN$ behaves like $\sqrt{-\Delta_S}$ plus lower order terms. These estimates were carried out in \cite{sz06} without detailing explicitly how the constants in the inequalities depend on the various norms of the solution.  Here we will detail how the constants in the energy inequalities depend on several derivatives of the solution in the $L^\infty$ norm.  \emph{ For the remainder of this section we assume that $S$ is given by the graph of a smooth function $h$   and that   $\|h\|_{W^{5, p} (\RR^2)} \le \varepsilon_0$ for some $2<p<4$, and $\varepsilon_0$ sufficiently small. Thus  for any $\ell\ge 0$ we do not need to distinguish between $H^\ell(S)$ and $H^\ell(\RR^2)$, i.e. for any $\varphi\in H^\ell(S)$ we have}
$$
\|\varphi\|_{H^{\ell} (S) } + \|h\|_{H^{\ell }(S)}  \sim \|\varphi\|_{H^{\ell} (\RR^2) } + \|h\|_{H^{\ell }(\RR^2)}.
$$
\begin{proposition}\label{prop:elliptic}
For any $\varphi$ defined on $S$ 
\begin{align}\label{eq:1}
&\|\CN \varphi\|_{W^{2,\infty}(\RR^2)} \lesssim \|D \varphi\|_{W^{3,\infty}} +   \|\Lambda^\frac12\varphi\|_{L^\infty}
\end{align}
\end{proposition} 
The proof of this proposition follows from standard arguments applied to double layer potential and will be given in  appendix \ref{ap:e}.   Note that inequality \eqref{eq:1} can be improved by using $C^\alpha$ norms, however \eqref{eq:1} is sufficient for our purpose.  One should compare the above proposition to the detailed behavior of $\CN$ as an analytic function of $h$  obtained in \cite{cn00}.

 Based on this proposition and the work done in \cite{sz06}  we make the following observations :
\medskip

\noindent $\bullet$ Since  $\De_S h = (1+ |\nabla h|^2)^{-2} [(1 + |\p_2h|^2)\p_1^2 h  -2 \p_1h\p_2h\p_{12}^2h + (1 + |\p_1h|^2)\p_2^2 h ]$  and the mean curvature is given by $\kappa = -(1+|\nabla h|^2)^{\frac 12} \De_S h$,  then  for $k \ge \ell \ge 2$ 
\[
\|h\|^2_{H^\ell(\RR^2)} \sim \int_S  h^2 + h(-\De_S^\ell) h\; dS \sim \int_S  h^2 + \kappa(-\De_S)^{\ell-2} \kappa\; dS  
\]

\noindent $\bullet$  By defining 
$$
\|\varphi\|_{H^{\ell + 1/2}(S)} = \|\varphi\|_{H^{\ell}(\RR^2)} + \|\Lambda^{\frac 12}\CD^\ell\varphi\|_{L^2(\RR^2)}
$$
then from the single layer potential formula \eqref{eqrho} it is easy to see that 
$$ \|\varphi\|_{L^{2}(S)}  + \|\Lambda\varphi\|_{L^2(\RR^2)}\sim \|\varphi\|_{L^{2}(\RR^2)} + \|\CN\varphi\|_{L^2(\RR^2)}
$$
and by interpolation
$$
\|\varphi\|_{H^{\ell + 1/2}(S)}  \sim \|\varphi\|_{H^{\ell}(S)} + \|\CN^{\frac 12}\CD^\ell\varphi\|_{L^2(\RR^2)}
$$

%


\noindent $\bullet$ For any function $g$ defined on $\Omega$,   and  for any function $\varphi$ defined on $S$
\begin{align}
&D^2 \varphi_\CH (N, N) = \CN^2 (\varphi) - 2\nabla_N (- \Delta)^{-1} (D N_\CH \cdot D^2\varphi_\CH) -
\CN(N) \cdot (\CN(\varphi) N + \nabla^\top \varphi) \\
&\Delta g= \Delta_S g + \kappa \nabla_N g + D^2 g (N, N) \quad \mathrm{on}\quad   S\Longrightarrow\\
&(-\Delta_S - \CN^2) \varphi = \kappa \CN(\varphi ) - 2\nabla_N (-
\Delta)^{-1} (D N_\CH \cdot D^2 \varphi_\CH) - \CN(N) \cdot (\CN(\varphi ) N +
\nabla^\top \varphi ) \label{eq:N2}
\end{align}
Consequently from equation \eqref{eq:N2}  $\CN^2$ is like $-\Delta_S$ plus lower order terms (involving derivatives of $h$),
 since after integrating by parts twice we have

\begin{equation*}
\int_{S} \varphi (-\Delta_{S}\varphi) - | \CN\varphi|^2 \; dS = \int_{S} \kappa \varphi \CN(\varphi) + \varphi \CN(N) \cdot (\CN(\varphi) N + \nabla^\top \varphi)
\; dS  - 2 \int_\Omega D N_\CH (\nabla \varphi_\CH) \cdot \nabla \varphi_\CH \; dx.
\end{equation*}
Therefore  
 for $ \ell\ge 2$
\[
\|\varphi\|^2_{H^\ell(S)} +  \|h\|^2_{H^\ell} \sim \int_{S}| \varphi |^2 +  | \CN^\ell\varphi|^2 \; dS + \|h\|^2_{H^\ell}.
\]

\noindent $\bullet$ Since the Euler  flow is assumed to be irrotational, then from \cite{sz06} (equation (3.6))
\begin{equation}\label{dtk}
\D_t \kappa = 
- (\Delta_{S}  \nabla \psi_\CH )\cdot N -
2 \Pi \cdot ((D^\top|_{S})  \nabla \psi_\CH),
%
\end{equation}
and from  proposition 4.3 of \cite{sz06},  we have
\begin{proposition}
$
\|\nabla \psi_\CH \|_{H^{\ell-1/2} (S) } \lesssim \| \D_t \kappa \|_{H^{\ell-5/2} (S)}   +\|h\|_{H^{\ell-1/2}} + \sqrt{E},
$ for $\ell \ge 3$.
\end{proposition}
\proof[Sketch of the proof]  Since the above statement is contained in the afore cited  proposition we will only present the idea of the proof:   on $S$,  split $\nu = N\cdot \nabla \nabla \psi_\CH$   into tangential and normal components,  compute  the surface divergence and curl of these quantities, and use \eqref{dtk}  
\[
- (\Delta_{S}\,  \nabla\psi_\CH )\cdot N=  \D_t \kappa +
2 \Pi \cdot ((D^\top  \nabla\psi_\CH)
\]
to obtain the stated estimate.

Let $\nu^\top$ and $\nu^\perp$ denote the tangential   and normal components of $\nu$.  The tangential divergence of $\nu^\top$ is given by\footnote{Here we introduced the notation $A=O(B)$ to mean $\|A\|_{H^s}\lesssim \|B\|_{H^s}$. Thus $A=O(\sum_{j=1,2} |D^j\psi|+ |D^jh|)^2)$ imply  $\|A\|_{H^s}\lesssim( \|D\psi\|_{W^{1,\infty}} +  \|Dh\|_{W^{1,\infty}})  ( \|D\psi\|_{H^s} +  \|Dh\|_{H^s})$.}
\begin{equation}\label{E:cddotv1}
\CD \cdot \nu^\top =  (\Delta_S\, \nabla\psi_\CH) \cdot N +
(D^\top \nabla\psi_\CH )\cdot \Pi  =  -\D_t \kappa  + O(\sum_{j=1,2} |D^j\psi|+ |D^jh|)^2). 
\end{equation}
and the tangential curl by
\begin{equation}
\omega_\nu^\top (X_1) \cdot X_2 =\Pi(X_1) \cdot(X_2\cdot \nabla\nabla\psi_\CH )- \Pi(X_2) \cdot (X_1\cdot \nabla\nabla\psi_\CH)=  O(\sum_{j=1,2} |D^j\psi| + |D^j h|)^2. 
\label{E:tcurl1}
\end{equation}
where   $\{X_1,   X_{2}\}$   an orthonormal
frame of $T S$.   Thus  
\begin{equation}\label{eq:123}
\| D \nu^\top\|_{H^{\ell-5/2}} \lesssim   \| \D_t \kappa \|_{H^{\ell-5/2} (S)} +   \| D \psi\|_{H^{\ell -3/2}}+\|D h\|_{H^{\ell -3/2}} +\sqrt{E}.
\end{equation}
%
To bound $\nu^\perp$ let $\tilde\nu = N_\CH\cdot \nabla \nabla \psi_\CH$ and note that
\[
\begin{split}
&\CN\nu^\perp= \nabla_N (\nabla_{N_\CH}\, (\nabla\psi_\CH) \cdot N_{\CH}) -\Delta^{-1}[\Delta
(\nabla_{N_\CH}\, (\nabla\psi_\CH) \cdot N_{\CH})]\\
&\CD\cdot \nu^\top =\nabla\cdot \tilde\nu - \kappa \nu^\perp - \nabla_N
(\nabla_{N_\CH}\, (\nabla\psi_\CH \cdot N_{\CH}) + N \cdot D (\nabla\psi_\CH)(\CN (N)).
\end{split}
\]
Since  $\nabla\cdot  \tilde \nu =  D(\nabla\psi_\CH)\cdot D N_\CH$,  we obtain 
$\CN\nu^\perp=  -  \CD\cdot \nu^\top  + O(\sum_{j=1,2} |D^j\psi|+ |D^jh|)^2)
$  which together with \eqref{eq:123} imply
\[
\| D  N\cdot \nabla \nabla \psi_\CH \|_{H^{\ell-5/2}(S)} = \| D \nu\|_{H^{\ell-5/2}(S)} \lesssim   \| \D_t \kappa \|_{H^{\ell-5/2} (S)} +   \| D \psi\|_{H^{\ell -3/2}(S)}+\|D h\|_{H^{\ell -3/2}(S)} + \sqrt{E}.
\]
Since $\CN$ is equivalent to one derivative in norm, we obtain the stated  bound.
\endproof

These observations imply 

\begin{proposition}\label{prop:2} Assume that $v=\nabla \psi_\CH$ and $h$ solve  the {\rm(E-BC)}  system,   then for $\ell \ge 3$
 \[
 \|h(t)\|_{H^\ell(S)} +\|v(t)\|_{H^{\ell- 1/2}(S)}
\sim \|\D_t \kappa(t) \|_{H^{\ell- 5/2}(S)}  +\|\kappa(t)\| _{H^{\ell -2}(S)} + \sqrt{E}.
\]

\end{proposition}

\subsection*{Commutators estimates}  From \cite{sz06} we have for any function $f$ defined on $\Omega$ and $\varphi$ defined on  $S$, and where $S$ is moving by the normal component of the velocity
 $v=  \nabla \psi_\CH$  
\begin{align}
&\D_t \nabla f = \nabla \D_t f - (D^2\psi_\CH) (\nabla f), \label{eq:dtgrad}\\
&\D_t \varphi_\CH = (\D_t \varphi)_\CH + \Delta^{-1} (2D^2\psi_\CH \cdot D^2 \varphi_\CH)\\
&\D_t \Delta^{-1} f = \Delta^{-1} \D_t f + \Delta^{-1} ( 2D^2\psi_\CH \cdot
D^2 \Delta^{-1} f ).  \label{eq:dtlapinv}
\end{align}
These equations lead to the following commutators formulas 
%
\begin{align}\label{CM1}
&[\Delta_{S}, \D_t] f= 2 \CD^2 f
\cdot ((D^\top|_{T S}) \nabla \psi_\CH)  + \nabla^\top f \cdot \Delta_S  \nabla \psi_\CH  - \kappa \nabla_{\nabla^\top f}  \nabla \psi_\CH \cdot N\\
&[\D_t, \CN]f =\nabla_N \label{dtn}
\Delta^{-1} 2D^2 \psi_\CH \cdot D^2 f_\CH  -
2D^2\psi_\CH (\nabla f_\CH, N) +\nabla f_\CH \cdot N D^2\psi_\CH (N, N).
\end{align}
which together with \eqref{eq:N2} imply that for $s\ge 2$ these commutators are bounded operators on spaces given in  equations (4.23) and (A.14) of \cite {sz06}.
Amplifications  of these  bounds  are given in the  proposition below.  Recall that 
$u =  h + i \Lambda^{1/2} \psi$.

\begin{proposition} \label{prop:2.5}  Assume that $v=\nabla \psi_\CH$ and $h$ solve the 
{\rm(E-BC)} system.  Let $w(t)= (\Lambda^{\frac12} h(t), v(t))$ defined on $S$, then the following commutator  estimates hold
\begin{align}\label{CMLAP}
&\| [\Delta_{S}, \D_t] f(t)\|_{H^\ell(S)} \lesssim 
\|u(t)\|_{W^{3,\infty}}\|f(t)\|_{H^{\ell +2}(S)} +  \|f(t)\|_{W^{3,\infty}}\|w(t)\|_{H^{\ell+2}(S)}\\
&\label{CMN}
\| [\CN, \D_t] f(t)\|_{H^\ell(S)} \lesssim 
\|u(t)\|_{W^{2,\infty}}\|f(t)\|_{H^{\ell +1}(S)} +  \|f(t)\|_{W^{2,\infty}}\|w(t)\|_{H^{\ell+2}(S)}\\
&\|[\Delta_{S}, \CN] f(t)\|_{H^\ell(S)} \lesssim  
\|u(t)\|_{W^{4,\infty}}\|f(t)\|_{H^{\ell +2}(S)} +  \|f(t)\|_{W^{3,\infty}} \|w(t)\|_{H^{\ell+5/2}(S)}
\label{CMLAPN}
\end{align}
\end{proposition}
\proof
From equation \eqref{CM1} we note that  $ [\Delta_{S}, \D_t] $ is an operator of order $2$ with coefficients depending on second derivatives of $h$ and  $ \nabla \psi_\CH$ on $S$.  Thus by H\"older and Sobolev inequalities we conclude
\[
\| [\Delta_{S}, \D_t] f(t)\|_{H^\ell(S)} \lesssim 
(\|\nabla \psi_\CH\|_{W^{2,\infty}} + \|h\|_{W^{2,\infty}})\|f(t)\|_{H^{\ell +2}} +  \|f(t)\|_{W^{2,\infty}}(\|\nabla \psi_\CH\|_{H^{\ell+2}} + \|h\|_{H^{\ell+2}}),
\]
and by proposition \ref{prop:elliptic} we conclude inequality \eqref{CMLAP}.  The proof of  \eqref{CMN}  differs from the above by the way we treat $w=  \Delta^{-1} 2D^2 \psi_\CH \cdot D^2 f_\CH$.  This is done by applying vector fields $X_a$,  which are  tangential  to $S$,   to the equation 
$$ 
\Delta w =  2D^2 \psi_\CH \cdot D^2 f_\CH  \qquad w|_S = 0
$$
to obtain $\|\nabla X^\ell_a w\|_{L^2}$ bounds.  The proof of  \eqref{CMLAPN} is similar to the proof of \eqref{CMN}.
\begin{remark}
Proposition \ref{prop:2.5} holds with $\ell$ replaced by $\ell + \frac12$.  This  follows from the definition of  $H^{\ell + \frac12}$. 
\end{remark}

\subsection*{Equation for $\kappa$ and energy estimates}  
In the presence of gravity the evolution of  $\kappa$ can be derived from \cite{sz06} and \cite{sz09} section 6. problem {\bf II}.
\begin{proposition} The evolution equation of the mean curvature $\kappa$ is given by
\begin{equation} \label{eq:kappa}
\D_t^2 \kappa +( - \nabla_N p)\CN \kappa = R
\end{equation}
where  $R = O(\sum_0^2 |D^j\nabla\psi_\CH|^2  + |D^jh|^2)$ in norm.
\end{proposition}

\proof From \cite{sz06} equation (3.15) we have
\begin{equation}\begin{split}
&\D_t^2 \kappa =  - \D_t \Delta_{S} v \cdot N - 2 \Pi \cdot (D^\top|_{S} \D_t v) + \Delta_{S} \nabla\psi_\CH\cdot (D^2\psi_\CH)(N)^\top + 2[ \CD \lf(((D^2\psi_\CH)(N))^\top\rt)   \\
&+ \Pi((D^\top|_{S} \nabla\psi_\CH)^\top)] \cdot (D^\top|_{S}\nabla\psi_\CH) + 2\Pi \cdot (\CD \nabla\psi_\CH|_{S})^2  - 2 ((D^2\psi_\CH)^*(N))^\top \cdot \Pi ((D^2\psi_\CH)^*N)^\top)
\label{E:dttk1}\end{split}
\end{equation}
From the  equation for $[\D_t, \Delta_{S}]$  \eqref{CMLAP} and Euler's equations $\D_tv = -\nabla p -ge_3$ we obtain 

\begin{equation*}
\D_t^2 \kappa = N \cdot \Delta_{S} \nabla p + \tilde R
\end{equation*}
where   $\tilde R =  O(\sum_0^2 |D^j\nabla\psi_\CH|^2  + |D^jh|^2)$ in norm.  Computing 
$N \cdot \Delta_{S} \nabla p$ 
\begin{align*}
N \cdot\Delta_{S} \nabla p = &N \cdot \Delta\nabla p  +N \cdot(\Delta_{S} \nabla p - \Delta\nabla p )\\
= &\nabla_N\Delta p - N \cdot ( \kappa\nabla_N \nabla p  +
D^2 (\nabla p) (N, N)\\
= &\nabla_N\Delta p  - N_\CH \cdot ( \kappa_\CH\nabla_{N_\CH} \nabla p  + D^2 (\nabla p) (N_\CH, N_\CH))\\
=&\nabla_N\Delta p  - \nabla_N (\kappa_\CH
\nabla_{N_\CH} p + D^2 p (N_\CH, N_\CH)) \\
&+ \nabla_N
p\CN\kappa +\kappa \nabla p \cdot \nabla_N N_\CH +2 D^2 p (N,
\nabla_N N_\CH).
\end{align*}

Since $\Delta p  = -\text{tr}(Dv)^2 = -\frac12\Delta |\nabla\psi_\CH|^2$ and $ p |_S=0$, then
\[
\|\kappa_\CH\nabla_{N_\CH} p + D^2 p (N_\CH, N_\CH)\|_{H^\ell(S)}= \|\Delta  p - \Delta_S  p \|_{H^\ell(S)} = \|  \text{tr}(Dv)^2\|_{H^\ell(S)}
\]
\[
\| \Delta (\kappa_\CH 
\nabla_{N_\CH} p + D^2 p (N_\CH, N_\CH))   \|_{H^{\ell-3/2} (\Omega)}  \lesssim  
 \|  \text{tr}(Dv)^2\|_{H^{\ell+1/2}(\Omega)}^2    \| \kappa \|_{H^{\ell-1}(S)}
\]
which implies that $N \cdot\Delta_{S} \nabla p =  \nabla_N p\, \CN\kappa  + \tilde{\tilde R}$ where $\tilde{\tilde R}=   O(\sum_0^2 |D^j\nabla\psi_\CH|^2  + |D^jh|^2)$.  This gives the stated  evolution equation for $\kappa$.
\endproof

Based on this equation we define the high energy $E_k(t) = \int_S e_k dS + E$   for  $k \ge 3$ as

\begin{equation}
\label{el}
\begin{split}
 \int_S e_k dS = & \int_S [\CN (-\Delta_S)^k \D_t  \kappa]  (-\Delta_S)^k \D_t  \kappa+ ( -  \nabla_N p)|\CN(-\Delta_S)^k \kappa|^2 \; dS \\
= &\langle \CN ( -\De_S ) ^k \D_t \kappa ,
( - \De_S )^k \D_t \kappa \rangle + \langle  ( - \nabla_N p)
\CN (-\Delta_S )^k \kappa,
\CN (-\Delta_S )^k \kappa \rangle
\end{split}
\end{equation}
where $\langle , \rangle $ denotes the inner product on
$L^2 (S)$.  Note that since  
\[
\begin{split}
&q= p_{v,v} +gh_\CH \Rightarrow
-\nabla_N p = -\nabla_N q + gN\cdot e_3\ge g-c\varepsilon_0 \\
&D_t\kappa\in H^{2k + \frac12} \; \text{and} \; \kappa \in H^{2k + 1}  \Rightarrow \nabla \psi_\CH \in H^{2k +\frac 52} (S)  \; \text {and} \; h \in H^{2k +3} 
\end{split}
 \] 
then  $ E_k \sim  \|\nabla \psi_\CH  \|_{H^{2k +5/2}(S) }^2 + \|h\|^2_{H^{2k + 3} (S) }$.
\begin{proposition} For $t \geq 2$, 
\begin{equation}
\label{elb}
E_k(t) \lesssim E_k(2) + \int_2^t \| {u}(s)\|_{W^{4,\infty}} E_k(s) \; ds
\end{equation}
\end{proposition}

\proof
To compute $\D_t e_k$ we proceed as follows:
\smallskip

\noindent $\bullet$  $|\D_t \nabla_N p|_{L^\infty} 
\le | u(t) |_{W^{4, \infty}}$. 
This follows from the identity  $\nabla p= \nabla p_{v,v} + \nabla h_\CH -g e_3$ and the fact that 
\[\begin{gathered}
p_{v,v} = - \De^{-1} D^2 \psi_{\CH} \cdot D^2 \psi_{\CH} , \\
\D_t \De^{-1} f 
= \De^{-1} \D_t f
+ \De^{-1} 
(2 D^2 \psi_{\CH} D^2 \De^{-1} f) .
\end{gathered} \]

\noindent $\bullet$  
$\D_t \CN (\Delta_S  )^k \kappa
= \CN \D_t (-\Delta)^k \kappa
+ [\D_t, \CN ]
(-\Delta_S  )^k \kappa$.  From \eqref{CMN} we have
\[
| \langle [\D_t, \CN ] (-\Delta_S  )^k \kappa, 
\CN (-\Delta_S  )^k \kappa \rangle |
\le |u|_{W^{4, \infty} }\,
E_k(t) .
\]

\noindent $\bullet$   $ \displaystyle{
\CN \D_t  (-\Delta_S  )^k \kappa
= \CN (-\Delta_S  )^k \D_t\kappa
+ \sum^{k -1}_{i =0 }
\CN (-\Delta_S )^{i}  [\D_t, -\Delta_S  ]
(-\Delta_S ) ^{ k -i -1} \kappa}$.  Since $[\D_t, \Delta_S]$ is a second order operator
we have from \eqref{CMLAP}
\[
\langle \CN\D_t  (-\Delta_S  )^k \kappa,
\CN (-\Delta_S  )^k \kappa \rangle
= \langle  \CN (-\Delta_S  )^k \D_t\kappa,
\CN (-\Delta_S  )^k \kappa \rangle
\\+ O ( |u|_{W^{4, \infty} }\,
E_k(t) ) .
 \]
\vskip .1cm

\noindent $\bullet$  Computing $\p_t  \langle \CN (-\Delta_S  )^k \D_t\kappa,
 (-\Delta_S  )^k \D_t\kappa \rangle$
\begin{multline*}
\p_t  \langle \CN (-\Delta_S  )^k \D_t\kappa,
(-\Delta_S  )^k \D_t\kappa \rangle 
= 2 \langle \D_t  (-\Delta_S  )^k \D_t\kappa,
\CN (-\Delta_S  )^k \D_t\kappa \rangle \\+ \langle  [\D_t, \CN] (-\Delta_S  )^k \D_t\kappa , (-\Delta_S  )^k \D_t\kappa \rangle
\end{multline*}

Since $[\D_t, \CN]$ is a first order operator 
we have from \eqref{CMN}
\[
| \langle  [\D_t, \CN] (-\Delta_S  )^k \D_t\kappa ,
(-\Delta_S  )^k \D_t\kappa \rangle |
\le |u|_{W^{4, \infty} }\,
E_k(t)  
\]

\noindent $\bullet$ 
\( \displaystyle{
\D_t (-\Delta_S  )^k \D_t\kappa
= (-\Delta_S  )^k \D_t^2 \kappa
+\sum^{k -1}_{i=0}
(-\Delta_S  )^i  [\D_t , -\Delta_S ]
(-\Delta_S  )^{k -i -1} \D_t \kappa.
}\)  Since  $[\D_t, \CN]$ is a second order operator 
we have by \eqref{CMLAP}
\[
\langle \D_t (-\Delta_S  )^k \D_t\kappa,
\CN (-\Delta_S  )^k \D_t\kappa \rangle
=  \langle (-\Delta_S  )^k \D_t^2 \kappa,
\CN (-\Delta_S  )^k \D_t\kappa \rangle
\\+O (  |u|_{W^{4, \infty} }\,
E_k(t) )
\]
Using equation \eqref{eq:kappa} we obtain
\[
\frac d{dt}\, E_k
= 2 \langle( -  \nabla p)
\CN (-\Delta_S  )^k \kappa -  (-\Delta_S  )^k ( \nabla_N p ) \CN \kappa \;\;,\; \;
\CN (-\Delta_S  )^k \D_t \kappa \rangle\\ +O (  |u|_{W^{4, \infty} }\,
E_k(t) )
\]
Commuting $\nabla_N p$ and $\CN$ with $(-\Delta_S )^k$
and using the fact that $[\CN, \Delta_S  ]$ is a second order operator
with error bounds given in \eqref{CMLAPN}
we conclude that
\[
\frac d{dt}\, E_k (t)
= O (  |u|_{W^{4, \infty} }\,
E_k(t) )
\]
and thus
\[
E_k (t)
\le E_k (2)
+ \int^t_2
C(\varepsilon_0 )
 |u|_{W^{4, \infty} }\,
E_k(t) \, ds.
\]
\endproof
Thus from the assumption   $\|h\|_{W^{5, p} (\RR^2)} \le \varepsilon_0$,  reference \cite{sz09}, and a limiting argument to go from smooth   initial data to data  in  $H^N$ where $N= 2k+3 \gg 5/2$, we have local well posedness  and the  energy estimate stated in the above proposition.

\section{Space time resonances of quadratic and cubic terms}

\label{decay}

Expanding $(WW)$ in powers of $h$ and $\psi$, setting $g=1$,  and keeping track of quadratic and cubic terms we obtain 
\begin{equation}
\label{aaa}
\left\{
\begin{array}{l} \partial_t h = \Lambda \psi - \nabla \cdot (h \nabla \psi) - \Lambda (h\Lambda \psi) - \frac{1}{2} \left( \Lambda(h^2 \Lambda^2 \psi) + \Lambda^2(h^2 \Lambda \psi) - 2\Lambda(h\Lambda(h \Lambda \psi)) \right) +R_1 \\
\partial_t \psi = -h - \frac{1}{2}|\nabla \psi|^2 + \frac{1}{2} |\Lambda \psi|^2 + \Lambda \psi (h \Lambda^2 \psi - \Lambda (h \Lambda \psi)) +R_2.\end{array}
\right.
\end{equation}
where $R_1$ and $R_2$ are of order $4$.  We refer to the book of Sulem and Sulem~\cite{sulemsulem} for  the above expansion (also see the remark at the end of appendix \ref{ap:f}).

\subsection*{Writing the equation in Fourier space} Recall that
$$
u \overset{def}{=} (h+ i\Lambda^{\frac 12} \psi)\quad,\quad u_0 \overset{def}{=} e^{-2i\Lambda^{1/2}}(h_0+ i\Lambda^{\frac 12} \psi_0)\quad\mbox{and}\quad
f \overset{def}{=} e^{it\Lambda^{1/2}} u = e^{it\Lambda^{1/2}} (h + i \Lambda^{1/2} \psi ).
$$
Writing  Duhamel formula for $f$ in Fourier space yields
\begin{multline}
\label{eqfourier}
\widehat{f}(t,\xi) =  \widehat u_0(\xi) + \sum_{\tau_{1,2} = \pm} \sum_{j=1}^2 c_{j,\tau_1,\tau_2} \int_2^t \!\!\int e^{is\phi_{\tau_1,\tau_2}} m_j(\xi,\eta) \widehat{f_{-\tau_1}}(s,\eta) \widehat{f_{-\tau_2}}(s,\xi-\eta)\,d\eta \,ds \\
+ \sum_{\tau_{1,2,3}= \pm} \sum_{j=3}^4 c_{j,\tau_1,\tau_2,\tau_3} \int_2^t \!\!\int\!\!\!\int e^{is\phi_{\tau_1,\tau_2,\tau_3}} m_j(\xi,\eta,\sigma) \widehat{f_{-\tau_1}}(s,\eta) \widehat{f_{-\tau_2}}(s,\sigma) \widehat{f_{-\tau_3}}(s,\xi-\eta-\sigma)\,d\eta\, d\sigma\, ds \\
+ \int_2^t e^{is |\xi|^{1/2}} \widehat{R} (s,\xi )\,ds \\
\end{multline} 
where $c_{j,\pm,\pm}$ and $c_{j,\pm,\pm,\pm}$ are complex coefficients, $f_{+} \stackrel{def}{=} f$, $  f_{-} \stackrel{def}{=} \bar{f}$, and $R \overset{def}{=}R_1+i\Lambda^{\frac 12}R_2$ is the remainder term of order $4$.  The phases are given  by
\begin{equation*}
\begin{aligned}
& \phi_{\pm,\pm}(\xi,\eta) = |\xi|^{1/2} \pm |\eta |^{1/2} \pm |\xi-\eta|^{1/2} \\
& \phi_{\pm,\pm,\pm}(\xi,\eta,\sigma) = |\xi|^{1/2} \pm |\eta |^{1/2} \pm |\sigma|^{1/2} \pm |\xi-\eta-\sigma|^{1/2}      ,
\end{aligned}
\end{equation*}
and the multilinear symbols are defined by
\begin{equation*}
\begin{aligned}
& m_1(\xi,\eta) \overset{def}{=} \frac{1}{|\eta|^{1/2}}\left(\xi \cdot \eta - |\xi||\eta| \right)\\
& m_2(\xi,\eta) \overset{def}{=} \frac{1}{2} \frac{|\xi|^{1/2}}{|\eta|^{1/2}|\xi-\eta|^{1/2}} \left( \eta \cdot (\xi-\eta) + |\eta||\xi-\eta| \right) \\
& m_3(\xi,\eta,\sigma) \overset{def}{=} - \frac{1}{2} |\xi|\left(|\xi-\eta-\sigma|^{3/2} + |\xi||\xi-\eta-\sigma|^{1/2} - 2 |\xi-\eta||\xi-\eta-\sigma|^{1/2} \right) \\
& m_4(\xi,\eta,\sigma) \overset{def}{=} |\xi|^{1/2} |\eta|^{1/2} \left( |\xi-\eta-\sigma|^{3/2} - |\xi-\eta| |\xi-\eta-\sigma|^{1/2} \right) .
\end{aligned}
\end{equation*}

Note that $m_1$ and $m_2$ are homogeneous of degree $3/2$ and that $m_3$ and $m_4$ are homogeneous of degree $5/2$.  
Also, note that since these multilinear forms are homogeneous,   we only need to estimate them on the sphere $|\xi|^2 +|\eta|^2= 1$, or $|\xi|^2 +|\eta|^2 + |\sigma|^2= 1$ and extend  all estimates by homogeneity.  Moreover the exact form of the above equation is not really important; thus  in order to focus on the information which is relevant to us,  \emph{we shall ignore from now on the distinction between $f_+$ and $f_-$ whenever this notation occurs}.

\subsection*{Examination of the quadratic symbols}

\label{quadsymb}

The symbols $m_1$ and $m_2$ have two important features, they vanish  when one of the Fourier coordinates ($\xi$, $\eta$ or $\xi-\eta$) is zero, and they are not smooth. These two facts are made more precise in the following lines.

Notice that the vanishing of $m_1$ and $m_2$ is crucial: as we will see, it corresponds to a null property on the time resonant set; on the other hand, the lack of smoothness is a hindrance, since it prevents one from applying the standard Coifman-Meyer theorem \cite{coifmanmeyer}.

We always use the convention that $\mathcal{A}$ stands for a smooth function in all its arguments and start with $m_1$:
\begin{itemize}
\item If $|\eta|<<|\xi|\sim 1$, $m_1(\xi,\eta) = |\eta|^{1/2} \mathcal{A}\left( |\eta|^{1/2},\frac{\eta}{|\eta|},\xi \right)$.
\item If $|\xi|<<|\eta|\sim 1$, $m_1(\xi,\eta) = |\xi| \mathcal{A}\left(\frac{\xi}{|\xi|},\eta\right)$.
\item If $|\xi-\eta| << |\xi|\sim 1$, $m_1(\xi,\eta) = |\xi-\eta|^2 \mathcal{A}\left(|\xi-\eta|^{1/2},\frac{\xi-\eta}{|\xi-\eta|},\xi \right)$.
\end{itemize}

Now $m_2$:
\begin{itemize}
\item If $|\eta|<<|\xi|\sim 1$, $m_2(\xi,\eta) = |\eta|^{1/2} \mathcal{A}\left( |\eta|^{1/2},\frac{\eta}{|\eta|},\xi \right)$.
\item If $|\xi|<<|\eta|\sim 1$, $m_2(\xi,\eta) = |\xi|^{5/2} \mathcal{A}\left(\eta,\xi\right)$.
\item If $|\xi-\eta| << |\xi|\sim 1$, $m_2(\xi,\eta) = |\xi-\eta|^{1/2} \mathcal{A}\left(|\xi-\eta|^{1/2},\frac{\xi-\eta}{|\xi-\eta|},\xi\right)$.
\end{itemize}

\subsection*{Space and time resonances} We define the bilinear and trilinear time resonant sets as
$$
\mathscr{T}_{\pm,\pm} = \{ \phi_{\pm,\pm} = 0 \}, \quad\mbox{ \phantom{respectively}}\quad\mathscr{T}_{\pm,\pm,\pm} = \{ \phi_{\pm,\pm,\pm} = 0 \}
$$
respectively.  We define the bilinear and trilinear space resonant sets as
$$
\mathscr{S}_{\pm,\pm} = \{\nabla_\eta \phi_{\pm,\pm} = 0 \} \mbox{,  \phantom{respectively}}\quad\mathscr{S}_{\pm,\pm,\pm} = \{\nabla_{\eta,\sigma} \phi_{\pm,\pm,\pm} = 0 \}
$$
respectively. 
The  bilinear and trilinear  space-time resonant sets are given by
$$
\mathscr{R}_{\pm,\pm} = \mathscr{S}_{\pm,\pm} \cap \mathscr{T}_{\pm,\pm} \quad\mbox{,  \phantom{respectively}}\quad\mathscr{R}_{\pm,\pm,\pm} = \mathscr{S}_{\pm,\pm,\pm} \cap \mathscr{T}_{\pm,\pm,\pm}.
$$
respectively.

\subsection*{Examination of the quadratic phases}

The phase $\phi_{++}=|\xi|^{1/2}+|\eta|^{1/2}+|\xi-\eta|^{1/2}$ is better behaved than the others, since it only vanishes at $(\xi,\eta)=(0,0)$.
Therefore we focus on the three other quadratic phases, namely $\phi_{--}$, $\phi_{-+}$ and $\phi_{+-}$. Up to multiplication by $-1$, and permutation of the three Fourier variables $\eta$, $\xi$, $\xi-\eta$, these three phases are the same. Let us consider  
$$
\phi_{--}(\xi,\eta) = |\xi|^{1/2} - |\eta|^{1/2} - |\xi-\eta|^{1/2}.
$$
A small computation shows that
\begin{equation*}
\mathscr{T}_{--} = \{ \eta = 0 \quad\mbox{or}\quad\xi-\eta = 0 \},
\end{equation*}
and  the vanishing of $\phi_{--}$ may be  described as follows
\begin{equation}
\label{vanishphi}
\mbox{if $|\eta|<<|\xi|\sim 1$,}\quad\phi_{--}(\xi,\eta) = |\eta|^{1/2}\mathcal{A} \left( |\eta|^{1/2},\frac{\eta}{|\eta|},\xi \right) \quad\mbox{with $\mathcal{A}(0,\cdot,\cdot) = -1$}
\end{equation}
(the case $\xi-\eta = 0$ being identical up to an obvious change of variables).
Finally, along the surface $\{ \xi = 0 \}$, $\phi_{--}$ does not vanish, but is not smooth:
\begin{equation}\label{eq:addaxis}
\mbox{if $|\xi|<<|\eta|\sim 1$,}\quad\phi_{--}(\xi,\eta) = \mathcal{A} \left( |\eta|^{1/2},\frac{\eta}{|\eta|},\xi \right) + |\xi|^{1/2} \mathcal{A}' \left( |\eta|^{1/2},\frac{\eta}{|\eta|},\xi \right)
\end{equation}
with $\mathcal{A}(0,\cdot,\cdot) = - 2 |\eta|^{1/2}$. 
As we shall see, it turns out that quadratic terms can be treated simply by a normal form transform, thus there is no need to investigate the quadratic space resonant set.

\subsection*{Examination of the cubic phases}

\label{examinationcubic}

Cubic phases fall into two categories: some have relatively few time resonances, we call them 'weakly resonant phases'; and some give a large space-time resonant set, we call them 'strongly resonant phases'.

\subsubsection*{The weakly resonant phases: $\phi_{+++}$, $\phi_{-++}$, $\phi_{+-+}$, $\phi_{++-}$, $\phi_{---}$}

As in the case of quadratic phases, the phase $\phi_{+++} = |\xi|^{1/2}+|\eta|^{1/2}+|\sigma|^{1/2} + |\xi-\eta-\sigma|^{1/2}$ is easily dealt with, since it only vanishes at $(\xi,\eta,\sigma)=(0,0,0)$. 

The four other quadratic phases, $\phi_{-++}$, $\phi_{+-+}$, $\phi_{++-}$, $\phi_{---}$, are identical, up to multiplication by $-1$, and permutation of the four Fourier variables $\xi$, $\eta$, $\sigma$, $\xi-\eta-\sigma$.
Let us therefore focus on 
$$
\phi_{---}(\xi,\eta,\sigma) = |\xi|^{1/2} - |\eta|^{1/2} - |\sigma|^{1/2} - |\xi-\eta-\sigma|^{1/2}.
$$
It is easily seen that
$$
\mathscr{T}_{---} = \{ \eta = \sigma = 0 \quad\mbox{or}\quad \sigma =\xi-\eta-\sigma= 0 \quad\mbox{or}\quad \eta =\xi-\eta-\sigma= 0 \}.
$$
The vanishing of $\phi_{---}$ on this set can be more precisely described as follows
\begin{equation}
\label{ZZZ}
\mbox{if $|\eta|,|\sigma| << |\xi|\sim 1$,} \quad\phi_{---} = - |\eta|^{1/2} - |\sigma|^{1/2} + \mathcal{A} (\xi,\eta,\sigma)[\eta+\sigma] .
\end{equation}
(recall $\mathcal{A} (\xi,\eta,\sigma)[\eta+\sigma]$ stands for a function smooth in its three first arguments and linear in the fourth).
The other cases $\sigma =\xi-\eta-\sigma= 0$ and $\eta =\xi-\eta-\sigma= 0$ are the same up to a change of variables.  Finally, $\phi_{---}$ is not smooth along the axes $\{\xi =0\}\bigcup \{\eta =0\}\bigcup\{\sigma =0\}\bigcup \{\xi-\eta-\sigma =0\}$.  In a neighborhood of these axes it can be written in a form similar to \eqref{eq:addaxis}.

As for the quadratic phases, it turns out that a normal form transform is sufficient to treat the weakly resonant cubic terms, therefore, we do not investigate the space resonant set.

\subsubsection*{The strongly resonant phases: $\phi_{--+}$, $\phi_{-+-}$, $\phi_{+--}$}

Once again, these three phases are identical up to a permutation of the Fourier variables $\eta$, $\sigma$ and $\xi-\eta-\sigma$. We therefore focus on one of them, namely
$$
\phi_{--+} = |\xi|^{1/2} - |\eta|^{1/2} - |\sigma|^{1/2} + |\xi-\eta-\sigma|^{1/2}.
$$
In this case, the time resonant $\mathscr{T}_{--+}$ set has dimension $5$, and $\phi_{--+}$ vanishes at order $1$ on it: this makes a normal form transform almost necessarily unbounded. We therefore need to turn to the space resonant set, which is
$$
\mathscr{S}_{--+} = \{ \xi = \eta = \sigma \},
$$
and this set is contained in $\mathscr{T}_{--+}$. In other words $\mathscr{S}_{--+} = \mathscr{R}_{--+}$ and it has dimension equal to $2$ and 
the phase  vanishes at order $1$, therefore  an argument based on an integration by parts in the $\eta$ and $\sigma$ seems doomed to fail.

The way out of this problem appears if one develops $\partial_\xi \phi_{--+}$, $\partial_\eta \phi_{--+}$ and $\partial_\sigma \phi_{--+}$ in a neighborhood of $\mathscr{S}_{--+}$. Namely, if $|\eta-\xi|,|\sigma-\xi| << |\xi|\sim 1$,
\begin{equation*}
\begin{split}
& \partial_\eta \phi_{--+} = \frac{1}2{|\xi|^{3/2}}\left( (\sigma-\xi)+\frac{3}{2}(\xi-\sigma)\cdot\frac{\xi}{|\xi|} \frac{\xi}{|\xi|} + \mathcal{A}(\xi,\eta,\sigma)[(\eta-\xi,\sigma-\xi),(\eta-\xi,\sigma-\xi)] \right) \\
& \partial_\sigma \phi_{--+} = \frac{1}{2|\xi|^{3/2}}\left( (\eta-\xi)+\frac{3}{2}(\xi-\eta)\cdot\frac{\xi}{|\xi|} \frac{\xi}{|\xi|} + \mathcal{A}(\xi,\eta,\sigma)[(\eta-\xi,\sigma-\xi),(\eta-\xi,\sigma-\xi)] \right) \\
& \partial_\xi \phi_{--+} = \frac{1}{2|\xi|^{3/2}}\Big( (2\xi-\eta-\sigma)+\frac{3}{2}(\eta+\sigma-2\xi)\cdot\frac{\xi}{|\xi|} \frac{\xi}{|\xi|} +  \mathcal{A}(\xi,\eta,\sigma)[(\eta-\xi,\sigma-\xi),(\eta-\xi,\sigma-\xi) \Big)
\end{split}
\end{equation*}
(in the above expressions, $\mathcal{A}$ is smooth in the three first arguments, and bilinear in the arguments between brackets).
The above expressions imply that $\partial_\xi \phi = - \partial_\eta \phi - \partial_\sigma \phi$ up to second order terms. Therefore, if $|\eta-\xi|,|\sigma-\xi| << |\xi|\sim 1$,
\begin{equation}
\label{cat}
\partial_\xi \phi_{--+} = \mathcal{A}(\xi,\eta,\sigma)[\partial_\eta \phi_{--+},\partial_\sigma \phi_{--+}].
\end{equation}
This identity will enable us to convert $\xi$ derivatives of $\phi$ into $\eta$ and $\sigma$ derivatives of $\phi$. The former occur when applying $\partial_\xi$ to our multilinear expression (which corresponds to the $x$ weight), and are problematic since they come with a $t$ or $s$ factor; the latter are harmless: an integration by parts gets rid of them. See Section~\ref{sectionstrong} for the details.

\section{Normal form transform}
\label{normalform}
Integrate by parts in $s$ the quadratic terms of~(\ref{eqfourier}) with the help of the formula 
$\frac{1}{i\phi_{\pm,\pm}} \partial_s e^{is \phi_{\pm,\pm}} = e^{is \phi_{\pm,\pm}}$. Doing so, the $\partial_s$ derivative will hit $\widehat{f}(s,\eta)$ or $\widehat{f}(s,\xi-\eta)$; then use~(\ref{eqfourier}) to substitute for $\partial_s \widehat{f}(s,\eta)$ or $\partial_s \widehat{f}(s,\xi-\eta)$. This gives
\begin{equation*}
\begin{aligned}
& \int_2^t \!\!\int e^{is\phi_{\tau_1,\tau_2}} m_l(\xi,\eta) \widehat{f}(s,\eta) \widehat{f}(s,\xi-\eta)\,d\eta ds \\
& = \left. \int e^{is\phi_{\tau_1,\tau_2}} \frac{m_l(\xi,\eta)}{i \phi_{\tau_1,\tau_2}} \widehat{f}(s,\eta) \widehat{f}(s,\xi-\eta)\,d\eta \right]_2^t - \int_2^t \!\!\int e^{is\phi_{\tau_1,\tau_2}} \frac{m_l(\xi,\eta)}{i \phi_{\tau_1,\tau_2}} \partial_s \left[ \widehat{f}(s,\eta) \widehat{f}(s,\xi-\eta) \right] \,d\eta ds \\
&= \left. \int e^{is\phi_{\tau_1,\tau_2}} \frac{m_l(\xi,\eta)}{i \phi_{\tau_1,\tau_2}} \widehat{f}(s,\eta) \widehat{f}(s,\xi-\eta)\,d\eta \right]_2^t  - \sum_{\widetilde{\tau_1},\widetilde{\tau_2}=\pm} \sum_{j=1,2} c_{\widetilde{\tau_1},\widetilde{\tau_2},j}\times\Big {\{}\\
&\;\;\;\;\;\; \int_2^t \!\!\int \!\!\!\int e^{is\phi_{\tau_1,\widetilde{\tau_1},\widetilde{\tau_2}}} \frac{m_l(\xi,\eta)}{i \phi_{\tau_1,\tau_2}}  m_j(\xi-\eta,\sigma) \widehat{f}(s,\eta) \widehat{f}(s,\sigma) \widehat{f}(s,\xi-\eta-\sigma) \,d\eta \,d\sigma \,ds\;  +\\
& \;\;\;\;\;\;
\int_2^t \!\!\int \!\!\!\int e^{is(|\xi|^{1/2}+\tau_1 |\eta|^{1/2})} \frac{m_l(\xi,\eta)}{i \phi_{\tau_1,\tau_2}} \widehat{f}(s,\eta) \widehat{Q}(\xi-\eta) \,d\eta \,ds\Big{\}}\, + \mbox{\{ symmetric terms \}} ,
\end{aligned}
\end{equation*}
where the ``symmetric terms'' come from the fact that $\partial_s$ may hit either $\widehat{f}(s,\xi-\eta)$ or $\widehat{f}(s,\eta)$, thus these symmetric terms look very much like the above ones with $\eta$ and $\xi-\eta$ exchanged. The term $Q$ corresponds to terms of order 3 and higher in $e^{-it\Lambda^{1/2}} \partial_t f$ i.e.
$$
\widehat{Q}(t,\xi) \overset{def}{=} \sum_{\tau_{1,2,3}= \pm} \sum_{j=3}^4 c_{i,\tau_1,\tau_2,\tau_3} \int\!\!\!\int m_j(\xi,\eta,\sigma) \widehat{u}(t,\eta) \widehat{u}(t,\sigma) \widehat{u}(t,\xi-\eta-\sigma)\,d\eta\, d\sigma + \widehat{R} (t,\xi ) .
$$
The normal form transform that we just performed gives quadratic (without time integration), 
cubic, and higher order terms. Also, cubic and higher order terms occurring in~(\ref{eqfourier}) need to be taken into account. In the end, we see that $f$ can be written as a sum of quadratic terms of the type
\begin{equation}
\label{butterfly1}
 \left.   \int e^{it\phi_{\pm,\pm}} \frac{m_l(\xi,\eta)}{i \phi_{\pm,\pm}} \widehat{f}(t,\eta)\widehat{f}(t,\xi-\eta)\,d\eta \right|_2^t \quad\mbox{with $l=1,2$}
\end{equation}
(here it is understood that the symbol $\pm$ stands each time for either $+$ or $-$),  cubic terms of the type
\begin{subequations}
\begin{gather}
\label{butterfly2}
\int_2^t \!\!\int \!\!\!\int e^{is\phi_{\pm,\pm,\pm}} \frac{m_l(\xi,\eta)}{i \phi_{\pm,\pm}}  m_j(\xi-\eta,\sigma) \widehat{f}(s,\eta) \widehat{f}(s,\sigma) \widehat{f}(s,\xi-\eta-\sigma) \,d\eta \,d\sigma \,ds \quad\mbox{with $l,j=1,2$}\\
\label{butterfly3}
\int_2^t \!\!\int \!\!\!\int e^{is\phi_{\pm,\pm,\pm}} m_l(\xi,\eta,\sigma) \widehat{f}(s,\eta) \widehat{f}(s,\sigma) \widehat{f}(s,\xi-\eta-\sigma)\,d\eta\,d\sigma\,ds \quad\mbox{with $l=3,4$},
 \end{gather}
\end{subequations}
and terms of order 4 of the type
\begin{subequations}
\begin{gather}
\label{butterfly4}
\int_2^t e^{is|\xi|^{1/2}} \widehat{R}(s,\xi)\,ds \\
\label{butterfly5}
\int_2^t \!\!\int \!\!\!\int e^{is(|\xi|^{1/2}+\tau_1 |\eta|^{1/2})} \frac{m_l(\xi,\eta)}{i \phi_{\tau_1,\tau_2}} \widehat{f}(s,\eta) \widehat{Q}(s,\xi-\eta) \,d\eta \,ds
\end{gather}
\end{subequations}

\section{Estimates for the quadratic terms}

\label{quadratic}

In this section, we derive the desired estimates in $L^\infty$ and $L^2(x^2dx)$ for terms of the form~(\ref{butterfly1}). We shall denote a generic term of this type by
\begin{equation}
\label{defg1}
\widehat{g_1}(t,\xi) \overset{def}{=} \int e^{it\phi} \mu(\xi,\eta) 
\widehat{f}(t,\eta) \widehat{f}(t,\xi-\eta)\,d\eta ,
\end{equation}
where
$$
\phi  \overset{def}{=}  \phi_{\pm,\pm}\quad\mbox{and}\quad\mu(\xi,\eta) \overset{def}{=} \frac{m_l(\xi,\eta)}{i \phi_{\pm,\pm}} ,
$$ 
with the index $l$ equal to $1$ or $2$, and $\pm$ either $+$ or $-$.

Of course, we adopt this lighter notation because the precise value of $\pm$ or $l$ will not affect the argument which follows.

\subsection*{Preliminary observations and reductions}

\label{preliminary}

Let us first consider the symbol $\mu(\xi,\eta)$ occurring in the definition of $g_1$: regardless of the precise indices $l$ and $\pm,\pm$, it is
\begin{itemize}
\item Homogeneous of degree 1.
\item Smooth if none of $\eta$, $\xi$, $\xi-\eta$ vanish.
\item Of the form $\mu(\xi,\eta) = \mathcal{A}\left( |\eta|^{1/2},\frac{\eta}{|\eta|},\xi \right)$ if $|\eta| << |\xi|\sim 1$.
\item Of the form $\mu(\xi,\eta) = \mathcal{A}\left( |\xi-\eta|^{1/2},\frac{\xi-\eta}{|\xi-\eta|},\xi \right)$ if $|\xi-\eta| << |\xi|\sim 1$.
\item Of the form $\mu(\xi,\eta) = |\xi|^{1/2} \mathcal{A}\left( |\xi|^{1/2},\frac{\xi}{|\xi|},\eta \right)$ if $|\xi| << |\eta|\sim 1$,
\end{itemize}
the last three points follow from the developments given in Section~\ref{quadsymb}. Thus
we conclude that $\mu$ belongs to the class $\mathcal{B}_1$, see appendix~\ref{appendquad} for the definition.

\bigskip

Next, define a cut off function $\chi(\xi,\eta)$ which is valued in $[0,1]$, homogeneous of degree $0$, smooth outside of $(0,0)$, and such that $\chi(\xi,\eta) = 0$ in a neighborhood of $\{ \eta = 0 \}$, and $\chi(\xi,\eta) = 1$ in a neighborhood of $\{ \xi-\eta = 0 \}$ on the sphere. Then one can split $g_1$ as follows
\begin{equation*}
\begin{split}
\widehat{g_1}(\xi) = & \int e^{it\phi} \chi(\xi,\eta) \mu(\xi,\eta) 
\widehat{f}(t,\eta) \widehat{f}(t,\xi-\eta)\,d\eta \\
& + \int e^{it\phi} \left[ 1 - \chi(\xi,\eta) \right] \mu(\xi,\eta) 
\widehat{f}(t,\eta) \widehat{f}(t,\xi-\eta)\,d\eta.
\end{split}
\end{equation*}
By symmetry it suffices to consider the first term of the above right-hand side, which corresponds to a region where  $|\eta| \gtrsim |\xi|$,  $ |\xi-\eta|$ (the interest of that condition is that $\xi$ derivatives always hit $\widehat{f}(\xi-\eta)$, which corresponds to low frequencies). Thus in the following, we shall consider that
$$
\widehat{g_1}(\xi) = \int e^{it\phi} \chi(\xi,\eta) \mu(\xi,\eta) 
\widehat{f}(t,\eta) \widehat{f}(t,\xi-\eta)\,d\eta
$$
which means in physical space, using the notation introduced in appendix~\ref{appendixpp},
$$
g_1 = e^{ i t \Lambda^{1/2}} B_{\chi \mu} (u,u).
$$
Actually, depending on the $\pm,\pm$ appearing in $\phi_{\pm,\pm}$ the $u$ can be $u$ or $\bar u$; but as indicated in Section~\ref{decay}, we ignore for the sake of simplicity in the notations this distinction.  
Also, it is important to notice here that, since $\mu \in \mathcal{B}_1$ and $\chi \in \widetilde{\mathcal{B}}_0$, $\mu \chi \in \widetilde{\mathcal{B}}_1$.

\subsection*{Quantities controlled by the $X$ norm} 
Before estimating the term $g_1$ in $X$, we give some general estimates that will be useful in the whole proof. 
Interpolating between the different components of the $X$ norm gives the following lemma.  Let  $P_j$, $P_{>j}$, $P_{\leq j}$ denote the Littlewood-Paley projections defined in Section~\ref{appendixlin}
\begin{lemma}\label{pbound}
(i) If $2\leq p \leq \infty$ and $k<N + \frac{2}{p} - 1$, $\displaystyle \left\| \nabla^k P_{>j} u \right\|_p \lesssim 2^{j\left(-N+k+1-\frac{2}{p}\right)} t^\delta \|u\|_X$. 

\medskip

(ii) If $2\leq p \leq \infty$ and $k<N + \frac{2}{p} - 1$, $\displaystyle \|\nabla^k u\|_p \lesssim t^{-1 +\frac{2}{p} + \frac{k}{N+\frac{2}{p}-1}\left[ \delta - \frac{2}{p} + 1 \right]} \|u\|_X$.

\medskip

(iii) If $1< p \le2$, $\|f\|_p \lesssim t^{\left( \frac{2}{p}-1 \right) \delta} \|u\|_X$.
\end{lemma}

\begin{proof} The proof of $(i)$ is a standard application of~(\ref{LPderivative}) and (\ref{lemmadeltaj}):
\begin{equation*}
\begin{split}
\left\| \nabla^k P_{>j} u \right\|_p & \leq \sum_{\ell>j} 2^{\ell k}\left\| P_\ell u \right\|_p \lesssim \sum_{\ell>j} 2^{\ell(-N+k+1-\frac{2}{p})} \|u\|_{H^N} \\
& \lesssim 2^{j(-N+k+1-\frac{2}{p})} \|u\|_{H^N} \lesssim 2^{j(-N+k+1-\frac{2}{p})} t^\delta \|u\|_X .
\end{split}
\end{equation*}
For $(ii)$, we estimate separately low and high frequencies:
\begin{equation*}
\begin{split} 
\|\nabla^k u\|_p & \leq \|P_{\leq j} \nabla^k u \|_p + \|P_{>j} \nabla^k u \|_p \lesssim 2^{jk} \|u\|_p + 2^{j(-N+k+1-\frac{2}{p})} t^\delta \|u\|_X \\
& \lesssim 2^{jk} t^{-1+\frac{2}{p}} \|u\|_X + 2^{j(-N+k+1-\frac{2}{p})} t^\delta \|u\|_X .
\end{split}
\end{equation*}
Optimizing the above inequality over $j$ gives the desired conclusion. 
Finally, $(iii)$ follows from interpolating $L^p$ spaces between weighted $L^2$ spaces.
\end{proof}

\subsection*{Bound for $\nabla^k e^{- it\Lambda^{1/2}} g_1$ in $L^\infty$ with $0 \leq k \leq 4$}
The idea for this estimate, as for many estimates which follow, is to use Sobolev embedding in order to make sure that Lebesgue indices are finite when applying Theorem~\ref{bilinearbound}.

Using successively Sobolev embedding, Theorem~\ref{bilinearbound} and Lemma~\ref{pbound}, one gets
\begin{equation*}
\begin{split}
\left\| \nabla^k e^{- it\Lambda^{1/2}} g_1 \right\|_\infty & \lesssim \left\| \nabla^k e^{- it\Lambda^{1/2}} g_1 \right\|_{W^{1,8}} = \left\| \nabla^{k} B_{\chi \mu}( u , u) \right\|_{W^{1,8}}  \\
& \lesssim \| \Lambda^{k+1} u \|_{W^{1,16}}  \|u\|_{16}  \lesssim \frac{1}{t} \|u\|^2_X.
\end{split}
\end{equation*}

\subsection*{Bound for $ x g_1$ in $L^2$}
In Fourier space, $x g_1$ reads
\begin{subequations}
\begin{align}
\mathcal{F}(x g_1 ) = & \partial_\xi \left[ \int e^{it\phi} \chi(\xi,\eta) \mu (\eta,\xi)  \widehat{f}(t,\eta) \widehat{f}(t,\xi-\eta)\,d\eta \right] \\
\label{beetle1} =& \int e^{it\phi} \chi(\xi,\eta) \mu (\eta,\xi) \widehat{f}(t,\eta) \partial_\xi \widehat{f}(t,\xi-\eta)\,d\eta \\
\label{beetle2} & + \int e^{it\phi} \partial_\xi\left[ \chi(\xi,\eta) \mu (\eta,\xi) \right] \widehat{f}(t,\eta) \widehat{f}(t,\xi-\eta)\,d\eta \\
\label{beetle3} & + \int e^{it\phi} t \partial_\xi \phi (\xi,\eta)\chi(\xi,\eta) \mu (\eta,\xi) \widehat{f}(t,\eta) \widehat{f}(t,\xi-\eta)\,d\eta .
\end{align}
\end{subequations}

\subsubsection*{Bound for~(\ref{beetle1}) in $L^2$} 

Using successively Sobolev embedding, Theorem~\ref{bilinearbound} and Lemma~\ref{pbound}, one gets
\begin{equation*}
\begin{split}
\left\| (\ref{beetle1}) \right\|_2 & = \left\| B_{\chi \mu} \left(  u , e^{\pm it{\Lambda^{1/2}}}x f \right) \right\|_2  \lesssim \left\| B_{\chi \mu} \left(  u , e^{\pm it{\Lambda^{1/2}}}x f \right) \right\|_{W^{1,\frac 43}}\\
& \lesssim \left\| \Lambda u \right\|_{W^{1,4}}\left\| xf \right\|_2  \lesssim  \|u\|^2_X .
\end{split}
\end{equation*}

\subsubsection*{Bound for~(\ref{beetle2}) in $L^2$}
Let us take a closer look at $\chi(\xi,\eta) \mu (\eta,\xi)$. This is a symbol in $\widetilde{\mathcal{B}}_1$, which furthermore vanishes at order $\frac{1}{2}$ in $\xi=0$. Therefore, by Lemma~\ref{derivbs}.
$$
\partial_\xi \left[ \chi(\xi,\eta) \mu (\eta,\xi) \right] = \mu_1 + \frac{1}{|\xi|^{1/2}} \mu_2 + \frac{1}{|\xi-\eta|} \mu_3\quad\mbox{with $(\mu_1,\mu_2,\mu_3) \in \widetilde{\mathcal{B}}_0 \times \widetilde{\mathcal{B}}_{1/2} \times \widetilde{\mathcal{B}}_1$}
$$
Thus
\begin{subequations}
\begin{align}
\label{ant1}
(\ref{beetle2}) = & \int e^{it\phi} \mu_1 (\eta,\xi) \widehat{f}(t,\eta) \widehat{f}(t,\xi-\eta)\,d\eta \\
\label{ant2}
& + \int e^{it\phi} \frac{1}{|\xi|^{1/2}} \mu_2 (\eta,\xi) \widehat{f}(t,\eta) \widehat{f}(t,\xi-\eta)\,d\eta \\
\label{ant3}
& + \int e^{it\phi} \frac{1}{|\xi-\eta|} \mu_3 (\eta,\xi) \widehat{f}(t,\eta) \widehat{f}(t,\xi-\eta)\,d\eta .
\end{align}
\end{subequations}
The term~(\ref{ant1}) is easily estimated, so we skip it, and consider next the term~(\ref{ant2}).
\begin{equation*}
\left\|(\ref{ant2})\right\|_2 = \left\|\frac{1}{\Lambda^{1/2}} B_{\mu_2} (u,u) \right\|_2 \lesssim \left\| B_{\mu_2} (u,u) \right\|_{4/3} \lesssim \left\|\Lambda^{1/2}u\right\|_4 \left\|u\right\|_2 \lesssim   \|u\|^2_X.
\end{equation*}
Finally, we need to estimate the term~(\ref{ant3}): by Theorem~\ref{bilinearbound} and Lemma~\ref{linearbound}
\begin{equation*}
\begin{split}
\left\|(\ref{ant3})\right\|_2 & = \left\| B_{\mu_3} \left( u , e^{it{\Lambda^{1/2}}} \frac{1}{\Lambda} f \right) \right\|_2  \lesssim \left\| \Lambda u \right\|_4 \left\| e^{it{\Lambda^{1/2}}} \frac{1}{\Lambda} f \right\|_4 \\
& \lesssim \left\| \Lambda u \right\|_4 \left\| f \right\|_{4/3} \lesssim   \|u\|^2_X .
\end{split}
\end{equation*}

\subsubsection*{Bound for~(\ref{beetle3}) in $L^2$}
Since $\partial_\xi \phi = \frac{1}{2} \frac{\xi}{|\xi|^{3/2}} \pm \frac{1}{2} \frac{\xi-\eta}{|\xi-\eta|^{3/2}}$, and keeping in mind that $\mu$ vanishes at order $1/2$ in $\xi$, the symbol appearing in~(\ref{beetle3}) can be written
$$
\partial_\xi \phi (\xi,\eta)\chi(\xi,\eta) \mu (\eta,\xi) \overset{def}{=} \mu_1(\xi,\eta) + \frac{1}{|\xi-\eta|^{1/2}} \mu_2(\xi,\eta) \quad\mbox{with $(\mu_1,\mu_2) \in \widetilde{\mathcal{B}}_{1/2} \times \widetilde{\mathcal{B}}_1$}.
$$
We simply show how to estimate the term associated to the symbol $\frac{1}{|\xi-\eta|^{1/2}} \mu_2(\xi,\eta)$, the symbol $\mu_1(\xi,\eta)$ being easier to treat.
Using successively Theorem~\ref{bilinearbound}, Lemma~\ref{linearbound} and Lemma~\ref{pbound}, one gets
\begin{equation*}
\begin{split}
\left\| \mathcal{F}  \int e^{it\phi} t \frac{1}{|\xi-\eta|^{1/2}} \mu_2 \widehat{f}(t,\eta) \widehat{f}(t,\xi-\eta)\,d\eta \right\|_2 & = t \left\|B_{\mu_2} \left( u , e^{\pm it{\Lambda^{1/2}}} \frac{1}{\Lambda^{1/2}} f \right) \right\|_2 \\
& \lesssim t \left\| \Lambda u \right\|_{\frac{1}{\delta_0}} \left\| e^{\pm it{\Lambda^{1/2}}} \frac{1}{\Lambda^{1/2}} f \right\|_{\frac{2}{1-2\delta_0}} \\
& \lesssim t \left\| \Lambda u \right\|_{\frac{1}{\delta_0}} \left\| f \right\|_{\frac{4}{3-4\delta_0}} \\
& \lesssim t t^{-1 + 2 \delta_0 + \frac{1}{N +2\delta_0- 1}\left[ 1 - 2\delta_0 + \delta \right]} t^{(\frac{1}{2}-2 \delta_0)\delta}  \|u\|^2_X\\
& \lesssim t^\delta  \|u\|^2_X .
\end{split}
\end{equation*}
where  $\delta_0 >0$ denotes a small constant. Notice that the last inequality holds since $N$ and $\delta_0$ have been picked  to be large and small enough respectively .


\section{Estimates for the weakly resonant cubic terms}

\label{sectionweak}

In this section, we derive the desired estimates in $L^\infty$ and $L^2(x^2dx)$ for terms of the form either~(\ref{butterfly2}) or~(\ref{butterfly3}) corresponding to weakly resonant phases as defined in Section~\ref{examinationcubic}. We shall denote a generic term of this type by
\begin{equation}
\label{defg2}
\widehat{g_2}(\xi,t) \overset{def}{=} \int_2^t \!\!\int \!\!\!\int e^{is\phi} \mu(\xi,\eta,\sigma) 
\widehat{f}(s,\sigma) \widehat{f}(s,\eta) \widehat{f}(s,\xi-\eta-\sigma)\,d\eta\,d\sigma\,ds ,
\end{equation}
where
$$
\phi = \phi_{+++} \quad\mbox{or}\quad\phi_{-++}\quad\mbox{or}\quad\phi_{+-+} \quad\mbox{or}\quad\phi_{++-}\quad\mbox{or}\quad \phi_{---}
$$
and
$$
\mu(\xi,\eta,\sigma) = m_i(\xi,\eta,\sigma) \quad\mbox{or}\quad \frac{m_k(\xi,\eta)}{i \phi_{\pm,\pm}(\xi,\eta)}m_j(\xi-\eta,\sigma)\quad\mbox{or}\quad \frac{m_k(\xi,\xi-\eta)}{i \phi_{\pm,\pm}(\xi,\xi-\eta)}m_j(\xi-\eta,\sigma) ,
$$
with the indices $i$ equal to $3$ or $4$, $j,k$ equal to $1$ or $2$, and $\pm$ equal to $+$ or $-$.

Of course, we adopt this lighter notation because the precise form of $\phi$ or $\mu$ will not affect the argument which follows.

\subsection*{Preliminary observations and reductions}

\label{POR}

We begin with the symbol $\mu(\xi,\eta,\sigma)$ occurring in the definition of $g_2$: regardless of its precise form.   From Section~\ref{quadsymb} and the observations at the beginning of Section~\ref{preliminary} we have

\begin{itemize}
\item It is homogeneous of degree $5/2$ in $(\xi,\eta,\sigma)$.
\item It is smooth except on $\{\xi = 0 \} \cup \{ \eta = 0 \} \cup \{ \sigma = 0\} \cup \{ \xi-\eta-\sigma = 0\} \cup \{ \xi-\eta = 0\}$.
\item It might have a singularity for $\xi-\eta =0$ of type, in the worst possible case, $|\xi-\eta|^{1/2}$.
\item It vanishes to order (at least) $1/2$ in $\xi$.
\item It belongs to the class $\mathcal{T}_{5/2}$ (see appendix~\ref{appendcub} for the definition).
\end{itemize}

\medskip

Next, define cut-off functions $\chi_1$, $\chi_2$, $\chi_3$ such that
\begin{itemize}
\item $\chi_1$, $\chi_2$, $\chi_3$ are valued in $[0,1]$ and $\chi_1+\chi_2+ \chi_3 = 1 $.
\item $\chi_1$, $\chi_2$, $\chi_3$ are homogeneous of degree $0$ and smooth outside of $(0,0,0)$.
\item $\chi_1(\xi,\eta,\sigma) = 0$ in a neighborhood of $\{ \sigma = 0 \}$.
\item $\chi_2(\xi,\eta,\sigma) = 0$ in a neighborhood of $\{ \eta = 0 \}$.
\item $\chi_3(\xi,\eta,\sigma) = 0$ in a neighborhood of $\{ \xi-\eta-\sigma = 0 \}$.
\end{itemize}
Then one can split $g_2$ as follows
\begin{equation*}
\begin{split}
\widehat{g_2}(\xi) = \sum_{k=1}^3 \int_2^t \!\!\int \!\!\!\int e^{is\phi} \chi_k(\xi,\eta,\sigma) \mu(\xi,\eta,\sigma) 
\widehat{f}(s,\sigma) \widehat{f}(s,\eta) \widehat{f}(s,\xi-\eta-\sigma)\,d\eta\,d\sigma\,ds .
\end{split}
\end{equation*}
By symmetry it suffices to consider the first summand ($k=1$) of the above right-hand side, which corresponds to the region $|\sigma| \gtrsim |\xi|,|\eta|,|\xi-\eta-\sigma|$ (the interest of that condition is that $\xi$ derivatives always hit $\widehat{f}(\xi-\eta-\sigma)$, which corresponds thus to low frequencies). Thus in the following, we shall consider that
$$
\widehat{g_2}(\xi) = \int_2^t \!\!\int \!\!\!\int e^{is\phi} \chi_1(\xi,\eta,\sigma) \mu(\xi,\eta,\sigma)
\widehat{f}(s,\eta) \widehat{f}(s,\sigma) \widehat{f}(s,\xi-\eta-\sigma)\,d\eta\,d\sigma\,ds .
$$
Since $\chi_1 \in \widetilde{\mathcal{T}}_0$ and $\mu \in \mathcal{T}_{5/2}$, we will constantly use in the estimates which follow that $\chi_1 \mu \in \widetilde{\mathcal{T}}_{5/2}$.

\bigskip

Finally, we will need the following lemma
\begin{lemma}
If $\phi$ is either $\phi_{+++}$, $\phi_{-++}$, $\phi_{+-+}$, $\phi_{++-}$, or $\phi_{---}$, then $\frac{1}{\phi}$ can be written as
\begin{equation}
\label{onephi}
\frac{1}{\phi} = \nu_0 + \frac{1}{|\xi|^{1/2}} \nu_1 + \frac{1}{|\eta|^{1/2}} \nu_2 + \frac{1}{|\sigma|^{1/2}} \nu_3 +  \frac{1}{|\xi-\eta-\sigma|^{1/2}} \nu_4,
\end{equation}
where $(\nu_0,\nu_1,\nu_2,\nu_3,\nu_4) \in \mathcal{T}_{-1/2} \times \mathcal{T}_{0} \times \mathcal{T}_{0} \times \mathcal{T}_{0} \times \mathcal{T}_{0}$.
\end{lemma}
\begin{proof}
Relies on~(\ref{ZZZ}).
\end{proof}
\begin{remark}
Depending on which one of $\phi_{+++}$, $\phi_{-++}$, $\phi_{+-+}$, $\phi_{++-}$, or $\phi_{---}$ is considered, one  of the four symbols $(\nu_1,\nu_2,\nu_3,\nu_4)$ can be taken equal to 0.
\end{remark}

\subsection*{Control of $e^{-it\Lambda^{1/2}} \partial_t f$}

The  normal form transformation, i.e., integrating by part with respect to time, introduces terms of the type $e^{-it\Lambda^{1/2}} \partial_t f$. This explains the importance of the following lemma.

\begin{lemma}
\label{lemmadt}
If $2\leq p < \infty$ and $0\leq k \leq 10$,
$$
\left\| \nabla^k e^{-it\Lambda^{1/2}} \partial_t f \right\|_p \lesssim t^{-2+\frac{2}{p} + \frac{k+\frac{5}{2}}{N+\frac{1}{p}-1}(1-\frac{1}{p}+\delta)} \|u\|_X^2 .
$$
\end{lemma}

\begin{proof}
Differentiating~(\ref{eqfourier}) with respect to $t$ and applying $e^{-it\Lambda^{1/2}}$ gives that $e^{-it\Lambda^{1/2}} \partial_t f$ is a sum of terms of the type
\begin{subequations}
\begin{align}
\label{fox1}
& \int  m_j(\xi,\eta) \widehat{u}(t,\eta) \widehat{u}(t,\xi-\eta)\,d\eta \quad j =1,2\\
\label{fox2}
& \int \!\!\!\int m_j(\xi,\eta,\sigma) \widehat{u}(t,\eta) \widehat{u}(t,\sigma) \widehat{u}(t,\xi-\eta-\sigma)\,d\eta\,d\sigma  \quad j =3,4.
\end{align}
\end{subequations}
plus remainder terms. With the usual justifications,
\begin{equation*}
\begin{split}
& \left\|\nabla^k (\ref{fox1})\right\|_p \lesssim \left\| \Lambda^{k+3/2} u \right\|_{2p} \left\| u \right\|_{2p} \lesssim t^{-2+\frac{2}{p} + \frac{k+\frac{3}{2}}{N+\frac{1}{p}-1}(1-\frac{1}{p}+\delta)} \|u\|_X^2 \\
& \left\|\nabla^k (\ref{fox2})\right\|_p \lesssim \left\| \Lambda^{k+5/2} u \right\|_{3p} \left\| u \right\|_{3p} \left\| u \right\|_{3p}  \lesssim t^{-3+\frac{2}{p} + \frac{k+\frac{5}{2}}{N+\frac{2}{3p}-1}(1-\frac{2}{3p}+\delta)} \| u \|_X^3 , \\
\end{split}
\end{equation*}
and the estimate follows by using Proposition~\ref{propR} for the remainder terms.
\end{proof}

\subsection*{Bound for $\nabla^k e^{-it\Lambda^{1/2}} g_2$ in $L^\infty$, for $0\leq k \leq 4$}

Integrating by parts in $s$ in~(\ref{defg2}) using the relation $\frac{1}{i\phi} \partial_s e^{i s \phi} = e^{is\phi}$ gives
\begin{equation*}
\begin{split}
\widehat{g_2}(t,\xi) = &\left. \int \!\!\!\int e^{is\phi} \frac{1}{i\phi} \chi_1(\xi,\eta,\sigma) \mu(\xi,\eta,\sigma) 
\widehat{f}(s,\sigma) \widehat{f}(s,\eta) \widehat{f}(s,\xi-\eta-\sigma)\,d\eta\,d\sigma\right]_2^t \\
& - \int_2^t \!\!\int \!\!\!\int e^{it\phi} \frac{1}{i\phi} \chi_1(\xi,\eta,\sigma) \mu(\xi,\eta,\sigma) 
\partial_s \left[ \widehat{f}(s,\sigma) \widehat{f}(s,\eta) \widehat{f}(s,\xi-\eta-\sigma) \right] \,d\eta\,d\sigma\,ds .
\end{split}
\end{equation*}
The boundary term at $s=2$ is determined by the initial  data and is easy to estimate.
Next, replace $\frac{1}{i \phi}$ by its decomposition on the right-hand side of~(\ref{onephi}). The term $\nu_0$ is harmless of course, and so is the term $\frac{1}{|\xi|^{1/2}} \nu_1$, due to the vanishing of $\mu$ in $\xi$ at order $1/2$. The three remaining terms can be treated in essentially the same way, thus for the sake of illustration we only retain the term $\frac{1}{|\eta|^{1/2}} \nu_2(\xi,\eta,\sigma)$. Furthermore, if one distributes the $\partial_s$ derivative occurring in the second summand of the above right-hand side, it suffices 
to consider one of the three resulting terms; we choose to consider the case where $\partial_s$ hits $\widehat{f}(s,\sigma)$.
Thus $\mathcal{F} \nabla^k e^{-it\Lambda^{1/2}} g_2$ can be written as a sum of terms of the type
\begin{subequations}
\begin{align}
\label{bee1}
& \int \!\!\!\int \xi^k \frac{1}{|\eta|^{1/2}} \nu_2 \chi_1 \mu \widehat{u}(t,\sigma) \widehat{u}(t,\eta) \widehat{u}(t,\xi-\eta-\sigma)\,d\eta\,d\sigma \\
\label{bee2}
& \int_2^t \!\!\int \!\!\!\int \xi^k  e^{-i(t-s) |\xi|^{1/2}} \frac{1}{|\eta|^{1/2}}\nu_2 \chi_1 \mu e^{\pm is|\sigma|^{1/2}} \partial_s \widehat{f}(s,\sigma) \widehat{u}(s,\eta) \widehat{u}(s,\xi-\eta-\sigma) \,d\eta\,d\sigma\,ds .
\end{align}
\end{subequations}
Using successively Sobolev embedding, Theorem~\ref{trilinearbound} and Lemma~\ref{linearbound}, we get
\begin{equation*}
\begin{split}
\left\|\nabla^k \mathcal{F}^{-1} (\ref{bee1}) \right\|_\infty & = \left\|  \nabla^k B_{\nu_2 \mu \chi_1} \left( u,\frac{1}{\Lambda^{1/2}}u,u \right) \right\|_\infty \lesssim  \left\|  \nabla^k B_{\nu_2 \mu \chi_1} \left( u,\frac{1}{\Lambda^{1/2}}u,u \right) \right\|_{W^{1,8}}  \\
& \lesssim \left\|\Lambda^{5/2+k} u\right\|_{W^{1,24}} \left\|\frac{1}{\Lambda^{1/2}} u\right\|_{24} \left\|u\right\|_{24} 
\lesssim \left\|\Lambda^{5/2+k} u\right\|_{W^{1,24}} \left\| u\right\|_{24/7} \left\|u\right\|_{24}  \lesssim \frac{1}{t} \|u\|_X^3 .
\end{split}
\end{equation*}
In order to estimate (\ref{bee2}), we split it into two pieces $\int_2^{t-1} + \int_{t-1}^t $. 
The first summand is estimated using the dispersive estimate of Lemma~\ref{linearbound} and point $(ii)$ of Theorem~\ref{trilinearbound}:
\begin{equation*}
\begin{split}
& \left\|\mathcal{F}^{-1} \int_2^{t-1}\int \!\!\!\int \xi^k  e^{-i(t-s) |\xi|^{1/2}} \frac{1}{|\eta|^{1/2}}\nu_2 \chi_1 \mu e^{\pm is|\sigma|^{1/2}} \partial_s \widehat{f}(s,\sigma) \widehat{u}(s,\eta) \widehat{u}(s,\xi-\eta-\sigma) \,d\eta\,d\sigma \,ds \right\|_\infty \\
&\lesssim \int_2^{t-1} \frac{ds}{t-s} \left\| \nabla^k  B_{\nu_2 \mu \chi_1} \left(e^{\pm is\Lambda^{1/2}} \partial_s f ,  \frac{1}{\Lambda^{\frac12}}u , u \right) \right\|_{B^2_{1,\infty}} \\
&\lesssim \int_2^{t-1} \frac{ds}{t-s}  \left\|\Lambda^{k+5/2} \partial_s f \right\|_{H^2} \left\|\frac1{\Lambda^{\frac12}}u \right\|_{4} \left\|u\right\|_{4}  \lesssim
 \int_2^{t-1} \frac{1}{t-s}  \left\|\Lambda^{k+5/2} \partial_s f \right\|_{H^2} \left\| u \right\|_{2} \left\|u\right\|_{4} ds
\lesssim \frac{1}{t} \|u\|_X^3 .
\end{split}
\end{equation*}
The second summand 
is estimated via Sobolev embedding:
\begin{equation}
\label{walrus}
\begin{split}
& \left\|\mathcal{F}^{-1} \int_{t-1}^t\int \!\!\!\int \xi^k  e^{-i(t-s) |\xi|^{1/2}} \frac{1}{|\eta|^{1/2}}\nu_2 \chi_1 \mu e^{\pm is|\sigma|^{1/2}} \partial_s \widehat{f}(s,\sigma) \widehat{u}(s,\eta) \widehat{u}(s,\xi-\eta-\sigma) \,d\eta\,d\sigma \,ds \right\|_\infty \\
&\qquad\lesssim \int_{t-1}^t \left\| \nabla^k B_{\nu_2 \mu \chi_1} \left(e^{\pm is\Lambda^{1/2}} \partial_s f , \frac{1}{\Lambda^{1/2}} u , u \right) \right\|_{H^2}\, ds\\
&\qquad\lesssim \int_{t-1}^t \left\|\Lambda^{k+2} e^{\pm is\Lambda^{1/2}} \partial_s f\right\|_{W^{2,4}} \left\| \frac{1}{\Lambda^{1/2}} u \right\|_8 \left\| u \right\|_8
\,ds \\
&\qquad\lesssim \int_{t-1}^t \left\|\Lambda^{k+2} e^{\pm is\Lambda^{1/2}} \partial_s f\right\|_{W^{2,4}} \left\| u \right\|_{8/3} \left\| u \right\|_8 
\,ds  \lesssim \frac{1}{t} \|u\|_X^3 .
\end{split}
\end{equation}

\subsection*{Bound for $x g_2$ in $L^2$.}

In Fourier space, $x g_2$ reads
\begin{subequations}
\begin{align}
\mathcal{F} (x g_2) = & \partial_\xi \left[ \int_2^t \!\!\int \!\!\!\int e^{is\phi} \chi_1 \mu 
\widehat{f}(s,\sigma) \widehat{f}(s,\eta) \widehat{f}(s,\xi-\eta-\sigma)\,d\eta\,d\sigma\,ds \right] \\
\label{bull2}
= & \int_2^t \!\!\int \!\!\!\int e^{is\phi} \chi_1 \mu 
\widehat{f}(s,\sigma) \widehat{f}(s,\eta) \partial_\xi \widehat{f}(s,\xi-\eta-\sigma)\,d\eta\,d\sigma\,ds \\
\label{bull3}
& + \int_2^t \!\!\int \!\!\!\int e^{is\phi} \partial_\xi (\chi_1 \mu) 
\widehat{f}(s,\sigma) \widehat{f}(s,\eta)  \widehat{f}(s,\xi-\eta-\sigma)\,d\eta\,d\sigma\,ds \\
\label{bull4}
& + \int_2^t \!\!\int \!\!\!\int e^{is\phi} s (\partial_\xi \phi) \chi_1 \mu
\widehat{f}(s,\sigma) \widehat{f}(s,\eta)  \widehat{f}(s,\xi-\eta-\sigma)\,d\eta\,d\sigma\,ds .
\end{align}
\end{subequations}

\subsubsection*{Bound for~(\ref{bull2}) in $L^2$} Using successively Sobolev embedding, Theorem~\ref{trilinearbound} and Lemma~\ref{pbound}, one gets
\begin{equation*}
\begin{split}
\left\| (\ref{bull2}) \right\|_2 & = \left\|\int_2^t  \nabla^k B_{\chi_1 \mu} \left(u,  u , e^{\pm it{\Lambda^{1/2}}}x f \right)\,ds \right\|_2  \lesssim \int_2^t \left\| \nabla^k  B_{\chi_1 \mu} \left( u, u , e^{\pm it{\Lambda^{1/2}}}x f \right) \right\|_{W^{1,4/3}} ds \\
 &\lesssim \int_2^t  \left\| \Lambda^{k+5/2} u \right\|_{W^{1,8}} \left\| u \right\|_{8} \left\| xf \right\|_2 ds \lesssim \|u\|_X^3 .
\end{split}
\end{equation*}

\subsubsection*{Bound for~(\ref{bull3}) in $L^2$} Let us look more closely at the symbol $\chi_1 \mu$. It belongs to $\mathcal{T}^{5/2}$, but it is a bit smoother than general symbols of this class: namely, it has (in the worst case) a singularity of type $|\xi-\eta|^{1/2}$ at $\xi-\eta = 0$, it vanishes at order $1/2$ in $\xi$, and it is smooth for $\xi-\sigma =0$ and $\eta + \sigma =0$. Combining these observations with Lemma~\ref{derivts}, we deduce that $\partial_\xi(\chi_1 \mu)$ can be written
$$
\partial_\xi \left[ (\chi_1 \mu)(\xi,\eta,\sigma) \right] = \mu^1 + \frac{1}{|\xi|^{1/2}} \mu^2 + \frac{1}{|\xi-\eta|^{1/2}} \mu^3 + \frac{1}{|\xi-\eta-\sigma|} \mu^4
$$
with $(\mu_1,\mu_2,\mu_3,\mu_4) \in \widetilde{\mathcal{T}}_{3/2} \times  \widetilde{\mathcal{T}}_{2} \times  \widetilde{\mathcal{T}}_{2} \times \widetilde{\mathcal{T}}_{5/2}$. Thus
\begin{subequations}
\begin{align}
\label{baldeagle0}
(\ref{bull3}) & = \int_2^t \!\!\int \!\!\!\int e^{is\phi} \mu^1
\widehat{f}(s,\sigma) \widehat{f}(s,\eta)  \widehat{f}(s,\xi-\eta-\sigma)\,d\eta\,d\sigma\,ds \\
\label{baldeagle1}
& + \int_2^t \!\!\int \!\!\!\int e^{is\phi} \frac{1}{|\xi|^{1/2}} \mu^2
\widehat{f}(s,\sigma) \widehat{f}(s,\eta)  \widehat{f}(s,\xi-\eta-\sigma)\,d\eta\,d\sigma\,ds \\
\label{baldeagle2}
& + \int_2^t \!\!\int \!\!\!\int e^{is\phi} \frac{1}{|\xi-\eta|^{1/2}} \mu^3
\widehat{f}(s,\sigma) \widehat{f}(s,\eta)  \widehat{f}(s,\xi-\eta-\sigma)\,d\eta\,d\sigma\,ds \\
\label{baldeagle3}
& + \int_2^t \!\!\int \!\!\!\int e^{is\phi} \frac{1}{|\xi-\eta-\sigma|} \mu^4
\widehat{f}(s,\sigma) \widehat{f}(s,\eta)  \widehat{f}(s,\xi-\eta-\sigma)\,d\eta\,d\sigma\,ds .
\end{align}
\end{subequations}
The term~(\ref{baldeagle0}) is easily estimated, thus we skip it, and focus first on~(\ref{baldeagle1}) and~(\ref{baldeagle2}). These two terms can be estimated in essentially the same way, except that fractional integration is handled for the former with Lemma~\ref{linearbound} and for the latter with Proposition~\ref{fifs}. We simply show how to deal with~(\ref{baldeagle2}): using Proposition~\ref{fifs},
\begin{equation*}
\begin{split}
\left\| (\ref{baldeagle2}) \right\|_2 & \leq \int_2^t \left\|B_{\frac{1}{|\xi-\eta|^{1/2}} \mu^3}(u,u,u) \right\|_2 \,ds  \lesssim \int_2^t \left\|\Lambda^{2} u \right\|_4 \left\| u \right\|_4 \left\| u \right\|_4 \,ds  \lesssim \|u\|_X^3 .
\end{split}
\end{equation*}
Finally, the term~(\ref{baldeagle3}) is estimated in a very similar way to~(\ref{ant3}), thus we skip it.

\subsubsection*{Bound for~(\ref{bull4}) in $L^2$}
First compute $\displaystyle \partial_\xi \phi = \frac{1}{2} \frac{\xi}{|\xi|^{3/2}} \pm \frac{1}{2} \frac{\xi-\eta-\sigma}{|\xi-\eta-\sigma|^{3/2}}$. Therefore
\begin{subequations}
\begin{align}
\label{maple1}
(\ref{bull4}) = & \int_2^t \!\!\int \!\!\!\int e^{is\phi} s \frac{1}{2} \frac{\xi}{|\xi|^{3/2}} \chi_1 \mu
\widehat{f}(s,\sigma) \widehat{f}(s,\eta)  \widehat{f}(s,\xi-\eta-\sigma)\,d\eta\,d\sigma\,ds \\
\label{maple2}
& - \int_2^t \!\!\int \!\!\!\int e^{is\phi} s \frac{1}{2} \frac{\xi-\eta-\sigma}{|\xi-\eta-\sigma|^{3/2}} \chi_1 \mu
\widehat{f}(s,\sigma) \widehat{f}(s,\eta)  \widehat{f}(s,\xi-\eta-\sigma)\,d\eta\,d\sigma\,ds.
\end{align}
\end{subequations}
Observe that in the term~(\ref{maple1}) the singularity $\frac{1}{2} \frac{\xi}{|\xi|^{3/2}}$ is cancelled by the vanishing of the symbol $\chi_1 \mu$ in $\xi$. In other words, the symbol $\frac{1}{2} \frac{\xi}{|\xi|^{3/2}} \chi_1 \mu$ belongs to $\widetilde{\mathcal{T}}_2$. With the help of Theorem~\ref{bilinearbound}, this makes the estimate of~(\ref{maple1}) straightforward:
\begin{equation*}
\begin{split}
\left\| (\ref{maple1}) \right\|_2 & \lesssim \int_2^t s \left\| B_{\frac{\xi}{|\xi|^{3/2}} \chi_1 \mu} (u,u,u) \right\|_2\,ds  \lesssim \int_2^t s \left\|\Lambda^{2} u \right\|_6 \left\| u \right\|_6 \left\| u \right\|_6 \,ds \\
& \lesssim \|u\|_X^3 \int_2^t s s^{-\frac{2}{3} + \frac{2}{N-\frac{2}{3}}\left(\delta + \frac{2}{3} \right)} s^{-\frac{2}{3}} s^{-\frac{2}{3}} \,ds   \lesssim t^\delta \|u\|_X^3 
\end{split}
\end{equation*}
since $N$ has been taken big enough.

The estimate of~(\ref{maple2}) is a little more technical, and uses fractional integration (Lemma~\ref{linearbound}): $\delta_0$ standing for a small constant,
\begin{equation*}
\begin{split}
\left\| (\ref{maple2}) \right\|_2 & \lesssim \int_2^t s \left\| B_{
 \chi_1 \mu} \left(u,u, \frac{1}{\Lambda^{1/2}} u \right) \right\|_2\,ds \\
&\lesssim \int_2^t s \left\|\Lambda^{5/2} u \right\|_{\frac{1}{\delta_0}} \left\| u \right\|_{\frac{1}{\delta_0}} \left\|e^{is \Lambda^{1/2}} \frac{1}{\Lambda^{1/2}} f \right\|_{\frac{2}{1-4\delta_0}} \,ds\lesssim \int_2^t s \left\|\Lambda^{5/2} u \right\|_{\frac{1}{\delta_0}} \left\| u \right\|_{\frac{1}{\delta_0}} \left\|f \right\|_{\frac{4}{3-8\delta_0}} \,ds\\
& \lesssim \|u\|_X^3 \int_2^t s s^{-1+2\delta_0+\frac{5/2}{N+2\delta_0-1}(\delta - 2\delta_0 + 1)} s^{-1+2\delta_0} s^{\delta\left(\frac{1}{2}-4\delta_0\right)}  \lesssim \|u\|_X^3 t^\delta
\end{split}
\end{equation*}
since $N$ has been chosen big enough, and $\delta_0$ small enough.


\section{Estimates for the strongly resonant cubic terms}

\label{sectionstrong}

In this section, we derive the desired estimates in $L^\infty$ and $L^2(x^2dx)$ for terms of the form either~(\ref{butterfly2}) or~(\ref{butterfly3}) corresponding to strongly resonant phases as defined in Section~\ref{examinationcubic}. 
The three strongly resonant phases are identical up to a permutation of the Fourier variables. Thus, making a change of variables in the $\eta,\sigma$ integral if needed, we shall only consider the phase $\phi_{--+}$.
The generic term we consider reads
\begin{equation}
\widehat{g_3}(\xi,t) \overset{def}{=} \int_2^t \!\!\int \!\!\!\int e^{is\phi} \mu(\xi,\eta) 
\widehat{f}(s,\sigma) \widehat{f}(s,\eta) \widehat{f}(s,\xi-\eta-\sigma)\,d\eta\,d\sigma\,ds ,
\end{equation}
where
$
\phi = \phi_{--+}
$
and
$$
\mu(\xi,\eta) = m_i(\xi,\eta,\sigma) \quad\mbox{or}\quad \frac{m_k(\xi,\eta)}{i \phi_{\pm,\pm}(\xi,\eta)}m_j(\xi-\eta,\sigma)\quad\mbox{or}\quad \frac{m_k(\xi,\xi-\eta)}{i \phi_{\pm,\pm}(\xi,\xi-\eta)}m_j(\xi-\eta,\sigma) ,
$$
with $i$ equal to $3$ or $4$, $j,k$ equal to $1$ or $2$, and $\pm$ equal to $+$ or $-$.

\subsection*{Preliminary observations and reductions}

Proceeding as in Section~\ref{POR}, and keeping the same notations, matters reduce to
$$
\widehat{g_3}(t,\xi) = \int_2^t \!\!\int \!\!\!\int e^{is\phi_{\pm,\pm,\pm}} \chi_1(\xi,\eta,\sigma) \mu(\xi,\eta,\sigma)
\widehat{f}(s,\eta) \widehat{f}(s,\sigma) \widehat{f}(s,\xi-\eta-\sigma)\,d\eta \, d\sigma \,ds ,
$$
where $\chi_1 \mu \in \widetilde{\mathcal{T}}_{5/2}$.

\subsection*{Bound for $x g_3$ in $L^2$.}
In Fourier space, $x  g_3$ reads
\begin{subequations}
\begin{align}
\label{wolf2}
\mathcal{F} (x  g_3) =  & \int_2^t \!\!\int \!\!\!\int e^{is\phi} \chi_1 \mu 
\widehat{f}(s,\sigma) \widehat{f}(s,\eta) \partial_\xi \widehat{f}(s,\xi-\eta-\sigma)\,d\eta\,d\sigma\,ds \\
\label{wolf3}
& + \int_2^t \!\!\int \!\!\!\int e^{is\phi} \partial_\xi (\chi_1 \mu) 
\widehat{f}(s,\sigma) \widehat{f}(s,\eta)  \widehat{f}(s,\xi-\eta-\sigma)\,d\eta\,d\sigma\,ds \\
\label{wolf4}
& + \int_2^t \!\!\int \!\!\!\int e^{is\phi} s (\partial_\xi \phi) \chi_1 \mu
\widehat{f}(s,\sigma) \widehat{f}(s,\eta)  \widehat{f}(s,\xi-\eta-\sigma)\,d\eta\,d\sigma\,ds .
\end{align}
\end{subequations}
The terms~(\ref{wolf2}),  (\ref{wolf3}) and (\ref{wolf4}) can be estimated 
exactly as~(\ref{bull2}),  (\ref{bull3}) and (\ref{bull4}) in Section~\ref{sectionweak}. This gives, respectively, bounds of $O(1)$, $O(1)$, and $O(t^\delta)$ as $t\rightarrow \infty$, which yields the desired a priori estimate for the norm of $g_3$ in $L^2(x^2dx)$. To derive the $L^\infty$ decay of $e^{-it\Lambda^{1/2}} g_3$ however, it will be convenient to prove first that the weighted norm of $g_3$ in $L^2(x^2dx)$ is bounded as $t$ goes to infinity.

\begin{claim}
\label{clacla}
The following bound holds: $\displaystyle \left\|(\ref{wolf4}) \right\|_2 \lesssim \|u\|_X^3$. As a consequence, $\displaystyle \left\| x g_3 \right\|_2 \lesssim \|u\|_X^3$.
\end{claim}

In order to estimate~(\ref{wolf4}) without the time growth factor,
 the first step is to distinguish between high and low frequencies. Denoting $\delta_0$ for a small constant (whose precise value will be fixed later on), we split~(\ref{wolf4}) as follows:
\begin{subequations}
\begin{align}
\label{elk1}
(\ref{wolf4}) = & \int_2^t \!\!\int \!\!\!\int e^{is\phi} s (\partial_\xi \phi) \chi_1 \mu
\left[ 1 - \Theta \left( \frac{\sigma}{s^{\delta_0}} \right) \right] \widehat{f}(s,\sigma) \widehat{f}(s,\eta)  \widehat{f}(s,\xi-\eta-\sigma)\,d\eta\,d\sigma\,ds\\
\label{elk2}
& + \int_2^t \!\!\int \!\!\!\int e^{is\phi} s (\partial_\xi \phi) \chi_1 \mu
\Theta \left( \frac{\sigma}{s^{\delta_0}} \right) \widehat{f}(s,\sigma) \widehat{f}(s,\eta)  \widehat{f}(s,\xi-\eta-\sigma)\,d\eta\,d\sigma\,ds ,
\end{align}
\end{subequations}
where $\Theta$ is defined in Section~\ref{appendixlin}. The point of such a decomposition is of course that there only appears in~(\ref{elk2}) frequencies (essentially) smaller than $s^{\delta_0}$, since $\operatorname{Supp}\left( \chi_1 (\xi,\eta,\sigma) \Theta \left( \frac{\sigma}{s^\delta_0} \right) \right) \subset \{ |\xi|,|\eta|,|\sigma| \lesssim s^{\delta_0} \}$.

\subsubsection*{Bound for~(\ref{elk1})}
Due to the control in $H^{N}$, it will be easy to estimate~(\ref{elk1}). First recall that $\partial_\xi \phi = \frac{\xi}{2|\xi|^{3/2}} + \frac{\xi-\eta-\sigma}{2|\xi-\eta-\sigma|^{3/2}}$. The singularity in $\xi $ is canceled by $\mu$, thus we can forget about it, replace $\partial_\xi \phi$ by $\frac{\xi-\eta-\sigma}{2|\xi-\eta-\sigma|^{3/2}}$ in~(\ref{elk1}), and estimate with the help of Theorem~\ref{trilinearbound} and Lemma~\ref{pbound}
\begin{equation*}
\begin{split}
& \left\| \int_2^t \!\!\int \!\!\!\int e^{is\phi} s \left(\frac{\xi-\eta-\sigma}{|\xi-\eta-\sigma|^{3/2}}\right) \chi_1 \mu
\left[ 1 - \Theta \left( \frac{\sigma}{s^{\delta_0}} \right) \right] \widehat{f}(s,\sigma) \widehat{f}(s,\eta)  \widehat{f}(s,\xi-\eta-\sigma)\,d\eta\,d\sigma\,ds \right\|_2 \\
& \lesssim \int_2^t s \left\| B_{\left(\frac{\xi-\eta-\sigma}{|\xi-\eta-\sigma|^{3/2}}\right) \chi_1 \mu} \left( P_{\geq \delta_0 \log(s)} u,u,u \right) \right\|_2 ds \lesssim \int_2^t s \left\| P_{\geq \delta_0 \log(s)} \Lambda^{5/2} u \right\|_{8} \left\| u \right\|_{8} \left\| \frac{1}{\Lambda^{1/2}} u \right\|_{4}ds \\
&\lesssim \int_2^t s \left\| P_{\geq \delta_0 \log(s)} \Lambda^{5/2} u \right\|_{8} \left\| u \right\|_{8} \left\| f \right\|_{2}\,ds \lesssim \|u\|_X^3 \int_2^t s 
s^{\delta_0\left(-N + \frac{13}{4} \right)} s^{-\frac{3}{4}} \,ds  \lesssim \|u\|_X^3 .
\end{split} 
\end{equation*}
In order for the last inequality to hold, we need 
\begin{equation}
\label{conditiondN} 
1+\delta_0\left( -N +\frac{13}{4} \right) - \frac{3}{4} < -1.
\end{equation}

\subsubsection*{Bound for~(\ref{elk2})}
In order to estimate~(\ref{elk2}), a further partition of frequency space is needed. Define
\begin{equation}
\begin{split}
& A_I  \overset{def}{=} \{ \partial_\sigma \phi = 0 \} = \{ \xi=\eta \} \\
& A_{II} \overset{def}{=} \{ \partial_\eta \phi = 0 \} = \{ \xi=\sigma \} \\
& A_{III} \overset{def}{=}\{ \partial_\eta \phi = 0 \} \cap \{ \partial_\sigma \phi = 0 \} = \{ \xi=\eta=\sigma \} .
\end{split}
\end{equation}
The associated cut-off functions are $\chi_I$, $\chi_{II}$, $\chi_{III}$, which are taken such that
\begin{itemize}
\item $\chi_I$, $\chi_{II}$, $\chi_{III}$ are valued in $[0,1]$.
\item $\chi_I$, $\chi_{II}$, $\chi_{III}$ are homogeneous of degree $0$ and smooth outside of $(0,0,0)$.
\item On the sphere $|(\xi,\eta,\sigma)| = 1$, $\chi_{III} = 1$ within a distance $\frac{1}{1000}$ of $A_{III}$, and $\chi_{III} = 0$ if the distance to $A_{III}$ is more than $\frac{1}{500}$. Thus on $\operatorname{Supp} \chi_{III}$, equation~(\ref{cat}) holds.
\item On the sphere $|\xi,\eta,\sigma)| = 1$, $\chi_{I}$ (respectively $\chi_{II}$) is $1$ on a neighborhood of $A_{I}$ (respectively $A_{II}$).
\end{itemize}
Then decompose~(\ref{elk2}) with the help of these cut-off functions:
\begin{subequations}
\begin{align}
(\ref{elk2}) = & \int_2^t \!\!\int \!\!\!\int e^{is\phi} s (\partial_\xi \phi) \chi_1 \mu
\Theta \left( \frac{\sigma}{s^{\delta_0}} \right) \widehat{f}(s,\sigma) \widehat{f}(s,\eta)  \widehat{f}(s,\xi-\eta-\sigma)\,d\eta\,d\sigma\,ds \\
\label{tiger1}
= & \int_2^t \!\!\int \!\!\!\int e^{is\phi} s (\partial_\xi \phi) \chi_{I} \chi_1 \mu
\Theta \left( \frac{\sigma}{s^{\delta_0}} \right) \widehat{f}(s,\sigma) \widehat{f}(s,\eta)  
\widehat{f}(s,\xi-\eta-\sigma)\,d\eta\,d\sigma\,ds \\
\label{tiger2}
& + \int_2^t \!\!\int \!\!\!\int e^{is\phi} s (\partial_\xi \phi) \chi_{II} \chi_1 \mu
\Theta \left( \frac{\sigma}{s^{\delta_0}} \right) \widehat{f}(s,\sigma) \widehat{f}(s,\eta)  \widehat{f}(s,\xi-\eta-\sigma)\,d\eta\,d\sigma\,ds \\
\label{tiger3}
& + \int_2^t \!\!\int \!\!\!\int e^{is\phi} s (\partial_\xi \phi) \chi_{III} \chi_1 \mu
\Theta \left( \frac{\sigma}{s^{\delta_0}} \right) \widehat{f}(s,\sigma) \widehat{f}(s,\eta)  \widehat{f}(s,\xi-\eta-\sigma)\,d\eta\,d\sigma\,ds .
\end{align}
\end{subequations}
The idea will be the following: on the support of the symbol of~(\ref{tiger1}) (respectively~(\ref{tiger2})), $\partial_\eta \phi$ (respectively $\partial_\sigma \phi$) does not vanish, thus one should integrate by parts in $\eta$ (respectively $\sigma$). On the support of the symbol appearing in~(\ref{tiger3}), both $\partial_\sigma \phi$ and $\partial_\eta \phi$ vanish, but the identity~(\ref{cat}) is the remedy: it essentially converts $\partial_\xi \phi$ into a combination of $\partial_\sigma \phi$ and $\partial_\eta \phi$, making the integration by parts in $(\eta,\sigma)$ possible.

\bigskip

\noindent
\emph{Bound for~(\ref{tiger1})}
As explained above, we shall in order to treat this term integrate by parts in $\sigma$ using the identity $\frac{\partial_\sigma \phi \cdot \partial_\sigma}{i s |\partial_\sigma \phi|^2} e^{is\phi} = e^{is\phi}$. This gives
(for the sake of simplicity in the notations, we do not differentiate between standard product between scalars, and scalar product between vectors)
\begin{subequations}
\begin{align}
\label{mouse1}
(\ref{tiger1}) = & \int_2^t \!\!\int \!\!\!\int e^{is\phi} \frac{\partial_\xi \phi \partial_\sigma \phi}{|\partial_\sigma \phi|^2} \chi_I \chi_1 \mu
\Theta \left( \frac{\sigma}{s^{\delta_0}} \right) \partial_\sigma \widehat{f}(s,\sigma) \widehat{f}(s,\eta)  \widehat{f}(s,\xi-\eta-\sigma)\,d\eta\,d\sigma\,ds \\
\label{mouse2}
& + \int_2^t \!\!\int \!\!\!\int e^{is\phi} \frac{\partial_\xi \phi \partial_\sigma \phi}{|\partial_\sigma \phi|^2} \chi_I \chi_1 \mu
\Theta \left( \frac{\sigma}{s^{\delta_0}} \right) \widehat{f}(s,\sigma) \widehat{f}(s,\eta)  \partial_\sigma \widehat{f}(s,\xi-\eta-\sigma)\,d\eta\,d\sigma\,ds \\
\label{mouse3}
& + \int_2^t \!\!\int \!\!\!\int e^{is\phi} \partial_\sigma \left[ \frac{\partial_\xi \phi \partial_\sigma \phi}{|\partial_\sigma \phi|^2} \chi_I \chi_1 \mu  \Theta \left( \frac{\sigma}{s^{\delta_0}} \right) \right] \widehat{f}(s,\sigma) \widehat{f}(s,\eta)   \widehat{f}(s,\xi-\eta-\sigma)\,d\eta\,d\sigma\,ds .
\end{align}
\end{subequations}
In order to estimate~(\ref{mouse1}), observe that the symbol $\frac{\partial_\xi \phi \cdot \partial_\sigma \phi}{|\partial_\sigma \phi|^2} \chi_I \chi_1 \mu$ belongs to $\widetilde{\mathcal{T}}_{5/2}$.
Indeed, the $\frac{1}{|\xi|^{1/2}}$ singularity of $\partial_\xi \phi$ is canceled by $\mu$, and the $\frac{1}{|\xi-\eta-\sigma|^{1/2}}$ singularity of $\partial_\xi \phi$ is canceled by $\frac{\partial_\sigma \phi}{|\partial_\sigma \phi|^2}$.
Thus, applying first~(\ref{lemmadeltaj}) - since the Fourier support of the integrand lies within a ball of radius comparable to $s^{\delta_0}$ - one gets
\begin{equation} \label{nad}
\begin{split}
\left\| (\ref{mouse1}) \right\|_2 & \lesssim \int_2^t s^{\frac{\delta_0}{4}} \left\| B_{\frac{\partial_\xi \phi \partial_\sigma \phi}{|\partial_\sigma \phi|^2} \chi_I \chi_1 \mu} \left( P_{<\delta_0 \log s} e^{it\Lambda^{1/2}} (xf) , u , u \right) \right\|_{8/5} \,ds \\
& \lesssim \int_2^t s^{\frac{\delta_0}{4}} \left\| \Lambda^{5/2} P_{<\delta_0 \log s} e^{it\Lambda^{1/2}} (xf) \right\|_2 \left\|u \right\|_{16} \left\| u \right\|_{16} \,ds  \\
& \lesssim \int_2^t s^{\frac{\delta_0}{4}} s^{\frac{5}{2}\delta_0} \left\|xf \right\|_2 \left\|u\right\|_{16} \left\| u \right\|_{16} \,ds  \lesssim \|u\|_X^3 \int_2^t s^{\frac{\delta_0}{4}} s^{\frac{5}{2}\delta_0} s^\delta s^{-\frac{7}{8}} s^{-\frac{7}{8}} \,ds \lesssim \|u\|_X^3 ,
\end{split}
\end{equation}
where the last inequality holds provided that $\frac{\delta_0}{4} + \frac{5}{2}\delta_0 + \delta -\frac{14}{8} < -1$. Thus we choose $\delta_0$ and $N$ such that this inequality holds, as well as~(\ref{conditiondN}).

The estimate of~(\ref{mouse2}) is almost identical; as for the estimate of~(\ref{mouse3}), it follows in a very similar way. Let us say a word about it: the symbol $\partial_\sigma \left[ \frac{\partial_\xi \phi \partial_\sigma \phi}{|\partial_\sigma \phi|^2} \chi_I \chi_1 \mu \right]$ only contributes singularities of the type $\frac{1}{|\sigma|}$, $\frac{1}{|\xi-\eta-\sigma|}$ and $\frac{1}{|\xi-\sigma|^{1/2}}$. We have seen many instances of how these singularities can be treated; the same strategy is valid here.

\bigskip

\noindent
\emph{Bound for~(\ref{tiger2})}
It can be derived in an identical fashion.

\bigskip

\noindent
\emph{Bound for~(\ref{tiger3})}
For this term, we use the identity~(\ref{cat}), that we rewrite as follows: there exists a symbol $\rho \in \mathcal{T}_0$ such that 
$$
\chi_{III} \partial_\xi \phi = \rho(\xi,\eta,\sigma) [\partial_\eta \phi,\partial_\sigma \phi]
$$
(in the above, $\rho$ is a matrix. It is linear in the arguments between brackets).

Using the identity $\chi_{III} \partial_\xi \phi = \frac{1}{is} \rho(\xi,\eta,\sigma) \left[ \partial_\eta e^{is\phi} , \partial_\sigma e^{is\phi} \right]$ to integrate by parts in~(\ref{tiger3}) gives terms which can all be estimated as above. This concludes the proof of claim~\ref{clacla}.

\subsection*{Bound for $\nabla^k e^{-it\Lambda^{1/2}} g_3$ in $L^\infty$, $0\leq k \leq 4$}
First, notice that we can not integrate by parts in time as we did in the weakly 
resonant case since now the phase vanishes on a large set. 
Actually, the proof of this bound follows the steps of the previous argument, in particular Claim~\ref{clacla}.  Indeed, 
in the previous argument, we derived a  bound  for  $\left\| x g_3 \right\|_2$; by the same token, one might prove bounds on $\left\| \nabla^{k+2} x g_3 \right\|_2$. These two quantities barely miss to control $\|\nabla^k g_3\|_{\dot{B}^{3/2}_{1,1}}$, which would give the desired $L^\infty$ decay, by the dispersive estimate in Lemma~\ref{linearbound}. What is needed is a little bit of additional integrability, for instance replacing the Lebesgue index $2$ by $2-\delta_1$, for a small constant $\delta_1$. As we will see, this is possible if one is prepared to lose a small power of $s$; but a small enough power of $s$ is harmless in that it does not prevent the integrals over $s$ occurring above from converging. Hence, we will prove the same estimate as in the previous argument, but in $L^{2-\delta_1}$ instead of $L^2$. 

In order to implement this program, we start by splitting $g_3$ into  high and low frequencies:
\begin{equation*}
\begin{split}
g_3 & =  \int_2^t \!\!\int e^{is\phi} \chi_1 \mu \Theta \left(\frac{\sigma}{s^{\delta_0}} \right) \widehat{f}(s,\eta) \widehat{f}(s,\sigma) \widehat{f}(s,\xi-\eta-\sigma)\,d\eta \\
& \;\;\;\;\;\;\;\;\; +  \int_2^t \!\!\int e^{is\phi} \chi_1 \mu \left[ 1 - \Theta \left(\frac{\sigma}{s^{\delta_0}} \right) \right] \widehat{f}(s,\eta) \widehat{f}(s,\sigma) \widehat{f}(s,\xi-\eta-\sigma)\,d\eta \\
& \overset{def}{=} g_{3,low} + g_{3,high} .
\end{split}
\end{equation*}

\subsubsection*{Bound for $\nabla^k e^{-it\Lambda^{1/2}} g_{3,low}$ in $L^\infty$, with $0 \leq k \leq 4$}

Rewrite $g_{3,low}$ as
$$
g_{3,low} = \int_2^t G_{3,low}(s) \,ds ,
$$
with the obvious definition for $G_{3,low}(s)$. Due to the frequency localization of $G_{3,low}$, one might even write
\begin{equation}
\label{sparrow}
g_{3,low} = \int_2^t P_{<C\delta_0 \log s} G_{3,low}(s) \,ds.
\end{equation}
Computing $x G_{3,low}(s)$ gives terms similar to those appearing inside the time integral of 
 (\ref{wolf2}), (\ref{wolf3}), (\ref{elk2}), namely

\begin{subequations}
\begin{align}
\label{wolf2low}
\mathcal{F} (x  G_{3,low}(s)) =  &  \int \!\!\!\int e^{is\phi} \chi_1 \mu  
 \Theta \left(\frac{\sigma}{s^{\delta_0}} \right)
\widehat{f}(s,\sigma) \widehat{f}(s,\eta) \partial_\xi \widehat{f}(s,\xi-\eta-\sigma)\,d\eta\,d\sigma\,ds \\
\label{wolf3low}
& + \int \!\!\!\int e^{is\phi} \partial_\xi (\chi_1 \mu) \Theta \left(\frac{\sigma}{s^{\delta_0}} \right)
\widehat{f}(s,\sigma) \widehat{f}(s,\eta)  \widehat{f}(s,\xi-\eta-\sigma)\,d\eta\,d\sigma\,ds \\
\label{wolf4low}
& + \int \!\!\!\int e^{is\phi} s (\partial_\xi \phi) \chi_1 \mu\Theta \left(\frac{\sigma}{s^{\delta_0}} \right)
\widehat{f}(s,\sigma) \widehat{f}(s,\eta)  \widehat{f}(s,\xi-\eta-\sigma)\,d\eta\,d\sigma\,ds .
\end{align}
\end{subequations}

 It has been proved in the previous section  that all of these terms can be bounded in $L^2$. More precisely, an inspection of the proof there reveals that
$x G_{3,low}(s)$ can be written
$$
x G_{3,low} (s) = e^{is\Lambda^{1/2}} \widetilde{G}_{3,low}(s)\quad\mbox{with}\quad\left\| \widetilde{G}_{3,low} (s) \right\|_2 \lesssim \frac{1}{s^{1 + \kappa_1}} \|u\|_X^3 ,
$$
for some number $\kappa_1>0$. By repeating the same argument (see for instance (\ref{nad})),
 one can prove that, denoting by $\delta_1$ a small enough constant whose precise value will be set later on,
$$
\left\| \widetilde{G}_{3,low}(s) \right\|_{2-\delta_1} \lesssim \frac{1}{s^{1 + \frac{\kappa_1}{2}}} \|u\|_X^3 .
$$
Lemma~\ref{linearbound} gives
\begin{equation}
\begin{split}
\left\|\nabla^\ell P_{<C\delta_0 \log s} G_{3,low}(s) \right\|_{2-\delta_1} & = \left\|\nabla^\ell P_{<C\delta_0 \log s} e^{is\Lambda^{1/2}} \widetilde{G}_{3,low}(s) \right\|_{2-\delta_1} \\
& \lesssim s^{\ell \delta_0} \left( s^{\delta_0} s^2 \right)^{\frac{1}{2-\delta_1}-\frac{1}{2}} \left\| \widetilde{G}_{3,low}(s) \right\|_{2-\delta_1} \\
& \lesssim \|u\|_X^3 s^{\ell \delta_0} \left( s^{\delta_0} s^2 \right)^{\frac{1}{2-\delta_1}-\frac{1}{2}} \frac{1}{s^{1 + \frac{\kappa_1}{2}}} .
\end{split}
\end{equation}
Thus, choosing $\delta_0$ and $\delta_1$ small enough, and using successively the dispersive estimate in Lemma~\ref{linearbound}, the equation~(\ref{sparrow}), and the above bound,
\begin{equation}
\begin{split}
\left\|\nabla^{k} e^{-it\Lambda^{1/2}} g_{3,low} \right\|_\infty & \lesssim \frac{1}{t}\left\| \nabla^{k} g_{3,low} \right\|_{\dot{B}^{3/2}_{1,1}}  \lesssim \frac{1}{t}\left[ \left\|x g_{3,low} \right\|_{2-\delta_1} + \left\| \nabla^{k+2} x g_{3,low} \right\|_{2-\delta_1} \right] \\
& \lesssim  \frac{1}{t} \int_2^t \left[ \left\| P_{<C\delta_0 \log s} x G_{3,low} \right\|_{2-\delta_1} + \left\| \nabla^{k+2} x P_{<C\delta_0 \log s} G_{3,low} \right\|_{2-\delta_1} \right] \,ds \\
& \lesssim \frac{1}{t} \|u\|_X^3 \int_2^t s^{(k+2) \delta_0} \left( s^{\delta_0} s \right)^{\frac{1}{2-\delta_1}-\frac{1}{2}} \frac{1}{s^{1 + \frac{\kappa_1}{2}}}\,ds\lesssim \frac{1}{t} \|u\|_X^3 .
\end{split}
\end{equation}

\subsubsection*{Bound for $\nabla^k e^{-it\Lambda^{1/2}} g_{3,high}$ in $L^\infty$, with $0 \leq k \leq 4$}

Rewrite $g_{3,high}$ as
$$
g_{3,high} = \int_2^t e^{is\Lambda^{1/2}} \widetilde{G}_{3,high}(s) \,ds. 
$$
Computing $\nabla^\ell x \widetilde{G}_{3,high}(s)$ gives terms which can be estimated just like~(\ref{elk1}); one gets a bound of the type
$$
\left\| \nabla^\ell x \widetilde{G}_{3,high}(s) \right\|_2 \lesssim \frac{1}{s^{1+\kappa_2}} \|u\|_X^3 ,
$$
for a constant $\kappa_2 > 0$. Following the same arguments, but going to a smaller Lebesgue index $2-\delta_2$ gives, provided $\delta_2$ is chosen small enough, and $\ell \leq 6$,
$$
\left\| \nabla^\ell x \widetilde{G}_{3,high}(s) \right\|_{2-\delta_2} \lesssim \frac{1}{s^{1+\frac{\kappa_2}{2}}} \|u\|_X^3 .
$$
As above, this implies that
$$
\left\| \nabla^k \widetilde{G}_{3,high}(s) \right\|_{\dot{B}^{3/2}_{1,1}} \lesssim \frac{1}{s^{1+\frac{\kappa_2}{2}}} \|u\|_X^3 .
$$
To deduce the $L^\infty$ estimate, write
$$
\nabla^k e^{-it\Lambda^{1/2}} g_{3,high} = \int_2^{t-1} \nabla^k e^{i(s-t)\Lambda^{1/2}} \widetilde{G}_{3,high}(s)\,ds + \int_{t-1}^t \nabla^k e^{i(s-t)\Lambda^{1/2}} \widetilde{G}_{3,high}(s)\,ds.
$$
The second summand of the above right-hand side can be estimated using Sobolev embedding and 
 proceeding as in~(\ref{walrus}). As for the first summand,
\begin{equation}
\begin{split}
\left\| \int_2^{t-1}\nabla^k e^{i(s-t)\Lambda^{1/2}} \widetilde{G}_{3,high}(s)\,ds \right\|_\infty \lesssim \|u\|_X^3 \int_2^{t-1} \frac{1}{t-s} \frac{1}{s^{1+\frac{\kappa_2}{2}}} \, ds \lesssim \frac{1}{t} \|u\|_X^3 .
\end{split}
\end{equation}

\section{Estimates for terms of order 4 and higher}

\label{sectionremainder}

In this section, we derive the desired estimates in $L^\infty$ and $L^2(x^2dx)$ for terms of the form~(\ref{butterfly4}) and (\ref{butterfly5}). These terms gather the contributions of order 4 and higher (in $u$) to the nonlinearity, once the normal form transform of Section~\ref{normalform} has been performed.

Being of high order, they decay very fast and can be estimated by brute force, leaving aside the question of resonances.
Since a straightforward approach is sufficient, we only illustrate it in the case of~(\ref{butterfly4}), which we denote
$$
g_4 = \int_2^t e^{is\Lambda^{1/2}} R(s) \,ds.
$$

\subsection*{Bound for $\nabla^k e^{-it\Lambda^{1/2}} g_4$ in $L^\infty$, for $0 \leq k \leq 4$}

We split
$$
\nabla^k e^{-it\Lambda^{1/2}} g_4 = \int_2^t \nabla^k e^{i(s-t)\Lambda^{1/2}} R(s) \,ds
$$
into $\int_2^{t-1} + \int_{t-1}^t$. The first summand is bounded via the dispersive estimate of Lemma~\ref{linearbound} and Proposition~\ref{propR}
$$
\left\| \int_2^{t-1} \nabla^k e^{i(s-t)\Lambda^{1/2}} R(s) \,ds \right\|_{\infty} \lesssim \int_2^{t-1} \frac{1}{t-s} \| R(s) \|_{\dot{B}^{k+\frac{3}{2}}_{1,1}} \,ds \lesssim \frac{1}{t} \|u\|_X^4.
$$
The second summand is bounded using Sobolev embedding and Proposition~\ref{propR}
$$
\left\| \int_{t-1}^t \nabla^k e^{i(s-t)\Lambda^{1/2}} R(s) \,ds \right\|_{\infty} \lesssim \int_{t-1}^t \left\| \nabla^k e^{i(s-t)\Lambda^{1/2}} R(s) \right\|_{\infty}\,ds \lesssim \int_{t-1}^t \left\| R(s) \right\|_{H^{2+k}} \,ds \lesssim \frac{1}{t} \|u\|_X^4 .
$$

\subsection*{Bound for $x g_4$ in $L^2$}

The $L^2$ norm of $xg_4$ can be bounded as follows: 
\begin{equation*}
\begin{split}
\left\| x g_4 \right\|_2 & = \left\| x \int_2^t e^{is\Lambda^{1/2}} R(s) \,ds \right\|_2 \\
& \leq \left\| \int_2^t s \frac{1}{\Lambda^{1/2}} e^{is\Lambda^{1/2}} R(s) \,ds \right\|_2 + \left\| \int_2^t  e^{is\Lambda^{1/2}} x R(s) \,ds \right\|_2 \\
& \lesssim \int_2^t s \left\| R(s) \right\|_{4/3} \, ds + \int_2^t \left\|xR\right\|_2 \,ds.
\end{split}
\end{equation*}
The first term in the last line above can be estimated directly using point $(i)$ of Proposition~\ref{propR}; we leave this to the reader.
The bound for the second term follows with the help of point $(ii)$ of Proposition~\ref{propR}:
$$
\int_2^t \left\|xR\right\|_2 \,ds \lesssim \int_2^t \|\langle x \rangle u \|_2 \|u\|_{W^{3,\infty}}^3 \,ds
\lesssim \int_2^t \left[ \|u\|_2 + s \left\| f \right\|_{4/3} + \left\| x f\right\|_2 \right] \|u\|_{W^{3,\infty}}^3\,ds \lesssim \|u\|^4_X.
$$

\section{Scattering}

\label{sectionscat}

In this section, we prove the scattering result (Corollary~\ref{scat}). On the one hand, decomposing $f$ into terms of the form~(\ref{butterfly1}--\ref{butterfly5}), and examining them separately, it appears that it is a Cauchy sequence in $L^2$ in the sense that there exists a constant $\kappa>0$ such that
$$
\mbox{if $s<t<2s$,}\;\;\;\;\;\;\left\|f(t)-f(s)\right\|_2 \lesssim \frac{1}{t^\kappa}.
$$
On the other hand, we know that
$$
\left\|f(t)\right\|_{H^N \cap L^2(x^2dx)} \lesssim t^\delta\,\,.
$$
Interpolating between these two inequalities gives, for a constant $C$,
$$
\mbox{if $s<t<2s$,}\;\;\;\;\;\;\left\|f(t)-f(s)\right\|_{H^{N-N C_0\delta} \cap L^2(x^{2-C_0 \delta} dx)} \lesssim t^{\delta \left[ -\kappa C_0 + 1 \right]}.
$$
Choosing $C_0$ properly makes $f(t)$ Cauchy in time, which gives the desired result.

\appendix

\section{A few results of linear harmonic analysis}

\label{appendixlin}

\subsection*{Littlewood-Paley theory}

Consider $\theta$ a function supported in the annulus $\mathcal{C}(0,\frac{3}{4},\frac{8}{3})$ such that
$$
\mbox{for $\xi \neq 0$,}\quad\sum_{j \in \mathbb{Z}} \theta \left( \frac{\xi}{2^j} \right) = 1 .
$$
Also define
$$
\Theta(\xi) \overset{def}{=} \sum_{j <0} \theta \left( \frac{\xi}{2^j} \right) .
$$
Define then the Fourier multipliers
$$
P_j \overset{def}{=} \theta\left( \frac{D}{2^j} \right) \quad P_{<j} = \Theta \left( \frac{D}{2^j} \right)\quad P_{\geq j} = 1- \Theta\left( \frac{D}{2^j} \right)
$$
and similarly $P_{\leq j}$, $P_{>j}$.
This gives a homogeneous and an inhomogeneous decomposition of the identity (for instance, in $L^2$)
$$
\sum_{j \in \mathbb{Z}} P_j = \operatorname{Id} \quad\mbox{and}\quad P_{<0} + \sum_{j> 0} P_j = \operatorname{Id}.
$$
All these operators are bounded on $L^p$ spaces:
$$
\mbox{if $1 \leq p \leq \infty$,}\quad \|P_j f \|_p \lesssim \|f\|_p \quad,\quad \|P_{<j} f \|_p \lesssim \|f\|_p\quad\mbox{and} \quad\|P_{>j} f \|_p \lesssim \|f\|_p.
$$
Furthermore, for $P_j f$, taking a derivative is essentially equivalent to multiplying by $2^j$:
\begin{equation}
\label{LPderivative}
\begin{split} 
& \mbox{if $1 \leq p \leq \infty$ and $\alpha \in \mathbb{R}$,}\quad \|\Lambda^\alpha P_j f \|_p \sim 2^{\alpha j} \|P_j f\|_p\\
& \mbox{if $1 \leq p \leq \infty$ and $\ell \in \mathbb{Z}$,} \quad \|\nabla^\ell P_j f \|_p \sim 2^{\ell j} \|P_j f\|_p.
\end{split}
\end{equation}
Also, we  recall Bernstein's lemma: if $1\leq q\leq p \leq \infty$,
\begin{equation}
\label{lemmadeltaj}
\|P_j f \|_p \leq 2^{2j\left( \frac{1}{q}-\frac{1}{p} \right)} \left\| P_j f \right\|_q\quad\mbox{and}\quad \left\| P_{<j} f \right\|_p \leq 2^{2j\left( \frac{1}{q}-\frac{1}{p} \right)} \left\| P_{<j} f \right\|_q .
\end{equation}
Finally, we will need the Littlewood-Paley square and maximal function estimates

\begin{theorem}
\label{LP}
(i) If $f = \sum f_j$, with $\operatorname{Supp}(f_j) \subset \mathcal{C}(0,c2^{-j},C2^{-j})$ (the latter denoting the annulus of center $0$, inner radius $c2^{-j}$, outer radius $C2^{-j}$), and $1 < p < \infty$,
$$
\left\| \sum_j f_j \right\|_p \lesssim \left\| \left[ \sum_j f_j^2 \right]^{1/2} \right\|_p.
$$
Furthermore, denoting $\displaystyle Sf \overset{def}{=} \left[ \sum_j (P_j f)^2 \right]^{1/2}$, $\displaystyle \|Sf\|_p \sim \| f \|_p .$

(ii) If $1 < p \leq \infty$, denoting $ \displaystyle Mf(x) \overset{def}{=} \sup_j \left| P_{<j} f (x) \right|$, $\displaystyle \|Mf\|_p \lesssim \|f\|_p$.
\end{theorem}
 
\subsection*{Fractional integration and dispersion}

\begin{lemma}
\label{linearbound}
\begin{enumerate}
\item (Fractional integration) If $1<p,q<\infty$, $0 < \alpha < \frac{2}{p}$, and $\alpha = \frac{2}{p} - \frac{2}{q}$, then $\displaystyle \left\|\frac{1}{\Lambda^\alpha} f \right\|_q \lesssim \|f\|_p$.
\item (Dispersive estimate) $\displaystyle \left\| e^{it\Lambda^{1/2}} f \right\|_\infty \lesssim \frac{1}{t} \|f\|_{\dot{B}^{3/2}_{1,1}}$.
\item If $1<p \leq 2 \leq q < \infty$, $0 < \alpha < \frac{2}{p}$, and $\alpha = \frac{2}{p} - \frac{2}{q}$, then $\displaystyle \left\|\frac{1}{\Lambda^\alpha} e^{it\Lambda^{1/2}} f \right\|_q \lesssim \|f\|_p$.
\item If $1 \leq p \leq 2 $, $\ell \geq 0$ and $2^j t^2 \geq 1$, $\displaystyle \left\|\nabla^\ell P_{<j} e^{it\Lambda^{1/2}} f \right\|_p \lesssim 2^{\ell j}\left(2^j t^2 \right)^{\left( \frac{1}{p}-\frac{1}{2} \right)} \|f\|_p$.
\end{enumerate}
\end{lemma}
\begin{proof} We only prove the second  and forth  points, the rest being standard or elementary.

\medskip

\noindent
\emph{Proof of $(2).$}
The stationary phase lemma gives
$$
\left\| \mathcal{F}^{-1} e^{it|\xi|^{1/2}} \tilde \theta(\xi) \right\|_\infty \lesssim \frac{1}{t} .
$$
This implies that
$$
\left\| P_0 e^{it\Lambda^{1/2}} f \right\|_\infty \lesssim \left\| P_0  f \right\|_1 \left\| \mathcal{F}^{-1} e^{it|\xi|^{1/2}} \tilde \theta(\xi) \right\|_\infty \lesssim \frac{1}{t} \left\| P_0 f \right\|_1 .
$$
By scaling,
$$
\left\| P_j e^{it\Lambda^{1/2}} f \right\|_\infty \lesssim 2^{\frac{3}{2}j} \frac{1}{t} \| P_j  f\|_1.
$$
This inequality gives immediately the desired conclusion.

\medskip
 
\noindent
\emph{Proof of $(4).$} First notice that it suffices to show this result for $\ell = 0$. It will follow from interpolation between the $L^2$ estimate, which is clear, and the $L^1$ estimate, which reads
$$
\mbox{if $2^j t^2 \geq 1$,}\qquad\left\| P_{<j} e^{it\Lambda^{1/2}} f \right\|_1 \lesssim \left(2^j t^2 \right) \|f\|_1 .
$$
By scaling, it suffices to prove this estimate if $t = 1$ and $j \geq 0$. This is done as follows
\begin{equation*}
\begin{split}
\left\| P_{<j} e^{i\Lambda^{1/2}} \right\|_{L^1 \rightarrow L^1} & \lesssim \left\| \mathcal{F}^{-1} \Theta \left( \frac{\xi}{2^j} \right) e^{i|\xi|^{1/2}} \right\|_1 \\
& \lesssim \left\| \mathcal{F}^{-1} \Theta \left( \xi \right) e^{i|\xi|^{1/2}} \right\|_1 + \sum_{k=1}^j \left\| \mathcal{F}^{-1} \theta \left( \frac{\xi}{2^k} \right)e^{i|\xi|^{1/2}} \right\|_1 \\
& \lesssim 1 + \sum_{k=1}^j \left\| \mathcal{F}^{-1} \theta \left( \frac{\xi}{2^k} \right) e^{i |\xi|^{1/2}} \right\|_2^{1/2} \left\| x^2 \mathcal{F}^{-1} \theta \left( \frac{\xi}{2^k} \right) e^{i|\xi|^{1/2}} \right\|_2^{1/2} \\
& \lesssim 1 + \sum_{k=1}^j \left\|\theta \left( \frac{\xi}{2^k} \right) e^{i |\xi|^{1/2}} \right\|_2^{1/2} \left\| \partial_\xi^2 \left[ \theta \left( \frac{\xi}{2^k} \right) e^{i|\xi|^{1/2}} 
  \right]  \right\|_2^{1/2} \\
& \lesssim 1 + \sum_{k=1}^j 2^{k/2} \left( 1 + 2^{-k} \right)^{1/2} \lesssim 2^{j/2}.
\end{split}
\end{equation*}
\end{proof}

\section{Some general facts on pseudo-product operators}

\label{appendixpp}

Let us first give the definition of pseudo-product operators, which were introduced by Coifman and Meyer; we only consider the bilinear and trilinear cases, which are of interest for our problem. 
Bilinear (respectively trilinear) operators are defined via their symbol $m(\xi,\eta)$ (respectively $m(\xi,\eta,\sigma)$) by
\begin{equation*}
\begin{split}
& B_{m(\xi,\eta)} (f_1,f_2) \overset{def}{=} {\mathcal{F}}^{-1} \int m(\xi,\eta) \hat{f}_1(\eta) \hat{f}_2(\xi-\eta) d\eta \\
& B_{m(\xi,\eta,\sigma)}(f_1,f_2,f_3) \overset{def}{=} {\mathcal{F}}^{-1} \int m(\xi,\eta,\sigma) \hat{f}_1(\sigma) \hat{f}_2(\eta) \hat{f}_3(\xi-\eta-\sigma) \,d\eta \,d\sigma.
\end{split}
\end{equation*}

The fundamental theorem of Coifman and Meyer~\cite{coifmanmeyer} states, under a natural condition, that these operators have the same boundedness properties as the ones given by H\"older's inequality for the standard product.
\begin{theorem}
\label{CoifmanMeyer}
Let $n$ be either $2$ or $3$, and suppose that $m$ satisfies
\begin{equation}
\label{butterfly}
\partial_{\xi_1}^{\alpha_1} \dots \partial_{\xi_n}^{\alpha_n} m (\xi_1,\dots,\xi_n) \lesssim \frac{1}{\left(|\xi_1| + \dots + |\xi_n| \right)^{|\alpha_1| + \dots + |\alpha_n|}}
\end{equation}
for sufficiently many multi-indices. Then
$$
\left\| B_m (f_1,\dots,f_n) \right\|_r \lesssim \|f_1\|_{p_1} \dots \|f_n\|_{p_n} \quad\mbox{if 
$\frac{1}{r} = \frac{1}{p_1} + \dots + \frac{1}{p_n}$, $1<p_1\dots p_n \leq \infty$ and  $0<r<\infty$.}
$$
\end{theorem}

The class of symbols allowed by the Coifman-Meyer theorem (typically, symbols homogeneous of degree 0 and smooth outside the origin) does not contain the symbols occurring in our analysis of the water wave problem. We will therefore be led to introducing a new class of symbols, in section~\ref{appendquad}.

In the trilinear case, these new symbols will amongst other discrepancies with the Coifman-Meyer case exhibit flag singularities, to which we now turn. What is a flag singularity? If one considers two bilinear symbols $m_1$ and $m_2$, and one trilinear symbol $m_3$, all of Coifman-Meyer type, the new symbol $m(\xi,\eta,\sigma)= m_3(\xi,\eta,\sigma) m_1(\xi,\eta) m_2(\xi-\eta,\sigma)$ does not satisfy the Coifman-Meyer estimates (in a neighborhood of $\xi=\eta=0$ and $\xi-\eta-\sigma =\sigma=0$). This type of symbol is said to have a flag singularity. The following theorem gives the boundedness of such operators. A result of Muscalu~\cite{Muscalu} is slightly less general, but we are able to give a simpler proof since the range of Lebesgue exponents we are interested in is smaller.

\begin{theorem}
\label{theoflag}
If $m_1(\xi,\eta)$, $m_2(\xi,\eta)$, $m_3(\xi,\eta,\sigma)$ are symbols satisfying the estimate~(\ref{butterfly}), then
$$
\left\| B_{m_3(\xi,\eta,\sigma)m_1(\xi,\eta)m_2(\xi-\eta,\sigma)} (f,g,h) \right\|_p \lesssim \|f\|_{p_1} \|g\|_{p_2} \|h\|_{p_3} 
$$
for $1<p,p_1,p_2,p_3<\infty$ and $\frac{1}{p_1}+\frac{1}{p_2}+\frac{1}{p_3} = \frac{1}{p}$.
\end{theorem}

\begin{proof} \noindent \emph{Step 1: Partition of the frequency space.}
Set $m(\xi,\eta,\sigma) = m_3(\xi,\eta,\sigma)m_1(\xi,\eta)m_2(\xi-\eta,\sigma)$. First observe that there are certain regions of the $(\xi,\eta,\sigma)$ plane where $m$ satisfies the Coifman-Meyer estimates~(\ref{butterfly}); then the Coifman-Meyer theorem applies, and the desired estimate is proved. 
Thus, using a cut-off function, we can reduce the problem to the regions where the Coifman-Meyer estimate~(\ref{butterfly}) does not hold for $m$, namely 
$$
A_1 \cup A_2 \overset{def}{=} \{ |\xi| + |\eta| < <
 |\sigma| \} \cup \{ |\xi-\eta| + |\sigma| < < |\xi| \},
$$
We further observe that on $A_1$ (respectively: on $A_2$), $m_2(\xi-\eta,\sigma)$ (respectively $m_1(\xi,\eta)$) satisfies the Coifman-Meyer estimate in $(\xi,\eta,\sigma)$. Thus with the help of cutoff functions, we can reduce matters to symbols of one of the two following types 
\begin{subequations}
\begin{align}
\label{lion1} 
& \chi_{A_1}(\xi,\eta,\sigma) \tilde m_3(\xi,\eta,\sigma) m_1(\xi,\eta)\quad\mbox{(where $\chi_{A_1}$ localizes near $A_1$)} \\
\label{lion2} 
& \chi_{A_2}(\xi,\eta,\sigma)  \tilde m_3(\xi,\eta,\sigma) m_2(\xi-\eta,\sigma)\quad\mbox{(where $\chi_{A_2}$ localizes near $A_2$)}
\end{align}
\end{subequations}
where $\tilde m_3$ stands for $m_3 m_2$ or $m_3 m_1$. 
Estimates for one of these symbols can be deduced from the other by duality and using that for any two symbols $\mu$ and $\nu$, and $\langle\cdot , \cdot \rangle$ denoting the standard (complex) scalar product,
$$
\langle B_{\mu(\xi,\eta,\sigma) \nu(\xi,\eta)} (f_1,f_2,f_3) \,,\,f_4 \rangle = \langle B_{\mu(-\sigma,\xi-\eta-\sigma,-\xi) \nu(-\sigma,\xi-\eta-\sigma)} (\bar{f_4},f_3,f_2)\,,\,\bar{f_1} \rangle .
$$
Therefore, we simply prove estimates for the symbol~(\ref{lion1}). We define the cut-off function $\chi_{A_1}$ by
\begin{equation}
\label{squirrel}
B_{\chi_{A_1}(\xi,\eta,\sigma)}(f,g,h) = \sum_{|k-k'|\leq 1}  P_{<k-100} \left(P_k f  P_{k'} h \right)P_{<k-100} g.
\end{equation}
We will suppress the index $k'$ in the sequel and just take $k'=k$ to make notations lighter. 
\bigskip

\noindent
\emph{Step 2: series expansion for the symbol $\tilde m_3$}
Expanding the symbol $\tilde m_3(\xi,\eta,\sigma)$ near $(\eta,\xi)=0$ gives
$$ 
\tilde m_3(\xi,\eta,\sigma) = \sum_{|\alpha+\beta|=0}^{M-1} \Phi_{\alpha,\beta}(\sigma) \xi^\alpha \eta^\beta + R(\xi,\eta,\sigma),
$$ 
where $\Phi_{\alpha,\beta}(\sigma) = \frac{\partial^{\alpha+\beta}}{\partial^\alpha \xi \partial^\beta \eta} m_3(0,0,\sigma)$ has homogeneous bounds of degree $-|\alpha+\beta|$ i.e. satisfies $\left|\partial_\sigma^{\gamma} \Phi_{\alpha,\beta} \right| \lesssim |\sigma|^{-|\gamma|-|\alpha+\beta|}$, and the remainder $R$ is such that 
\begin{equation}
\label{boundR}
\left| \partial_{\eta,\xi}^\beta \partial^\gamma_\sigma R(\xi,\eta,\sigma) \right| = O \left( \frac{\left( |\xi| + |\eta| \right)^{M-|\beta|}}{|\sigma|^{M+|\alpha|}} \right).
\end{equation}
This estimate implies that, if $M$ is chosen big enough, the symbol resulting from the multiplication of $R$ and $m_1 \chi_{A_1}$ is of Coifman-Meyer type, thus we can forget about it. Thus it suffices to treat the summands of the first term of the above right-hand side; to simplify notations a little in the following, we simply consider the case $\alpha=0$, and therefore get symbols of the type
\begin{equation}
\label{marmotte}
m(\xi,\eta,\sigma) = \chi_{A_1}(\xi,\eta,\sigma) \Phi_i(\sigma) \eta^i m_1(\xi,\eta) \,\,,
\end{equation}
where $\Phi_i$ has homogeneous bounds of degree $-i$, and $m_1$ satisfies the Coifman-Meyer estimates. 

\bigskip

\noindent 
\emph{Step 3: Paraproduct decomposition for the symbol $m_1$.}
Proceeding as in the original proof of Coifman and Meyer~\cite{coifmanmeyer}, we will now perform a paraproduct decomposition of $m$, and then expand the resulting symbols scale by scale. Write
$$
B_{m_1}(f,g) = \sum_j B_{m_1} (P_j f P_{<j-1} g) + \sum_j B_{m_1} (P_{<j-1} f P_j g) + \sum_{|j-\ell|\leq 1} B_{m_1} (P_j f, P_\ell g).
$$
Next consider the symbol of one of the elementary bilinear operators above, for instance $B_{m_1}(P_j \cdot, P_{<j-1} \cdot)$. Denote $m_1^j(\xi,\eta)$ for this symbol, which is compactly supported (in $(\xi,\eta)$), and expand it in Fourier series 
$$
m_1^j(\xi,\eta) = \chi(2^{-j}(\xi,\eta)) \sum_{p,q \in \mathbb{Z}^2} a_{p,q}^j  e^{i c2^{-j} (p,q)\cdot (\xi,\eta)},
$$
where we denoted $c$ for a constant, $\chi$ for a cut-off function, and $a_{p,q}^j$ for the Fourier coefficients.

It is now possible to forget about the summation over $p,q$ because of the fast decay of the coefficients $a^j_{p,q}$, which results from the smoothness of the symbol. Indeed, the complex exponentials $e^{i c2^{-j} (p,q).(\xi,\eta) }$ become in physical space translations, which add polynomial factors to the estimates to come. These polynomial factors are offset by the decay of the $a_{p,q}$. This leads to replacing $a_{p,q}^j$ by $a_j \in \ell^\infty$. It is now possible to consider that the $a_j$ are actually constant in $j$: if this is not the case, it essentially corresponds to composing one of the argument functions by a Calderon-Zygmund operator bounded on Lebesgue spaces, which is harmless. For these reasons, it will suffice to treat the case when $m_1$ corresponds to one of the three paraproduct operators
$$
(f,g) \mapsto \quad \sum_j P_j f P_{<j-1} g \quad ; \quad \sum_j P_{<j-1} f P_j g \quad ; \quad \sum_{j}  P_j f P_j g
$$
(we suppressed the index $\ell$ in the last summation to make notations lighter by taking 
$\ell=j$). 

\bigskip

\noindent
\emph{Step 4: Derivation of the model operators.}
Combining this last line with~(\ref{marmotte}), we see that the operators of interest for us become
\begin{equation}
\begin{split}
& \sum_{j,k} P_{<k-100} P_j \left(\Phi_i(D) P_k f P_k h \right) \nabla^i P_{<j-1} P_{<k-100} g \\
& \sum_{j,k} P_{<k-100} P_{<j-1} \left(\Phi_i(D) P_k f P_k h \right) \nabla^i P_j P_{<k-100} g \\
& \sum_{j,k} P_{<k-100} P_j \left(\Phi_i(D) P_k f P_k h \right) \nabla^i P_j P_{<k-100} g.
\end{split}
\end{equation}
We now make some observations which allow us to simplify the above operators:
\begin{itemize}
\item First remark that $\Phi_i(D) P_k$, $\nabla^i P_j$ and $\nabla^i P_{<j}$ can be written respectively $2^{-ik} \widetilde{P}_k$, $2^{ij} \widetilde{\widetilde{P}}_j$ and $2^{ij} \widetilde{P}_{<j}$ with obvious notations. Since the operators with tildes have similar properties to the operators without tildes, we will in the following forget about the tildes.
\item Next notice that due to the Fourier space support properties of the different terms above, it is possible to restrict the summation to $j \leq k-97$.
\item Finally, since $P_{<k-100} P_j = P_j$ and $P_{<j-1} P_{<k-100} = P_{<j-1}$ for $j \leq k-103$, it is harmless to forget about the $P_{<k-100}$ operators in the above sums.
\end{itemize}
All these remarks lead to the following simplified versions of the above operators:
\begin{equation}
\label{kingfisher}
\begin{split}
& \sum_{k \geq j+97} 2^{i(j-k)} P_j \left(P_k f P_k h \right) P_{<j-1} g \\
& \sum_{k \geq j+97} 2^{i(j-k)} P_{<j-1} \left(P_k f P_k h \right) P_j g \\
& \sum_{k \geq j+97} 2^{i(j-k)} P_j \left(P_k f P_k h \right) P_j g
\end{split}
\end{equation}

\bigskip

\noindent
\emph{Step 5: The case $i=0$.}
If $i=0$, observe that, due to the Fourier support properties of the Littlewood-Paley operators, the operators in~(\ref{kingfisher}) are equal respectively to
\begin{equation}
\label{whale}
\begin{split}
& \sum_j P_j \left(\sum_k P_k f P_k h \right) P_{<j-1} g \\
& \sum_j  P_{<j-1} \left(\sum_k P_k f P_k h \right) P_j g \\
& \sum_j P_j \left(\sum_k P_k f P_k h \right) P_j g
\end{split}
\end{equation}
up to a difference term which is Coifman-Meyer. But the operators in~(\ref{whale}) are simply compositions of bilinear Coifman-Meyer operators. Thus the desired bounds follow for them.

\bigskip

\noindent
\emph{Step 6: The case $i>0$.}
If $i>0$, we see that it suffices to prove uniform estimates in $J \geq 0$ for the operators
\begin{subequations}
\begin{align}
\label{beaver1}
& \sum_{j} P_j \left(P_{j+J} f P_{j+J} h \right) P_{<j-1} g \\
\label{beaver2} & \sum_{j} P_{<j-1} \left(P_{j+J} f P_{j+J} h \right) P_j g \\
\label{beaver3} & \sum_{j} P_j \left(P_{j+J} f P_{j+J} h \right) P_j g
\end{align}
\end{subequations}
since the desired result follows then upon summation over $J$. The estimate relies on the Littlewood-Paley square and maximal function estimates (Theorem~\ref{LP}), and on the vector valued maximal function estimate $\left\| \left( \sum_j \left[ M f_j \right]^2 \right)^{1/2} \right\|_p \lesssim \|\left( \sum_j f_j^2 \right)^{1/2} \|_p$ (see Stein~\cite{stein}, chapter II). This gives for~(\ref{beaver1})
\begin{equation*}
\begin{split}
\left\| (\ref{beaver1}) \right\|_p & \lesssim \left\| \left( \sum_j \left[ P_j \left(P_{j+J} f P_{j+J} h \right) P_{<j-1} g \right]^2 \right)^{1/2} \right\|_p \lesssim \left\| Mg \left( \sum_j \left[ P_j \left(P_{j+J} f P_{j+J} h \right)\right]^2 \right)^{1/2} \right\|_p \\
& \lesssim \left\| Mg \right\|_{p_2} \left\| \left(\sum_j \left[ P_j \left(P_{j+J} f P_{j+J} h \right)\right]^2 \right)^{1/2} \right\|_{\frac{pp_2}{p-p_2}} \lesssim \left\| g \right\|_{p_2} \left\| \left(\sum_j \left[ M \left(P_{j+J} f P_{j+J} h \right) \right]^2 \right)^{1/2} \right\|_{\frac{pp_2}{p-p_2}} \\
& \lesssim \left\| g \right\|_{p_2} \left\| \left( \sum_j \left[ P_{j+J} f P_{j+J} h \right]^2 \right)^{1/2} \right\|_{\frac{pp_2}{p-p_2}} \lesssim \left\| g \right\|_{p_2} \left\| M f \left( \sum_j \left[ P_{j+J} h \right]^2 \right)^{1/2} \right\|_{\frac{pp_2}{p-p_2}}\\
&  \lesssim \left\| g \right\|_{p_2} \left\| M f \right\|_{p_1} \left\| S h \right\|_{p_3} \lesssim  \left\| g \right\|_{p_2} \left\| f \right\|_{p_1} \left\| h \right\|_{p_3}.
\end{split}
\end{equation*}
The terms~(\ref{beaver2}) and~(\ref{beaver3}) can be estimated similarly.
\end{proof}

\section{Analysis of a class of  bilinear pseudo-product operators}
\label{appendquad}
In this section, we define new classes of bilinear pseudo-product operators, which occur in the analysis of the water wave problem; then, we prove boundedness of these operators.

\subsection*{Definition of the classes $\mathcal{B}_s$ and $\widetilde{\mathcal{B}_s}$}

Before defining these classes, it will be convenient to think of a symbol in a more symmetric way than we have been doing so far. The motivation is the following: in order to  prove a bound  for the bilinear operator $B_m(f,g)$ in $L^r$ one uses a duality argument and proves that for $h \in L^{r'}$, we have $\int B_m(f,g) h dx < \infty$; this shows that the Fourier variables $\eta,\xi-\eta,-\xi $  play symmetric roles. 

Given a symbol $m(\xi,\eta)$, we can as well write it $m'(\xi,\xi-\eta)\overset{def}{=}m(\xi,\eta)$ or more generally as a function of two of the three variables $(\eta,\xi-\eta,-\xi)$. By an abuse of notation, we will denote indifferently $m$ for all these symbols.
Thus, denoting $(\xi_1,\xi_2,\xi_3)$ for the three Fourier variables $(\eta,\xi-\eta,-\xi)$, and picking two indices $i$ and $j$, we can always write $m = m(\xi_i,\xi_j)$.

\begin{defi}
A symbol $m$ belongs to the class $\mathcal{B}_s$ if
\begin{itemize}
\item It is homogeneous of degree $s$.
\item It is smooth outside of $\{ \xi_1 = 0 \} \cup \{ \xi_2 = 0 \}\cup \{ \xi_3 = 0 \}$.
\item For $i=1,2,3$, if $|\xi_i| << |\xi_{i+1}|, |\xi_{i+2}| \sim 1$, it is possible to write $m = \mathcal{A}\left( |\xi_i|^{1/2} , \frac{\xi_i}{|\xi_i|} , \xi_{i+1} \right)$.
\end{itemize}
(where we use the convention that $\xi_4 = \xi_1$ and $\xi_5 = \xi_2 $). 
\end{defi}

Thus, roughly speaking, symbols in $\mathcal{B}_0$ are of Coifman-Meyer type except along the coordinate axes, where they are allowed to have a singularity like Mihlin-H\"ormander (linear!) multipliers.

Notice that given the boundedness properties of Mihlin-H\"ormander multipliers and Coifman-Meyer operators, Theorem~\ref{bilinearbound} of next section, which gives boundedness of bilinear operators with symbols in $\mathcal{B}_0$, should be no surprise.

We now define a new class of symbols, which corresponds to bilinear operators of paraproduct type.

\begin{defi}
A symbol $m$ belongs to the class $\widetilde{\mathcal{B}}_s$ if
\begin{itemize}
\item It belongs to $\mathcal{B}_s$.
\item It satisfies the following support property: $\operatorname{Supp} m(\xi,\eta) \subset \{ |\eta| \gtrsim |\xi| \}$. 
\end{itemize}
\end{defi}

The interest of the class $\widetilde{\mathcal{B}}_s$ is the following: taking derivatives of $B_\mu(f,g)$ corresponds to multiplying it by $\xi$ in Fourier space.  If $m\in\widetilde{\mathcal{B}}_s$  then the support restriction on $m$ means that $\xi$ is always dominated by $\eta$. Thus one expects (see next section for a precise formulation) something like $|\nabla^k B_m(f,g)| \lesssim |\nabla^k f| |g|$.

\subsection*{Calculus with symbols in $\mathcal{B}_s$ and $\widetilde{\mathcal{B}}_s$}

We begin with the action of derivatives on $\mathcal{B}_s$.

\begin{lemma}
\label{derivbs}
(i) If $\mu \in \mathcal{B}_s$, one can write
\begin{equation*}
\partial_\xi \mu(\eta,\xi) = \mu^1 + \frac{1}{|\xi|} \mu^2 + \frac{1}{|\xi-\eta|} \mu^3 \quad\mbox{with $(\mu^1, \mu^2,\mu^3) \in \mathcal{B}_{s-1} \times \mathcal{B}_s \times \mathcal{B}_s$}
\end{equation*}
(ii) If $\mu \in \mathcal{B}_s$ and $\nu \in \mathcal{B}_{s'}$, then $\mu \nu \in \mathcal{B}_{ss'}$.
\end{lemma}
\begin{proof}
(ii) follows from the definition of $\mathcal{B}_s$.  To prove (i) one notes that $\partial_\xi \mu$ contributes $\mu_1$ if $|\eta|$ is the smallest variable,  it contributes $\frac{\mu_2}{|\xi|}$  if $|\xi|$ is the smallest variable, and  it contributes $\frac{\mu_3}{|\xi-\eta|}$ if $|\xi - \eta|$ is the smallest variable.
\end{proof}

Finally, the following theorem gives the crucial boundedness properties of operators with symbols in $\mathcal{B}_s$ and $\widetilde{\mathcal{B}}_s$.

\begin{theorem}
\label{bilinearbound}
(i) If $m$ belongs to the class $\mathcal{B}_0$,
$$
\left\| B_m (f,g) \right\|_r \lesssim \|f\|_p \|g\|_q \quad\mbox{if $\frac{1}{p}+\frac{1}{q} = \frac{1}{r}$ and $1 < p,q,r < \infty$}
$$

(ii) If $m$ belongs to the class $\widetilde{\mathcal{B}}_s$, and if $k$ is an integer,
$$
\left\| \nabla^k B_m (f,g) \right\|_r \lesssim \|\Lambda^{k+s} f\|_p \|g\|_q \quad\mbox{if $\frac{1}{p}+\frac{1}{q} = \frac{1}{r}$ and $1 < p,q,r < \infty$}
$$
\end{theorem}

\begin{remark} Since Riesz transforms are not bounded on $L^\infty$, the statement of the above theorem becomes wrong if any of the indices $p,q,r$ is $\infty$. We need to go around this difficulty when deriving estimates for the water wave system, and this unfortunately makes estimates a bit longer.
\end{remark}

\begin{proof} To begin with $(i)$, let $m$ be a symbol in $\mathcal{B}_0$. Away from $\{ \xi_1 = 0 \} \cup \{ \xi_2 = 0 \}\cup \{ \xi_3 = 0 \}$, the Coifman-Meyer theorem~\cite{coifmanmeyer} can be applied to $m$ and gives the desired result. We will simply consider the case $|\eta|<< |\xi|$ (the other cases can be reduced to this one by duality), thus we consider in the following
$$
B'(f,g) \overset{def}{=} \sum_{j} B_{m} \left( P_{<j-100} f, P_j g \right) .
$$
By definition of the class $\mathcal{B}_0$, $m$ can be written in the set $|\eta| << |\xi|$ as $m(\xi,\eta) = \mathcal{A} \left( |\eta|^{1/2},\frac{\eta}{|\eta|},\xi \right)$ or, by homogeneity, $m(\xi,\eta) = \mathcal{A} \left( \frac{|\eta|^{1/2}}{|\xi|^{1/2}},\frac{\eta}{|\eta|},\frac{\xi}{|\xi|} \right)$. Expanding this expression in $\frac{|\eta|^{1/2}}{|\xi|^{1/2}}$ yields
$$
m (\xi,\eta) = \sum_{k=1}^M \frac{|\eta|^{k/2}}{|\xi|^{k/2}} m_k \left( \frac{\eta}{|\eta|}, \frac{\xi}{|\xi|} \right)\quad\mbox{+ remainder}.
$$
First notice that the $m_k$ are smooth and homogeneous of degree 0. Second, observe that for $M$ big enough, the singularity of the remainder at $\eta =0$ becomes so weak that the remainder satisfies estimates of Coifman-Meyer type; thus we take $M$ big enough and forget about the remainder.
Expanding $m_k$ in spherical harmonics (which we denote $(Z_\ell)_{\ell \in \mathbb{N}}$) leads to
$$
m (\xi,\eta) = \sum_{k=1}^M \sum_{\ell,\ell' \in \mathbb{N}} \alpha_{k,\ell,\ell'} \frac{|\eta|^{k/2}}{|\xi|^{k/2}} Z_\ell \left( \frac{\eta}{|\eta|} \right) Z_{\ell'} \left( \frac{\xi}{|\xi|} \right) .
$$
By the Mihlin-H\"ormander multiplier theorem, the operators $Z_\ell \left( \frac{D}{\Lambda} \right)$ are bounded on Lebesgue spaces, with bounds growing polynomially in $\ell$; on the other hand, the smoothness of $m$ entails fast decay in $(\ell,\ell')$ of the coefficients $\alpha_{k,\ell,\ell'}$. In the end, we can thus ignore the summation over $\ell$, $\ell'$ and $k$, and matters reduce to
$$
\sum_j \Lambda^{-k/2} \left[ P_{<j-100}\Lambda^{k/2} f P_j g \right] .
$$
It is then routine to get the desired estimate using slight extensions of the Littlewood-Paley square and maximal function estimates (Theorem~\ref{LP})
\begin{equation}
\begin{split}
& \left\| \sum_j \Lambda^{-k/2} \left[ P_{<j-100}\Lambda^{k/2} f P_j g \right] \right\|_r 
\lesssim \left \|\left(\sum\nolimits_{j} 2^{-jk} ( P_{<j-100}\Lambda^{k/2} f P_j g )^2 \right)^{\frac12} \right\|_r \\
& \;\;\;\;\;\;\;\;\;\;\;\;\;\;\;\;\;\;\;\lesssim \left\| \sup\nolimits_j \left| 2^{-jk/2} P_{<j-100}\Lambda^{k/2} f (x) \right| \left( \sum\nolimits_j (P_j g)^2 \right)^{1/2} \right\|_r \\
& \;\;\;\;\;\;\;\;\;\;\;\;\;\;\;\;\;\;\;\lesssim \left\| \sup\nolimits_j 2^{-jk/2} P_{<j-100}\Lambda^{k/2} f \right\|_p \left\| \left( \sum\nolimits_j (P_j g)^2 \right)^{1/2} \right\|_q \lesssim \|f\|_p \|g\|_q .
\end{split}
\end{equation}

\bigskip

In order to prove $(ii)$, consider a symbol $m$ in $\widetilde{\mathcal{B}}_s$, and simply observe that
$$
\nabla^k B_m(f,g) = B_{\xi^k m} (f,g) = B_{\frac{\xi^k}{|\eta|^{k+s}} m} (\Lambda^{k+s} f,g) ,
$$
Due to the support condition satisfied by $m$, the symbol $\frac{\xi^k}{|\eta|^{k+s}} m$ belongs to $\mathcal{B}^0$, therefore applying $(ii)$ gives
$$
\left\| \nabla^k B_m(f,g) \right\|_r = \left\| B_{\frac{\xi^k}{|\eta|^{k+s}} m} (\Lambda^{k+s} f,g) \right\|_r \lesssim \left\| \Lambda^{k+s} f \right\|_p \left\| g \right\|_q .
$$
\end{proof}

\section{Analysis of a class of trilinear pseudo-product operators}

\label{appendcub}

In this section, we turn to the trilinear operators which occur in the analysis of the water-wave problem: we define new classes of symbols adapted to them, and prove their boundedness.

\subsection*{Definition of the classes $\mathcal{T}_s$ and $\widetilde{\mathcal{T}_s}$}

As in the quadratic case, it will be convenient to put on an equal footing the four Fourier variables $\xi,\eta,\sigma,\xi-\eta-\sigma$. We adopt the following convention: 
\begin{enumerate}
\item First, we call $a_1 = -\xi$, $a_2 = \eta$, $a_3 = \sigma$, $a_4 = \xi-\eta-\sigma$.
\item Then, we partition the $(\xi,\eta,\sigma)$ space into regions where the $(|a_i|)$ are essentially ordered (in other words, for each of these regions there is a permutation $\sigma$ such that $|a_{\sigma(1)}| \lesssim |a_{\sigma(2)}| \lesssim |a_{\sigma(3)}| \lesssim |a_{\sigma(4)}|$.
\item Finally, we set $\xi_i = a_{\sigma(i)}$.
\end{enumerate}
In other words, $(\xi_i)$ is a convenient labeling of $-\xi$, $\eta$, $\sigma$, $\xi-\eta-\sigma$ since it satisfies $|\xi_1|\lesssim |\xi_2| \lesssim |\xi_3| \lesssim |\xi_4|$ and  $\xi_1 + \xi_2 + \xi_3 + \xi_4=0$ .

As in the quadratic case, we abuse notations by denoting indifferently $m$ for the symbol $m(\xi,\eta,\sigma)$ or its expression in any coordinate system, for instance $m(\xi_1,\xi_2,\xi_3)$.

\begin{defi} 
A symbol $m$ belongs to the class $\mathcal{T}_s$ if
\begin{itemize}
\item $m$  is homogeneous of degree $s$.
\item $m$  is smooth outside a conical neighborhood  $O_{ij}\supset \{ \xi_1 = 0 \} \cup\{ \xi_i + \xi_j = 0 \}$ for $(i,j)=(1,2), (1,3)$ or $(2,3)$.
\item   For   $|\xi_1| \sim |\xi_2|<<|\xi_3|\sim 1$, and $|\xi_1 + \xi_2| \sim |\xi_1|$ if $(i,j)=(1,2)$, $m = \mathcal{A}\left(\xi_1,\xi_2 \right) \mathcal{A}\left(\xi_1,\xi_2,\xi_3 \right)$ (flag singularity).
\item For $|\xi_1| << |\xi_2|, |\xi_3|, |\xi_4| \sim 1$,  $m = \mathcal{A}\left( |\xi_1|^{1/2} , \frac{\xi_1}{|\xi_1|} , \xi_2,\xi_3 \right)$.
\item For $|\xi_i+\xi_j|<< |\xi_1|, |\xi_2|, |\xi_3|, |\xi_4| \sim 1$,  $m = \mathcal{A}\left( |\xi_i+\xi_j|^{1/2} , \frac{\xi_i+\xi_j}{|\xi_i+\xi_j|} , \xi_1,  \xi_2,\xi_3 \right)$.
\item For $|\xi_1|<<|\xi_2|<<|\xi_3|,|\xi_4|\sim 1$,  $m = \mathcal{A}\left(\frac{|\xi_1|^{1/2}}{|\xi|_2^{1/2}},\frac{\xi_1}{|\xi_1|},|\xi_2|^{1/2},\frac{\xi_2}{|\xi_2|},\xi_3 \right)$.
\item For $|\xi_1 + \xi_2|<<|\xi_1|<<|\xi_3|,|\xi_4|\sim 1$,  $m = \mathcal{A} \left(\frac{|\xi_1+\xi_2|^{1/2}}{|\xi_1|^{1/2}},\frac{\xi_1+\xi_2}{|\xi_1+\xi_2|},|\xi_1|^{1/2},\frac{\xi_1}{|\xi_1|},\xi_3 \right)$.
\end{itemize} 
\end{defi}

\begin{remark}
One should think of symbols in $\mathcal{\mathcal{T}}_0$ as being of Coifman-Meyer type, except that they might exhibit flag singularities, and they are allowed to have singularities like Mihlin-H\"ormander (linear!) multipliers along the coordinate axes, and along one axis corresponding to the sum of two coordinates being zero. 

Though this definition is quite involved, it somehow corresponds to the simplest class containing all the symbols which occur in the analysis of the water waves problem. For instance, symbols of the type $\frac{m_k(\xi,\eta)}{\phi_{\pm,\pm}(\xi,\eta)}m_j(\xi-\eta,\sigma)$ (which appear in Section 6) contribute singularities of the type $\frac{\xi}{|\xi|}$ if one of the coordinates vanishes, flag singularities, and singularities along one axis $\xi_i + \xi_j$. This already accounts for nearly all points of the above definition.

How does one prove that these symbols are associated to bounded operators (Theorem~\ref{trilinearbound})? It is important in the proof that symbols in $\mathcal{T}_0$ can, by power or Fourier series expansions, be reduced to tensorial products of functions of one $\xi_i$ only; in other words it is possible to separate variables. Notice that symbols of the form (for instance) $\frac{\xi}{|\xi|} \frac{\eta}{|\eta|}$ could not be treated by such a method; but such a behavior is not possible in the class $\mathcal{T}_s$.
\end{remark}
 
We next define a new class of symbols, which somehow corresponds to paraproduct operators.

\begin{defi}
A symbol $m$ belongs to the class $\widetilde{\mathcal{T}}_s$ if
\begin{itemize}
\item It belongs to $\mathcal{T}_s$.
\item It satisfies the following support property: $\operatorname{Supp} m(\xi,\eta,\sigma) \subset \{ |\sigma| \gtrsim |\xi|,|\eta| \}$. 
\end{itemize}
\end{defi}

\subsection*{Calculus with symbols in $\mathcal{T}_s$ and $\widetilde{\mathcal{T}_s}$}

We begin with the action of derivatives on $\mathcal{T}_s$.

\begin{lemma}
\label{derivts}
(i) If $\mu \in \mathcal{T}_s$, one can write
\begin{equation*}
\begin{split}
& \partial_\xi \mu(\xi,\eta,\sigma) = \mu^1 + \frac{1}{|\xi|} \mu^2 + \frac{1}{|\xi-\eta|} \mu^3 + \frac{1}{|\xi-\sigma|} \mu^4 + \frac{1}{|\xi-\eta-\sigma|} \mu^5 \\
\end{split}
\end{equation*}
with $(\mu^1, \mu^2,\mu^3,\mu^4,\mu^5) \in \mathcal{T}_{s-1} \times \mathcal{T}_s \times \mathcal{T}_s \times \mathcal{T}_s \times \mathcal{T}_s$.

\medskip

(ii) If $\mu \in \mathcal{T}_s$ and $\nu \in \mathcal{T}_{s'}$, then $\mu \nu \in \mathcal{T}_{ss'}$.
\end{lemma}
\begin{proof} Follows from the definition of $\mathcal{T}_s$. Actually, due to the fact that 
$\mu$ is smooth outside a neighborhood of $O_{ij}$  for $(i,j)=(1,2), (1,3)$ or $(2,3)$, we also 
have that $\mu_3=0 $ or $\mu_4 = 0$. 
\end{proof}
Finally, the following theorem gives the crucial boundedness properties of operators with symbols in $\mathcal{T}_s$ and $\widetilde{\mathcal{T}}_s$.
 
\begin{theorem}
\label{trilinearbound}
(i) If $m$ belongs to the class $\mathcal{T}_0$,
\begin{equation*}
\begin{split}
& \left\| B_m (f_1,f_2,f_3) \right\|_r \lesssim \|f_1\|_{p_1} \|f_2\|_{p_2} \|f_3\|_{p_3} \quad\mbox{if $\frac{1}{p_1}+\frac{1}{p_2}+\frac{1}{p_3} = \frac{1}{r}$ and $1 < p_1,p_2,p_3,r < \infty$} \\
& \left\| B_m (f_1,f_2,f_3) \right\|_{\dot{B}^0_{1,\infty}} \lesssim \|f_1\|_{p_1} \|f_2\|_{p_2} \|f_3\|_{p_3} \quad\mbox{if $\frac{1}{p_1}+\frac{1}{p_2}+\frac{1}{p_3} = 1$.}
\end{split}
\end{equation*}
(ii) If $m$ belongs to the class $\widetilde{\mathcal{T}}_s$, and if $k$ is an integer,
\begin{equation*}
\begin{split}
& \left\| \nabla^k B_m (f_1,f_2,f_3) \right\|_r \lesssim \|\Lambda^{k+s} f_1\|_{p_1} \|f_2\|_{p_2} \|f_3\|_{p_3} \quad\mbox{if $\frac{1}{p_1}+\frac{1}{p_2}+\frac{1}{p_3} = \frac{1}{r}$ and $1 < p_1,p_2,p_3,r < \infty$} \\
& \left\| \nabla^k B_m (f_1,f_2,f_3) \right\|_{\dot{B}^0_{1,\infty}} \lesssim \|\Lambda^{k+s} f_1\|_{p_1} \|f_2\|_{p_2} \|f_3\|_{p_3} \quad\mbox{if $\frac{1}{p_1}+\frac{1}{p_2}+\frac{1}{p_3} = 1$ \mbox{and} $1 < p_1,p_2,p_3 < \infty$.}
\end{split}
\end{equation*}
\end{theorem}

\begin{remark} The estimates above involving Besov spaces come as substitute for estimates in $L^1$, which are wrong (for instance, the Riesz transform is not bounded on $L^1$). One can get in the same way substitutes for estimates in $L^\infty$.
\end{remark}

\begin{proof} 
The proof of (i) is very similar to the one of Theorem~\ref{theoflag} and Theorem~\ref{bilinearbound}: first, partition the frequency space in order to distinguish the different regions appearing in the definition of $\mathcal{T}_s$. Then, in each of these regions, expand the symbol in order to separate variables; finally, perform the desired estimates using the Littlewood-Paley theorem.

In the region $|\xi_1| \sim |\xi_2|<<|\xi_3|\sim 1$, and $|\xi_1 + \xi_2| \sim |\xi_1|$ if $(i,j)=(1,2)$, it suffices to apply Theorem~\ref{theoflag} on symbols with flag singularities.

Let us also sketch how one deals with the region $|\xi_1|<<|\xi_2|<<|\xi_3|,|\xi_4|$. Expanding the symbol $\mathcal{A}\left(\frac{|\xi_1|^{1/2}}{|\xi|_2^{1/2}},\frac{\xi_1}{|\xi_1|},|\xi_2|^{1/2},\frac{\xi_2}{|\xi_2|},\xi_3 \right)$ in power series in $\frac{|\xi_1|^{1/2}}{|\xi_2|^{1/2}}$ and $|\xi_2|^{1/2}$; and in spherical harmonics in $\frac{\xi_1}{|\xi_1|}$ and $\frac{\xi_2}{|\xi_2|}$ gives
$$
\sum_{k,k'} \sum_{\ell,\ell'} \left( \frac{|\xi_1|}{|\xi_2|} \right)^{k/2} |\xi|_2^{k'/2} Z_\ell \left( \frac{\xi_1}{|\xi_1|} \right) Z_{\ell'} \left( \frac{\xi_2}{|\xi_2|} \right) \Phi_{kk'\ell\ell'}(\xi_3) ,
$$
where the functions $\Phi_{kk'\ell\ell'}$ are homogeneous of degree $-\frac{k'}{2}$ 
 and decay fast with the indices $\ell,\ell'$, by smoothness of $\mathcal{A}$. 

The above sum over $k,k'$ can be taken to be finite, the remainder giving a Coifman-Meyer operator; furthermore, it is easily checked that the powers $\left( \frac{|\xi_1|}{|\xi_2|} \right)^{k/2}$ and $|\xi|_2^{k'/2}$ appearing above cancel with the homogeneity of $\Phi_{kk'\ell\ell'}$; thus it is possible to ignore the sum over $k$ and $k'$.

As for the sum over $\ell,\ell'$, we rely on the fast decay of the functions $\Phi_{kk'\ell\ell'}$, which offsets the polynomial growth of the bounds (in Lebesgue spaces) of the Fourier multipliers $Z_\ell \left( \frac{\xi_1}{|\xi_1|} \right)$  and $Z_{\ell'} \left( \frac{\xi_2}{|\xi_2|} \right)$; thus it is possible to ignore the sum over $\ell$ and $\ell'$.

Finally, up to rotation of the Fourier variables, and a duality argument, it is possible to assume that $(\xi_1,\xi_2,\xi_3) = (\sigma,\eta,\xi-\eta-\sigma)$.

The above considerations lead to the model operator
$$
B(f,g,h) = \sum_{0 \leq k,k' \leq N} \sum_j \sum_{j'<j-100} P_{<j'-100} f \, P_{j'} g \, P_j h.
$$
The boundedness of this operator is easily established: it is essentially a composition of paraproducts.

\bigskip

The extension to Besov spaces based on $L^1$ follows from boundedness from $L^1$ to $\dot{B}^1_{0,\infty}$ of Mihlin-H\"ormander type Fourier multipliers.

Finally, the point $(ii)$ is proved just like for Theorem~\ref{bilinearbound}.
\end{proof}

Finally, we need the following proposition, which combines fractional integration and flag singularity:

\begin{proposition}
\label{fifs}
(i) If $m(\xi,\eta,\sigma) \in \mathcal{T}_0$, $0<\alpha<2$, and $1<p_1,p_2,p_3,r<\infty$, then
$$
\left\| B_{\frac{1}{|\xi-\eta|^\alpha} m(\xi,\eta,\sigma)}(f_1,f_2,f_3) \right\|_r \lesssim \|f_1\|_{p_1} \|f_2\|_{p_2} \|f_3\|_{p_3}\quad\mbox{for $\frac{1}{p_1}+\frac{1}{p_2} + \frac{1}{p_3} -\frac{1}{2\alpha} = \frac{1}{r}$}.
$$

(ii) If $m(\xi,\eta,\sigma) \in \tilde{\mathcal{T}}_s$, $0<\alpha<2$, and $1<p_1,p_2,p_3,r<\infty$, then
$$
\left\| \nabla^k B_{\frac{1}{|\xi-\eta|^{\alpha}} m(\xi,\eta,\sigma)}(f_1,f_2,f_3) \right\|_r \lesssim \|\Lambda^{k+s} f_1\|_{p_1} \|f_2\|_{p_2} \|f_3\|_{p_3}\quad\mbox{for $\frac{1}{p_1}+\frac{1}{p_2} + \frac{1}{p_3} -\frac{1}{2\alpha} = \frac{1}{r}$}.
$$
\end{proposition}
\begin{proof}
The proof follows the pattern of the proofs of Theorem~\ref{theoflag}, Theorem~\ref{bilinearbound}, and Theorem~\ref{trilinearbound}. In the end, it thus somehow reduces to the case $m=1$, for which the estimate is clear since
$$
B_{|\xi-\eta|^{-\alpha}}(f_1,f_2,f_3) = f_2 \,\frac{1}{\Lambda^{\alpha}} (f_1 f_3) 
$$
and thus, by H\"older's inequality and Lemma~\ref{linearbound},
$$
\left\| B_{\frac{1}{|\xi-\eta|^{\alpha}}}(f_1,f_2,f_3) \right\|_r \lesssim \left\| f_2 \right\|_{p_2} \left\| \frac{1}{\Lambda^{\alpha}} (f_1 f_3) \right\|_{\frac{p_2-r}{p_2r}} \lesssim  \|f_1\|_{p_1} \|f_2\|_{p_2} \|f_3\|_{p_3}.
$$
\end{proof}

\section{ Proof of proposition \ref{prop:elliptic}}
\label{ap:e}

Let $a = (x, z) \in \Omega$ and $b = (y, h(y) ) \in S$.
Then $\psi_{\CH}$ can be represented by the double layer potential
\[
\psi_{\CH} (a) = \int\limits_S \mu (b)N 
\cdot \nabla G (a -b) dS (b)
=\frac12 \mu (b_0) +  \int\limits_S (\mu (b)- \mu (b_0) ) N 
\cdot \nabla G (a -b) dS (b) 
\]
where $G(a-b) = \frac 1{4\pi} |a-b|^{-1}$
is the Newtonian potential and $b_0$ is an arbitrary point on $S$.  
By defining $K (x,y) = \sqrt{1+|\nabla h(y)|^2}N(b) \cdot \nabla G (b_0-b)$
for $b_0= (x, h(x) ) \in S$ and $b = (y, h(y) ) \in S$
then 
\begin{equation}\label{e3}
\begin{cases}
K (x,y) = \frac{{- \nabla} h(y) \cdot (x-y) + h(x) -h (y) }{ 4\pi (|x-y|^2 + (h(x) -h(y))^2) ^{\frac 32}}\\
\int |K (x,y) |(1+ |x-y|^{\frac 12 }) \, dg
\le \|h\|_{W^{2, \infty} }
+ \|\nabla h\|_{L^p }
\lesssim \varepsilon_0
\end{cases}
\end{equation}
for $2<p<4$ and from standard singular integral calculations \cite{fo95} as $a\to b_0$
\begin{equation}
\label{eq:e2}
\frac 12\, \mu (x)
+ \int \mu (y) K (x,y) dy
= \psi (x)  .
\end{equation}
From \eqref{e3} it is clear that $K$ maps $L^\infty\to L^\infty$ and $\dot{C}^{\alpha} \to \dot{C}^{\alpha} $ for  $0< \alpha \le \frac12$.   Therefore we can solve for $\mu$ by a Neumann series expansion to obtain
\begin{equation}\label{eq:a2}
\begin{cases}
\|\mu \|_{L^\infty }
\lesssim \|\psi\|_{L^\infty }\\
\|\mu\|_{\dot{C}^{\alpha} }
\lesssim \|\psi\|_{\dot{C}^{\alpha} }
\end{cases}
\end{equation}
To obtain estimates on derivatives of $\mu$, we change variables from $y$ to  $z=x-y$ in  \eqref{eq:e2}  and write $J(x,z)=K(x,x+z)$.  Since 
 $\p_x J(x,z)$ satisfies inequality \eqref{e3}
with one additional derivative on $h$, i.e.,
\[
\int | \p_xJ (x,z) |
( 1 + |z|^{\frac 12 } ) dz
\lesssim \|h\|_{W^{3, \infty} }
+ \| \nabla h\|_{W^{1,p}} \lesssim \varepsilon_o,
\]
then differentiating  \eqref{eq:e2}  with respect to $x$ gives 
\[
\frac 12 \p\mu(x) + \int \p\mu(x-z)J(x,z)dz + \int \mu(x-z)\p_x J(x,z)dz =\p \psi(x)
\]
which implies
$ \| \p \mu \|_{L^\infty }
\lesssim \|\p \psi \|_{L^\infty} 
+ \|\mu \|_{\dot{C}^{\frac 12}} \lesssim \| \p \psi \|_{L^\infty}  + \|\psi \|_{\dot{C}^{\frac 12}} $ by \eqref{eq:a2}. 
Repeating the above argument twice we obtain
\[
\|\p \mu \|_{C^2 }
\lesssim \|\p \psi \|_{C^2}
+ \|\psi \|_{\dot{C}^{\frac 12}}.
\]
Next we estimate $\CN\psi$.  To do this we  fix a point $b_0\in S$ and  use normal coordinates in a neighborhood of $S$  to  restrict $a$ near the boundary to the line  $a = b_0 + \nu N (b_0 )$. Thus 
\[
N(b_0)\cdot \nabla \psi_{\CH} (a)
= \int\limits_S (\mu(b) -\mu(b_0))
 D^2 G(N(b), N(b_0))(a-b) dS (b) .
\]
For $|b-b_0|$ large and $\nu$ small  $|D^2G(a-b)|\lesssim |b-b_0|^{-3}$ and thus the above integral can be bounded by $\|\mu \|_{\dot{C}^{\frac 12}} $.   For $|b-b_0|$ small we write

\[
N(b_0) = \theta(b,b_0)N(b) + \gamma(b,b_0) \tau \qquad \text{where} \quad  \tau \in T_{b_0}S.
\]

The term involving $\tau$ is integrable due to the vanishing  of $\gamma(b_0,b_0)$.   By repeating the argument that led to inequality \eqref{eq:a2} we obtain 

\[
| \int\limits_S (\mu(b) -\mu(b_0))D^2 G(N(b),\gamma \tau)(a-b) dS (b) | \lesssim 
\|\mu\|_{\dot{C}^{1/2}} \lesssim \|\psi\|_{\dot{C}^{1/2}}.
\]

The kernel of the remaining term 
\[
I = \int\limits_S (\mu(b) -\mu(b_0))D^2 G(N(b), N(b))(\nu N(b_0) +b_0-b) dS (b)
\]
is hypersingular as $\nu \to 0$ and can be dealt with by using the identity

\[
0 = \De G 
= \De_{S} G
+ \kappa N 
\cdot \nabla G
+ D^2 G (N, N)
\]
 for $\nu < 0$.  This allows us to re-express $I$ as
\[
I   = \int\limits_{S} \CD \mu(b) \CD G(\nu N(b_0) +b_0-b)
-  (\mu(b) -\mu(b_0)) \kappa N(b) \cdot \nabla G(\nu N(b_0) +b_0-b) \, dS 
\]
which is a singular kernel and can be bounded as before 
\[
\| I \| \lesssim  \| \p \mu \|_{\dot{C}^{\frac 12}  }
+ \| \mu  \|_{\dot{C}^{\frac12} } \lesssim \| \p \psi \|_{\dot{C}^{\frac 12}}
+ \| \psi \|_{\dot{C}^{\frac12} }.
\]

By repeating the above argument after applying 
tangential derivatives to $N(b_0)\cdot\nabla \psi$ we obtain
\[
\|  \CN \psi \|_{W^{2,\infty}(S) }
\le \| \CD \psi \|_{C^3 }
+ \| \psi \|_{\dot{C}^{\frac12} }.
\]
This proves inequality \eqref{eq:1}.

\section{Estimates on the remainder term $R$}
\label{ap:f}
Recall that $R$ is the remainder term defined at the beginning of Section~\ref{decay}. More explicitly, 
recalling that $u = h+ i\Lambda^{1/2} \psi$, $R$ consists of the terms of order 4 and higher in $u$ of the nonlinearity
\begin{equation}
\label{NL}
G(h) \psi + i \Lambda^{1/2} \left[ -\frac{1}{2}|\nabla \psi|^2 + \frac{1}{2(1 + |\nabla h|^2)} \left( G(h)\psi + \nabla h \cdot \nabla \psi \right)^2 \right].
\end{equation}
The following proposition gives bounds on $R$:
\begin{proposition}
\label{propR} If $\|u\|_{W^{3,\infty}}$ is small enough,
\begin{equation}
\begin{split}
& (i) \left\|\nabla^k R \right\|_r \lesssim \left\| \nabla^{k+2} u \right\|_p \left\| \nabla^{k+2} u \right\|_q \left\| u \right\|_{W^{3,\infty}}^2 \;\;\;\;\;\;\;\mbox{if $\frac{1}{p} + \frac{1}{q} = \frac{1}{r}$ and $1 \leq p,q,r < \infty$} \\
& (ii) \left\| x R \right\|_2 \lesssim \|\langle x \rangle u \|_2 \|u\|_{W^{3,\infty}}^3. 
\end{split}
\end{equation}

\begin{remark} Both of the above estimates say that $R$ behaves roughly like a four-fold product of some derivatives of $u$. Both of them are far from being optimal, except for one point: in estimate $(ii)$, the weighted norm on the right-hand side does not carry any derivatives.
\end{remark}

\end{proposition}

We will only prove $(ii)$, the estimate $(i)$ being much easier. Furthermore, we observe that if the bound $(ii)$ can be proved with $R$ replaced by the fourth-order terms of $G(h) \psi$, it follows easily for $R$. Indeed, a look at the nonlinearity~(\ref{NL}) shows that $R$ essentially consists of products of $G(h)$ with derivatives of $u$, but this is easily treated.
Thus we will in the following prove $(ii)$ with $R$ replaced by terms of order four and higher of~$G(h)$.

\bigskip

Given a harmonic function $\psi_{\CH}$ in $\Omega$, it can be represented 
via a single layer potential $\rho$ as\footnote{Note that in our representation $\rho$  differs from the standard  $\rho$  by a factor of $\sqrt{1 + |\nabla h|^2}$.}
\begin{equation}
\label{eqrho}
\psi_{\CH} (x,z)  = \frac{1}{2\pi} \int \rho (y)
(|x-y|^2 + |z - h(y)|^2 )^{-1/2} \,dy \quad \mbox{if $(x,z) \in \Omega$}.
\end{equation}
Then
\begin{equation}
\label{eqGh}
G(h) \psi =  \rho - \frac{1}{2\pi}
\int \rho (y) \frac{\nabla h(x) \cdot(x-y) + h(y) - h(x)}{( |x-y|^2 + |h(x)  
-h(y)|^2 )^{3/2} }\, dy
\end{equation}
Notice that the kernel is very similar to the one given in (\ref{e3})

Expanding~(\ref{eqrho}) in $h$, we get after evaluating on the boundary $(x,h(x))$ and
applying $\Lambda$ 
\begin{align*}
\Lambda \psi (x) &= \rho(x) + \sum_{n\ge 1} \alpha_{n}
\Lambda \int \rho (y) \frac{|h(x) - h(y)|^{2n} }{|x-y|^{2n +1} } \, dy\\
&\overset{def}{=}\rho + {\sum K_{n} (\rho) }
\end{align*}
(recall that $\Lambda = \Lambda$). Using Neumann series, $\rho$ can be expressed as
\[
\rho  =  \Lambda \psi + \sum_{N \in \mathbb{N}, (m_1 \dots m_N) \in \mathbb{N}^N} \beta_{m_1 \dots m_N}
K_{m_1} \dots K_{m_N} \Lambda \psi
\]
where $|\beta_{m_1 \dots m_N }| \le C^{m_1 + \dots + m_N}$.
In order to obtain weighted estimates, we will need the following claim, whose proof we postpone for the moment.
\begin{claim} \label{pansy} If $n\geq 1$,\label{cl}
\begin{subequations} 
\begin{align}
\| x K_{n} \Lambda \psi \|_2 & \leq C^{n} \left\| \langle x\rangle u \right\|_2 \|u\|^{2n -1}_{W^{3,\infty}} \label{**} \\
\| \langle x \rangle K_{n} Z \|_2 & \leq C^{n} \| \langle x \rangle Z\|_2 \, \|h \|^{2n  -1 }_{W^{3,\infty}} \label{abcde}
\end{align}
\end{subequations}
\end{claim}
Assuming the claim, Proposition~\ref{propR} follows. Indeed, the contribution of $\frac{1}{2} \rho$ in~(\ref{eqGh}) to $R$ is given by
$$
\sum_{m_1 + \dots + m_N  \geq 2} \beta_{m_1  \dots m_N}
K_{m_1} \dots K_{m_N} \Lambda \psi,
$$
and this can be estimated using the claim. In order to estimate the contribution of the second summand of~(\ref{eqGh}), the same procedure can be applied: expand the kernel in powers of $h$, and use the bound on the expansion of $\rho$ in powers of $h$ which follows from the claim.

\subsection*{Proof of~(\ref{**}) in Claim~\ref{pansy}} 

\noindent \emph{Step 1: A paraproduct inequality.}
\begin{lemma}\label{lemma-III}
We have for  $A \geq 0$ and $B \geq 1$,
\begin{equation}\label{III}
\| x \nabla^A \Lambda f
\nabla^B g\|_2
\le \|\Lambda^{1/2} f\|_{W^{A+B+1, \infty}}\,
\|\langle x\rangle g\|_2
+ \|\langle x\rangle \Lambda^{1/2} f \|_2\,
\| g\|_{W^{A+B+1, \infty}}\,
\end{equation}
\begin{equation}\label{III2}
\| x \nabla^{A+1}  f
\nabla^B g\|_2
\le \|  f\|_{W^{A+B+2, \infty}}\,
\|\langle x\rangle g\|_2
+ \|\langle x\rangle  f \|_2\,
\| g\|_{W^{A+B+2, \infty}}\,
\end{equation}
\end{lemma}
\begin{proof} We focus on the first inequality, the second one being proved in an identical way.
Using the paraproduct decomposition, we write
\[
x \nabla^A \Lambda f \nabla^B g
= x \Big[ \sum_j \nabla^A \Lambda P_j f \nabla^B P_{\leq j} g
+ \sum_j \nabla^A \Lambda P_{<j} f \nabla^B P_j g \Big]
\]
Notice that 
\[  \begin{split}
& x \nabla^B P_{<j} g = \nabla^{B-1} \widetilde{P}_{<j} g + \nabla^B P_{<j} x g \\
& x \nabla^A \Lambda P_{<j} f  = \Lambda^{A-\frac{1}{2}} \widetilde{P}_{<j} \Lambda^{1/2} f 
+ \nabla^A \Lambda^{1/2} P_{<j} x \Lambda^{1/2} f \\
\end{split} \]
(where we denote $\widetilde P_{<j}$ for a symbol of the type $\lambda \left( \frac{D}{2^j} \right)$, with $\lambda$ smooth and supported in $B(0,C)$). Thus
$$
x\nabla^A \Lambda f \nabla^B g
= \sum_j \nabla^A \Lambda P_j f
( \nabla^{B-1} \widetilde{P}_{\leq j} g + \nabla^B P_{\leq j} x g ) + ( \Lambda^{A-\frac{1}{2}} \widetilde{P}_{<j} \Lambda^{1/2} f + \nabla^A  \Lambda^{1/2} P_{<j} x \Lambda^{1/2} f ) \nabla^B P_j g.
$$
This implies
\[  
\begin{split}
\| x\nabla^A \Lambda f \nabla^B g \|_2
& \leq \sum_j 2^{j (A+1)} \|P_j f \|_\infty
(2^{j (B-1)} \|g\|_2 + 2^{jB} \| xg \|_2 ) \\
& \;\;\;\;\;\;\;\;\;\;\;\; + ( 2^{j(A-1/2)} \left\| \Lambda^{1/2} f \right\|_2
+ 2^{j (A + \frac 12 )} \| x \Lambda^{1/2} f \|_2 )
2^{jB} \| P_j g \|_\infty\\
&\le \|\Lambda^{1/2} f \|_{W^{A +B+1 , \infty}}
\| \langle x \rangle g\|_2 + \|\langle x \rangle \Lambda^{1/2} f \|_2\, \|g\|_{W^{A+B+1, \infty}}
\end{split} \]
proving~(\ref{III}). The proof of~(\ref{III2}) is very similar.  \end{proof}
\bigskip
\noindent
\emph{Step 2: Splitting of the integral.}
To bound $K_{n} \Lambda \psi $  for $n \ge 1$,
choose a smooth cut-off function $\chi$ equal to $1$ on $B (0, 1)$  and $0$ outside of $B(0, 2)$, then
split the integral
\begin{subequations}
\begin{align}
\label{armadillo1}
K_{n} \Lambda \psi =& \Lambda \int \Lambda \psi (y)
\frac {|h(x) - h (y) |^{2n}}{|x-y|^{2n +1 }} \chi (x-y)\, dy\\
& \label{armadillo2}
+ \Lambda \int \Lambda \psi (y)
\frac {|h(x) - h (y) |^{2n}}{|x-y|^{2n +1 }} ( 1 - \chi (x-y ) )\ dy\,\,.
\end{align}
\end{subequations}
Observe that since
\( \| x \Lambda f \|_2
\le \| x \nabla f \|_2 + \| f \|_2
\)
we can replace $\Lambda\int $ by $\nabla\int$ in our estimate.

\bigskip
\noindent
\emph{Step 3: Estimate of~(\ref{armadillo2}): The case $ |x-y| \gtrsim 1$.} In this region, the cancellation contained in $h(x)- h(y)$  is not needed
thus it suffices to replace $(h(x) - h(y) )^{2n}$ by the general term given by the expansion of this power, namely $h(x)^\ell h(y)^{2n -\ell}$. We get
\[
\int \Lambda \psi (y) \nabla_x
\bigg[
( 1 -\chi (x-y) )
\frac{h(x)^\ell h(y)^{2n -\ell } }{|x-y|^{2n +1} }
\bigg]\, dy.
\]
For $\ell =0$ the above integral is  given by
\[
\int|h (y)| ^{2n} \Lambda \psi (y) \nabla_x
\frac{1 -\chi (x-y) }{|x-y|^{2n +1} }\, dy =  \Gamma * ( \Lambda  
\psi h^{2n}),
\]
where $\Gamma$ satisfies $\|\langle x\rangle\Gamma\|_{L^1}\lesssim 1$. Therefore
\[
x ( \Gamma * ( \Lambda \psi h^{2n}))
= \underbrace{ \underbrace{ (x \Gamma)}_{L^{1}} * \underbrace{(\Lambda  
\psi h^{2n})}_{L^{2}} }_{L^2}
+\underbrace{ \underbrace{\Gamma}_{L^1} * \underbrace{ (x \Lambda  
\psi h^{2n})}_{L^2} } _ {L^2}
\]
and the conclusion follows since
\[
\| x \Lambda \psi h^{2n} \|_2
\lesssim \| x h \|_2 \,
\|\Lambda \psi \|_\infty \,
\| h \|_\infty^{2n -1}.
\]
For $\ell \ge 1$ the most problematic term is when $\nabla_x$ hits  
$h(x)$, since otherwise  we have  added decay in $x-y$. It reads
$$
\int \Lambda \psi (y) ( 1 -\chi (x-y) )
\frac{\nabla h(x)  h(x)^{\ell -1 } h(y)^{2n -\ell } }{|x-y|^{2n +1} }\, dy =
  \nabla h(x)  h(x)^{\ell -1 }
\Gamma_1 * ( \Lambda \psi h^{2n -\ell }  )
$$
where $\Gamma_1$ decays at least like $|x|^{-3}$.
For $\ell =1$ we have
\[
\|x (\nabla h)( \Gamma_1 * ( \Lambda \psi h^{2n-1})\| 
_{L^2} \lesssim \|\nabla h\|_{L^\infty}\||x ( \Gamma_1 * ( \Lambda  
\psi h^{2n-1})\|_{L^2}
\]
which can be estimated by
\[
x ( \Gamma_1 * ( \Lambda \psi h^{2n-1}))
= \underbrace{ \underbrace{ (x \Gamma_1)}_{L^{4/3}} *  
\underbrace{(\Lambda \psi h^{2n-1})}_{L^{4/3}} }_{L^2}
+\underbrace{ \underbrace{\Gamma_1}_{L^1} * \underbrace{ (x \Lambda  
\psi h^{2n-1})}_{L^2} } _ {L^2}
\]
and since $L^2(\langle x\rangle^2 dx)\hookrightarrow L^{\frac 43}(dx)$, this yields the  
desired estimate as in the case $\ell=0$.
For $\ell \ge 2$ the bound follows in an even simpler way, thus we skip this case.

\bigskip
\noindent
\emph{Step 4: Estimate for~(\ref{armadillo1}): The case $|x-y| \lesssim 1$.}  
We use Taylor's formula
\[
\frac{h(y) - h(x) }{|x-y|}
= \nabla h
(x)
\frac{y-x}{|y-x|}
+\int^1_0 \nabla^2 h (x +t (y-x) )
\frac{(y-x)^2}{|y-x|}
(1-t) dt
\]
to deal with the singular integral. More precisely, we substitute the above expression for each factor $h(x)-h(y)$ in~(\ref{armadillo1}); subsequently expanding the product gives $2^{2n}$ terms. We start with the one involving
$\left( \nabla h \cdot \frac{y- x}{|y-x|} \right)^{2n}$, writing it
\begin{subequations}
\begin{align}
&\nabla_x \int \Lambda \psi (y) \frac {\chi (x-y)}{|x-y|}
\bigg(\nabla h(x) \cdot
\frac{y- x}{|y-x|} \bigg)^{2n}\, dy\\
& \label{grackle2} \quad \quad \quad\quad \quad =
\nabla_x \int ( \Lambda \psi (y) - \Lambda \psi (x) )
\frac {\chi (x-y)}{|x-y|} \bigg(\nabla h(x) \cdot
\frac{y- x}{|y-x|} \bigg)^{2n}\, dy \\
& \label{grackle3} \quad \quad \quad\quad \quad \quad \quad \quad \quad + \nabla_x \bigg(
\Lambda \psi (x)
\int \frac {\chi (x-y)}{|x-y|}
\bigg(\nabla h(x) \cdot
\frac{y- x}{|y-x|} \bigg)^{2n}\, dy \bigg) 
\end{align}
\end{subequations}
The term~(\ref{grackle3}) can be written as
\[ \begin{split}
&\nabla \bigg( \Lambda \psi (x) \sum_{N_j =1,2}
\p_{N_1} h \cdots \p_{N_{2n}} h
\int \frac{\chi (x-y)}{|x-y|} \frac{(y-x)^{N_1}}{|y-x|}
\dots \frac {(y-x)^{N_{2n}}}{|y-x|}\, dy \bigg)\\
&\quad = \nabla \big( \Lambda \psi (x)
\sum \gamma_{N_1 \cdots N_{2n}}  \p_{N_1 } h \cdots \p _{N _n} h \big)
\end{split}\]
which is easy to bound by~(\ref{III}).
To bound part~(\ref{grackle2}), we note that if $\nabla_x$ hits $\Lambda \psi (x)$ one gets
\[
\nabla \Lambda \psi (x) \int \frac {\chi (x-y)}{|x-y|}
\bigg( \nabla h(x) \cdot \frac{y- x}{|y-x|} \bigg)^{2n}\, dy\]
which is similar to part~(\ref{grackle3}). If  $\nabla_x$ hits $\nabla h(x)$ one  
gets
\[
  \int (\Lambda \psi (y) - \Lambda \psi (x)) \frac{\chi(x-y)}{|x-y|}
\bigg(\nabla^2 h(x)  \frac{y- x}{|y-x|} \bigg)
\bigg( \nabla h(x) \cdot  \frac{y- x}{|y-x|}  \bigg)^{2n-1} \, dy
\]
which can be written as a sum of terms of the type
\[
C \nabla^2 h(x) \nabla h(x)^{2n -1} \Lambda \psi \qquad\mbox{or}\qquad \nabla^2 h(x) \nabla h(x)^{2n -1} K * \Lambda \psi \qquad
  K~\text{ an integrable kernel}
\]
and thus can be treated by Lemma \ref{lemma-III}.

If $\nabla_x$ hits the function of $(x-y)$ we turn $\nabla_x$ into $- 
\nabla _y$ and integrate by parts to get
\[
\int \nabla \Lambda \psi (y)  \frac{ \chi(y-x)}{|y-x|}
\bigg( \nabla h(x) \cdot
\frac {y- x}{|y-x|} \bigg)^{2n} \,dy
\]
which can be written as a sum of terms of the type
\[
\nabla h(x)^{2n}   K * \Lambda \psi\qquad
  K~\text{an integrable kernel}
\]
and can be bounded by lemma~\ref{lemma-III}.

\medskip

All the remaining terms in the expansion of~(\ref{armadillo1}) after substituting the Taylor expansion involve at least one power of
$\int^1_0 \nabla^2 h \dots$, i.e.
\begin{gather*}
\nabla_x \int \Lambda \psi (y) \chi (x-y) |x-y|^{\ell -1}
\bigg( \int^1_0 \nabla^2 h
(x + t(y-x) ) (1-t) dt  \bigg)^\ell
\nabla h(x)^{2n -\ell } \, dy
\end{gather*}
for  $ \ell \ge 1$ (here we abused notations by writing $|x-y|^{\ell -1}$ instead of the correct multilinear expression in $x-y$ involving scalar products).
Regardless of what term is hit by $\nabla_x$ we estimate the above in the  
following manner:  1) Split $\psi$ into low and high frequencies
$$
\psi = \psi_{low} + \psi_{high} \overset{def}{=} P_{>0} \psi + P_{\geq 0} \psi;
$$
2) if $\psi$ is low frequency, estimate directly after putting the  
weight on $\psi$; 3) if $\psi$ is high frequency, remove one  
derivative from $\psi$ by integration by parts  (thus getting $R  
\psi$, $R$~Riesz transform),
then estimate after putting the  weight on $\psi$.  Thus if 
$\psi$ is low frequency, we take the derivatives of $h$ in $L^\infty$ and get the bound
\[
\| \langle x\rangle \Lambda \psi_{\text{low}} \|_2\,
\| h \|_{W^{3,\infty}}^{2n}
\]
and if $\psi$ is high frequency, integrate by parts in $y$ using $ 
\Lambda \psi = \nabla_y R \psi$
and get the bound
\[
\| \langle x \rangle R \psi_{\text{high}} \|_2\,
\| h \|^{2n} _{W^{3, \infty}}
\]
The estimate follows since
\[
\| \langle x\rangle  \Lambda \psi_{\text{low}} \|_2\,
+\| \langle x\rangle R \psi_{\text{high}} \|_2
\lesssim \|\langle x\rangle \Lambda^{1/2} \psi \|_2
\]
This completes the proof of \eqref{**}.

\subsection*{Proof of~(\ref{abcde}) in Claim~\ref{pansy}} 

The proof of the inequality~(\ref{abcde}) is very similar to that of inequality~(\ref{**}), which was the object of the previous section.
One additional ingredient is however needed, which reduces to the boundedness of the operator
$$
T : Z \mapsto \nabla \int Z(y) \frac{|h(x)-h(y)|^{2n}}{|x-y|^{2n+1}} \chi(x-y) \,dy \;\;\;\;\;\mbox{on $L^2$}.
$$
In order to prove it, observe that $T$ is a singular integral operator if $h$ belongs to $W^{2,\infty}$. By the David and Journ\'e $T1$ theorem, its boundedness over $L^2$ will follow from the belonging of $T(1)$ and $T^*(1)$ to $BMO$. Since $T$ is antisymmetric, it suffices to check that $T1$ belongs to $BMO$. But $T1$ is given by
$$
\nabla \int \frac{|h(x)-h(y)|^{2n}}{|x-y|^{2n+1}} \chi(x-y) \,dy,
$$ 
which, as one sees easily, is bounded in $L^\infty$ provided $h \in W^{3,\infty}$.

\begin{remark}  Equations \eqref{aaa}  can also be easily derive from \eqref{eqrho} and \eqref{eqGh}.  Using the fact that $\Lambda = \sum R_i\p_i$, where $R_i$'s are  the Riez potentials, one can show that
\[
\Lambda \psi(x) = \mathrm{p.v.} \frac 1\pi \int_{\RR^2} \frac{\psi(x)-\psi(y)}{|x-y|^3}dy,
\] 
 thus  from \eqref{eqrho} $
\rho = \Lambda \psi - \frac \Lambda{2}[h^2 \Lambda^2 \psi- 2h\Lambda(h \Lambda \psi) + \Lambda(h^2 \Lambda \psi)]+\dots$, and from  \eqref{eqGh}
\[
G(h)\psi=  \Lambda \psi - \nabla \cdot (h \nabla \psi) - \Lambda (h\Lambda \psi) - \frac{1}{2} [\Lambda(h^2 \Lambda^2 \psi) + \Lambda^2(h^2 \Lambda \psi) - 2\Lambda(h\Lambda(h \Lambda \psi))] + \dots ,
 \]
which is the first part of \eqref{aaa}.  The second part follows immediately from the first
\end{remark}

\end{document}